\documentclass[11pt]{article}
\usepackage{geometry}
\usepackage{mathtools, amssymb, amsthm, bm, amsmath}
\usepackage{mathrsfs}
\usepackage{bm}
\usepackage{xcolor}
\usepackage{tikz-cd}
\usepackage{tikz}
\usepackage{thmtools}
\usepackage{thm-restate}
\usepackage[shortlabels]{enumitem}
\usepackage{hyperref}
\usepackage{cleveref}
\usepackage{autonum}

\DeclareMathOperator{\N}{\mathbb{N}}
\DeclareMathOperator{\Q}{\mathbb{Q}}

\DeclareMathOperator{\C}{\mathbb{C}}
\DeclareMathOperator{\R}{\mathbb{R}}
\DeclareMathOperator{\Z}{\mathbb{Z}}

\DeclareMathOperator{\E}{\mathcal{E}}
\DeclareMathOperator{\cE}{\mathcal{E}}
\DeclareMathOperator{\cW}{\mathcal{W}}

\DeclareMathOperator{\GL}{GL}

\DeclareMathOperator{\eps}{\varepsilon}

\DeclareMathOperator{\ind}{ind}

\DeclareMathOperator{\Sp}{Sp}

\DeclareMathOperator{\SO}{SO}

\DeclareMathOperator{\mult}{mult}
\DeclareMathOperator{\sgn}{sgn}
\DeclareMathOperator{\supp}{supp}
\DeclareMathOperator{\IC}{IC}

\DeclareMathOperator{\img}{Im}

\DeclareMathOperator{\jordbp}{\Delta}

\DeclareMathOperator{\PPsymp}{{\bm{\mathcal{P}^{\mathrm{symp}}}}}
\DeclareMathOperator{\Port}{\mathcal{P}^{\mathrm{ort}}}
\DeclareMathOperator{\PPort}{ \bm{\mathcal{P}^{ \mathrm{ort} } }  }
\DeclareMathOperator{\PPPort}{ \bm{\mathcal{P}}^{\mathrm{ort},2}}

\DeclareMathOperator{\epsmax}{\eps^{\mathrm{max}}}
\DeclareMathOperator{\lambdamax}{\lambda^{\mathrm{max}}}
\DeclareMathOperator{\lambdamin}{\lambda^{\mathrm{min}}}
\DeclareMathOperator{\epsmin}{\eps^{\mathrm{min}}}
\DeclareMathOperator{\slambdamin}{\prescript{s}{}\lambda^{\mathrm{min}}}
\DeclareMathOperator{\sepsmin}{\prescript{s}{}\eps^{\mathrm{min}}}
\DeclareMathOperator{\slambdamax}{\prescript{s}{}\lambda^{\mathrm{max}}}
\DeclareMathOperator{\seps}{\prescript{s}{}\eps}
\DeclareMathOperator{\slambda}{\prescript{s}{}\lambda}
\DeclareMathOperator{\sepsmax}{\prescript{s}{}\eps^{\mathrm{max}}}
\DeclareMathOperator{\Cmax}{\mathit{C}^{\mathrm{max}}}
\DeclareMathOperator{\Emax}{\mathcal E^{\mathrm{max}}}

\DeclareMathOperator{\lambdaodd}{\lambda^{\mathrm{odd}}}
\DeclareMathOperator{\lambdaeven}{\lambda^{\mathrm{even}}}

\DeclareMathOperator{\uodd}{\mathit{u}^{\mathrm{odd}}}
\DeclareMathOperator{\alphamax}{\alpha^{\mathrm{max}}}
\DeclareMathOperator{\epsodd}{\eps^{\mathrm{odd}}}
\DeclareMathOperator{\betamax}{\beta^{\mathrm{max}}}

\DeclareMathOperator{\triv}{\mathrm{triv}}
\DeclareMathOperator{\GSpr}{GSpr}

\DeclareMathOperator{\fS}{\mathfrak S}
\DeclareMathOperator{\cP}{\mathcal P}
\DeclareMathOperator{\cR}{\mathcal R}
\DeclareMathOperator{\cU}{\mathcal U}
\DeclareMathOperator{\cN}{\mathcal N}
\DeclareMathOperator{\NNG}{\mathcal{N}_G}

\newcommand{\transp}[1]{{}^t{#1}}
\newcommand{\transps}[1]{{}^s	{#1}}

\theoremstyle{definition}
\newtheorem{definition}{Definition}[section]
\newtheorem{example}[definition]{Example}
\newtheorem{remark}[definition]{Remark}

\newtheorem*{claim*}{Claim}

\theoremstyle{theorem}
\newtheorem{theorem}[definition]{Theorem}
\newtheorem*{theorem*}{Theorem}
\newtheorem{lemma}[definition]{Lemma}
\newtheorem{proposition}[definition]{Proposition}
\newtheorem{corollary}[definition]{Corollary}

\numberwithin{equation}{section}

\newtheorem{oldtheoremnr}{Theorem}
\newenvironment{theoremnr}[1]
{\renewcommand\theoldtheoremnr{#1}\oldtheoremnr}
{\endoldtheoremnr}

\title{Maximality properties of generalised Springer representations of $\SO(N,\C)$}
\author{Ruben La \\ \href{mailto:ruben.la@maths.ox.ac.uk}{\texttt{}{ruben.la@maths.ox.ac.uk}}}

\begin{document}

\maketitle
\begin{abstract}
The generalised Springer correspondence for $G = \mathrm{SO}(N,\mathbb{C})$ attaches to a pair $(C,\mathcal{E})$, where $C$ is a unipotent class of $G$ and $\mathcal{E}$ is an irreducible $G$-equivariant local system on $C$, an irreducible representation $\rho(C,\mathcal{E})$ of a relative Weyl group of $G$. We call $C$ the Springer support of $\rho(C,\mathcal{E})$. For each such $(C,\mathcal{E})$, $\rho(C,\mathcal{E})$ appears with multiplicity 1 in the top cohomology of some variety. Let $\bar\rho(C,\mathcal{E})$ be the representation obtained by summing over all cohomology groups of this variety. It is well-known that $\rho(C,\mathcal{E})$ appears in $\bar\rho(C,\mathcal{E})$ with multiplicity $1$ and that it is a `minimal subrepresentation' in the sense that its Springer support $C$ is strictly minimal in the closure ordering among the Springer supports of the irreducbile subrepresentations of $\bar\rho(C,\mathcal{E})$. Suppose $C$ is parametrised by an orthogonal partition consisting of only odd parts. We prove that there exists a unique `maximal subrepresentation' $\rho(C^{\mathrm{max}},\mathcal{E}^{\mathrm{max}})$ of multiplicity $1$ of $\bar\rho(C,\mathcal{E})$. Let $\mathrm{sgn}$ be the sign representation of the relevant relative Weyl group. We also show that $\mathrm{sgn} \otimes \rho(C^{\mathrm{max}},\mathcal{E}^{\mathrm{max}})$ is the minimal subrepresentation of $\mathrm{sgn} \otimes \bar\rho(C,\mathcal{E})$.These results are direct analogues of similar maximality and minimality results for $\mathrm{Sp}(2n,\mathbb{C})$ by Waldspurger.
\end{abstract}

\tableofcontents

\section*{Introduction}
\addcontentsline{toc}{section}{Introduction}

Let $G$ be a complex connected reductive algebraic group. Let $\cU$ be the set of unipotent classes of $G$ and let $\NNG$ be the set of pairs $(C,\E)$, where $C \in \cU$ and $\cE$ is an irreducible $G$-equivariant local system on $C$. Let $L$ be a Levi subgroup of $G$ containing a cuspidal pair
(see for instance \cite[2.4 Definition]{lusztig2}).
Lusztig showed that $W_L := N_G(L)/L$ is a Coxeter group. We call $W_L$ a \emph{relative Weyl group} of $G$.
For any group $\Gamma$, denote by $\Gamma^\wedge$ the set of equivalence classes of irreducible representations of $\Gamma$.
The generalised Springer correspondence \cite{lusztig2} is a certain bijection
\begin{align}
\GSpr \colon \NNG \to \bigsqcup_{L \text{ Levi with cuspidal pair}} W_L^\wedge.
\end{align}
We will state some results from \cite{lusztig1986character}.
Let $(C,\E),(C',\mathcal E') \in \mathcal N_G$ and let $\bar C'$ be the closure of $C'$.
Let $\IC(\bar C',\E')$ be the intersection cohomology complex on $\bar C'$ associated to $(C',\E')$.
The restriction of $\bigoplus_{m\in\Z}H^{2m}(\IC(\bar C',\mathcal E'))$ to $C$ is a direct sum of irreducible $G$-equivariant local systems on $C$.
Let $\mult(C,\mathcal E;C',\mathcal E')$ be the multiplicity of $\mathcal E$ in $\bigoplus_{m\in\Z}H^{2m}(\IC(\bar C',\mathcal E'))|_C$.
The representation
\begin{align}\label{eq:multgsc}
\bar\rho = \bar\rho(C,\cE) = \bigoplus_{(C',\cE') \in \NNG} \mult(C,\cE;C',\cE')\GSpr(C',\cE').
\end{align}
is equal to the sum of the cohomology groups of some perverse sheaf on a `generalised flag \nobreak{manifold}' (see for instance \cite[\S11]{lusztig1995cuspidal}). 
It is well-known that $\mult(C,\cE;C,\cE) = 1$ and that if $\mult(C,\cE;C',\cE') \neq 0$, then $C$ is strictly smaller than $C'$ in the closure order on $\mathcal U$ or $(C,\mathcal E) = (C',\mathcal E')$. In this sense, $(C,\cE)$ is minimal, or in other words, $\GSpr(C,\cE)$ is a `minimal subrepresentation' of $\bar\rho$.
A natural question to ask is whether there exists an analogous `maximal subrepresentation' $\GSpr(\Cmax,\Emax)$ of $\bar\rho$.

Let $n,N \in \N$. For $G = \GL(n)$, the only irreducible $G$-equivariant local system on $C$ is the trivial one and $\Cmax$ is the unipotent class parametrised by the partition $(n)$.
For $\Sp(2n)$, it was proved in\cite{waldspurger} that $(\Cmax,\Emax)$ uniquely exists when $C$ is parametrised by a symplectic partition of $2n$ -- that is, a partition whose odd parts all appear with even multiplicity -- that consists of only even parts. The goal of this paper is to prove the existence and uniqueness of $(\Cmax,\Emax)$ for $\SO(N)$ when $C$ is parametrised by an orthogonal partition of $N$ -- that is, a partition whose even parts all appear with even multiplicity -- that consists of only odd parts. 

Let $G = \SO(N)$.
We will first describe $\NNG$ and the relative Weyl groups of $G$.
Denote by $\mathcal P(n)$ the set of partitions of $n$ and denote by $\Port(N)$ the set of partitions $\lambda \in \mathcal P(N)$ that are orthogonal.
A unipotent element of $G$ can be characterised by its Jordan cells, and this gives a surjective map $\mathcal U \to \Port(N)$, such that the preimage of each $\lambda \in \Port(N)$ is a singleton, except when $\lambda$ is degenerate (i.e. $\lambda$ consists of only even parts), then the preimage of $\lambda$ consists of two unipotent classes, which we also call degenerate.
For each $\lambda \in \Port(N)$, let $\Delta(\lambda)$ be the set of distinct odd parts of $\lambda$.
Then $F_2[\Delta(\lambda)] = \{\pm1\}^{\Delta(\lambda)}$ is an $F_2$-vector space where $F_2$ is the field with two elements.
Let $F_2[\Delta(\lambda)]'$  be the quotient of $F_2[\Delta(\lambda)]$ by the line spanned by the sum of the unit vectors of $F_2[\Delta(\lambda)]$. 
If $C \in \cU$ is parametrised by $\lambda \in \Port(N)$, then $F_2[\Delta(\lambda)]'$  is in $1$--$1$ correspondence with the set of irreducible $G$-equivariant local systems on $C$.
Let $\PPort(N)$ be the set of pairs $(\lambda,[\eps])$, where $\lambda \in \Port(N)$, $\eps \in F_2[\Delta(\lambda)]$, and $[\eps] \in F_2[\Delta(\lambda)]'$ is the image of $\eps$ under the quotient map.
We obtain a surjective map $\mathcal N_G \to \PPort(N)$ such that if $\lambda$ is non-degenerate, then
the preimage of $(\lambda,[\eps])$ contains one element, which we denote by $(C_\lambda,\E_{[\eps]}) = (C_\lambda^+,\E_{[\eps]}^+) = (C_\lambda^-,\E_{[\eps]}^-)$, and if $\lambda$ is degenerate, then the the preimage of $(\lambda,[\eps])$ contains two elements, which we denote by $(C_\lambda^+,\E_{[\eps]}^+) \neq (C_\lambda^-,\E_{[\eps]}^-)$. 

For $n \in \N$, let $W(B_n) = W(C_n)$ be the Weyl group of type $B_n$ or $C_n$ and let $W(D_n)$ be the Weyl group of type $D_n$. 
Let $\cP_2(n)$ be the set of bipartitions of $n$.
We can parametrise $W(B_n)^\wedge = \{\rho_{(\alpha,\beta)} \colon (\alpha,\beta) \in \mathcal P_2(n)\}$. 
Define $\theta\colon \mathcal P_2(n) \to \mathcal P_2(n)$ by $\theta(\alpha,\beta) = (\beta,\alpha)$ for $(\alpha,\beta) \in \mathcal P_2(n)$. Let $\mathcal P_2(n) / \theta$ be the set of $\theta$-orbits and let $c_{(\alpha,\beta)}$ be the size of the $\theta$-orbit of $(\alpha,\beta)$. Then $W(D_n)^\wedge$ is parametrised as
$W(D_n)^\vee = \{\rho_{(\alpha,\beta),i} \colon (\alpha,\beta) \in \mathcal P_2 / \theta, i \in \{1,2/c_{(\alpha,\beta)	}\}\}$.

Let $W$ be the Weyl group of $G$. If $N$ is odd, then $W = W(B_{(N-1)/2})$, and if $N$ is even, then $W = W(D_{N/2})$.
Let 
\begin{align}
\mathcal W = W^\wedge \sqcup \bigsqcup_{k \in \Z_{\geq2},\,k \equiv N \mod 2} W(B_{(N-k^2)/2})^\wedge.
\end{align} 
The Weyl groups $W$ and $W(B_{(N-k^2)/2})$ are precisely the relative Weyl groups of $G$.
Recall that the generalised Springer correspondence is a bijective map $\GSpr \colon \NNG \to \cW$.
If $N$ is even, any two degenerate unipotent classes parametrised by the same element of $\Port(N)$ are mapped to $\rho_{(\alpha,\alpha),1}$ and $\rho_{(\alpha,\alpha),2}$ by $\GSpr$, for some $(\alpha,\alpha) \in \cP_2(N/2)/\theta$. 
As such, we obtain a bijection 
\begin{align}
\Phi_N \colon \PPort(N) \to (\mathcal P_2(N/2) / \theta) \sqcup \bigsqcup_{k \in \N \text{ even }} \mathcal P_2((N-k^2)/2).
\end{align}

\begin{theoremnr}{\ref{thm:max}}
Suppose $(\lambda,[\eps]) \in \PPort(N)$ such that $\lambda$ only has odd parts. Then there exists a unique $(\lambda^{\text{max}},[\epsmax]) \in \PPort(N)$ such that $\lambdamax$ is non-degenerate and 
\begin{enumerate}[(1)]
\item\label{thm:max1} $\mult(C_\lambda,\mathcal E_{[\eps]}; C_{\lambdamax},\mathcal E_{[\epsmax]}) = 1$,
\item\label{thm:max2} For all $(\lambda',[\eps']) \in \PPort(N)$ with $\mult(C_\lambda,\E_{[\eps]};C_{\lambda'}^+,\mathcal E_{[\eps']}^+)= \mult(C_\lambda,\E_{[\eps]};C_{\lambda'}^-,\mathcal E_{[\eps']}^-) \neq 0$, we have $\lambda' < \lambda^{\text{max}}$ or $(\lambda',[\eps']) = (\lambda^{\text{max}},[\epsmax])$.
\end{enumerate}
\end{theoremnr}

We also prove a related minimality theorem. Given $(\lambda,[\eps]) \in \PPort(N)$, let $(\slambda,[\seps])$ be the element of $\PPort(N)$ for which $\GSpr(C_{\slambda},\cE_{[\seps]}) = \GSpr(C_\lambda,\cE_{[\eps]}) \otimes \sgn$, where $\sgn$ is the sign representation of the relative Weyl group that $\GSpr(C_\lambda,\cE_{[\eps]})$ is a representation of.

\begin{theoremnr}{\ref{thm:min}}
Let $(\lambda,[\eps]) \in \PPort(N)$ with $\lambda$ only consisting of odd parts. Then there exists a unique $(\lambdamin,\epsmin) \in \PPort(N)$ such that $\slambdamin$ is non-degenerate and
\begin{enumerate}
\item $\mult(C_\lambda,\E_{[\eps]};C_{\slambdamin},\E_{[\sepsmin]}) = 1$,
\item For all $(\lambda',[\eps']) \in \PPort(N)$ with $\mult(C_\lambda,\mathcal E_{[\eps]};C_{\slambda'}^+,\E_{[\seps']}^+) = \mult(C_\lambda,\mathcal E_{[\eps]};C_{\slambda'}^-,\E_{[\seps']}^-) \neq 0$, we have $\lambdamin < \lambda'$ or $(\lambda',[\eps']) = (\lambdamin,[\epsmin])$.
\end{enumerate}
Furthermore, we have $(\lambda^{\text{max}},[\epsmax]) = (\slambdamin,[\sepsmin])$.
\end{theoremnr}

The proofs of \Cref{thm:max} and \Cref{thm:min} are based on Waldspurger's proofs of the maximality and minimality results for $\Sp(2n)$ \cite[Thm. 4.5, Thm. 4.7]{waldspurger}. We will briefly discuss some of the main ideas of the proofs, which are completely combinatorial.
Note that $\cN_{\Sp(2n)}$ is parametrised by pairs $(\lambda,\eps)$, where $\lambda$ is a symplectic partition and $\eps$ is a map $\{\lambda_i \colon i \in \N, \lambda_i \text{ is even}, \lambda_i \neq 0\} \to \{\pm 1\}$. Let $\PPsymp(2n)$ be the set of such pairs $(\lambda,\eps)$.
The relative Weyl groups of $\Sp(2n)$ are of type $C_{n-k(k+1)/2}$ for some $k \in \Z_{\geq0}$ and the generalised Springer correspondence for $\Sp(2n)$ gives a bijection $\PPsymp(2n) \to \bigsqcup_{k \in \Z_{\geq 0},k(k+1)\leq2n} \mathcal P_2(n-k(k+1)/2)$.

For any complex reductive group $G$ and $(C,\E), (C',\E') \in \NNG$, it was shown in \cite{lusztig1986character} that $\mult(C,\E;C',\E')$ is fully described by Green functions $p_{\rho',\rho}$, where $\rho = \GSpr(C,\cE)$ and $\rho' = \GSpr(C',\cE')$. 
The combinatorics for Green functions of the Weyl groups of type $B/C$ (i.e. when $\rho,\rho'$ are irreducible representations of a type $B/C$ Weyl group) were studied in \cite{shoji2001green}, using the theory of symmetric functions. A result on \cite[p. 685]{shoji2001green} can be used to obtain a method to compute such Green functions $p_{\rho',\rho}$ under certain conditions on $(C,\cE)$.
Extending this result, \cite[Proposition 4.2]{waldspurger} states this method of computation for $\Sp(2n)$ is valid under weaker conditions, namely when $C$ is parametrised by a symplectic partition consisting of only even parts.
This result is one of the main ingredients for the proof of \cite[Thm. 4.5]{waldspurger}. 
Afterwards, \cite[Thm. 4.7]{waldspurger} was proved using \cite[Thm. 4.5]{waldspurger} and some of Waldspurger's other combinatorial results.
Finally, a simple algorithm for the computation of $(\lambdamax,[\epsmax])$ in \Cref{thm:max} was described. 

The overall strategy for proving \Cref{thm:max} and \Cref{thm:min} are completely analogous. We will use certain combinatorial results in \cite{waldspurger}, and we will state and prove certain $\SO(N)$ analogues of the combinatorial results in \cite{waldspurger} that are specific to $\Sp(2n)$.
Most of the differences will arise from the difference in the symbol combinatorics for the generalised Springer correspondence for $\SO(N)$ and $\Sp(2n)$.
One of the key parts of the proof will be \Cref{prop:multshoji}, which is the $\SO(N)$ analogue of \cite[Proposition 4.2]{waldspurger} mentioned above and states that Shoji's method for computing Green functions discussed in the paragraph above holds for $G = \SO(N)$ when $C$ is parametrised by an orthogonal partition consisting of only odd parts. 
For all Green functions of $\SO(2n+1)$, as well as the Green functions of $\SO(2n)$ that are of type $B$, the proof of \Cref{prop:multshoji} will be almost completely analogous to \cite[Proposition 4.2]{waldspurger}, using results from \cite{shoji2001green} for type $B$ Weyl groups and symbols.
For Green functions of the Weyl group $W(D_n)$ of $\SO(2n)$, the proof will additionally involve the use of \cite[Proposition 4.9]{SHOJI2002563}, in which Green functions and symbols of type $D$ are described via a reduction to the type $B$ case.

The paper is structured as follows. In \S\ref{sec:comb}, we introduce the combinatorial objects involved in the study of the Green functions as in \cite[\S1,\S2]{waldspurger}.
In particular, for $(\alpha,\beta) \in \cP_2(n)$, we introduce and discuss the subset $P_{A,B;s}(\alpha,\beta,<)$ of $\cP_2(n)$  and the order $<_{\alpha,\beta,k}$ on the set of indices of $\alpha$ and $\beta$. 
This set is particularly important, as it will turn out later that under some conditions, this set will be a singleton that contains the unique bipartition that parametrises the maximal representation associated to $(\alpha,\beta)$.
Furthermore, the order $<_{\alpha,\beta,k}$ is important for the application of Shoji's results regarding the Green functions later.
The main purpose of \S\ref{sec:gsc} is to set up the notation for the combinatorics involved in the generalised Springer correspondence for $\SO(N)$. The notation we use for the symbols will be similar to the notation for the symbols in \cite{waldspurger}, which differs slightly from the usual notation used by Lusztig. We then state the generalised Springer correspondence for $\SO(N)$ as in \cite[\S13]{lusztig2}.
In \S\ref{sec:green}, we define the Green functions related to the Weyl groups of type $B$ and $D$. We discuss the connections between these Green functions and the theory of `type $B/D$' symmetric functions, particularly the Schur functions and the Hall-Littlewood functions, which were described in \cite{shoji2001green} and \cite{SHOJI2002563}. The main result in this section is \Cref{prop:multshoji}. 
Next, in \S\ref{sec:max} we state and prove \Cref{thm:max} and \Cref{thm:min}. They are the analogues of \cite[Thm. 4.5, Thm. 4.7]{waldspurger}. 
At the end of the section, we discuss the situation when $\lambda$ is any orthogonal partition. We use an induction theorem \cite{lusztig2004induction} and certain type $B/C/D$ versions of the Littlewood-Richardson rule to extend the maximality and minimality theorems to the case where $\lambda$ is arbitrary, but with the condition that $(\lambda,[\eps])$ appears in the (non-generalised) Springer correspondence. 
Afterwards, we state and prove an algorithm to compute $(\lambdamax,[\epsmax])$ in \S\ref{sec:alg} (cf. \cite[\S5]{waldspurger}).
Finally, in \S\ref{sec:extra}, we show that when $\lambda$ only has odd parts, then $\lambdamax$ parametrises a quasi-distinguished unipotent class.

The main motivation of the author for studying the maximality and minimality theorems is to determine an explicit algorithm for the Iwahori-Matsumoto dual of particular irreducible representations of certain type $B/C/D$ Hecke algebras with equal or unequal parameters. 
The author futhermore aims to use this to obtain an algorithm for the Aubert-Zelevinsky dual \cite{aubert1995dualite} for particular (but not all) irreducible representations of $p$-adic symplectic and $p$-adic special orthogonal groups. This will be from a different point of view than the work in \cite{atobe2020explicit}, in which different techniques were used to obtain a more general algorithm for the Aubert-Zelevinsky dual for all irreducible representations of $p$-adic symplectic and $p$-adic odd (but not even) special orthogonal groups. 
Furthermore, as an application of such an algorithm, the author aims to use \cite[Theorem 1.3.1]{ciubotaru2021some} to compute certain wavefront sets.
The author intends to pursue all of this in a subsequent paper.

\subsection*{Acknowledgements}
The author would like to thank Dan Ciubotaru for providing invaluable guidance and for many helpful discussions and conversations. The author would also like to thank Jonas Antor and Emile Okada for many useful discussions about the generalised Springer correspondence and perverse sheaves. Furthermore, the author wishes to thank Anne-Marie Aubert for useful suggestions regarding the structure of the introduction. The author was supported by EPSRC.

\section{Partition and symbol combinatorics}\label{sec:comb}

\subsection{Partitions}\label{subsec:partitions}
Let $\cR$ be the set of decreasing sequences $\lambda = (\lambda_1, \lambda_2, \dots)$ of real numbers such that for each $r \in \R$, there exist only finitely many $i \in \N$ such that $\lambda_i \geq r$.
For each $r \in \R$ and $\lambda \in \cR$, we define $\mult_\lambda(r) = \#\{i \in \N \colon \lambda_i = r\}$.
For each $c \in \N$, we define $S_c(\lambda) = \lambda_1 + \dots + \lambda_c$. 
We define an ordering $\leq$ on $\cR$ as follows: for $\lambda,\lambda' \in \cR$, we denote $\lambda \leq \lambda'$ if $S_c(\lambda) \leq S_c(\lambda')$ for all $c \in \N$.
For $\lambda, \lambda' \in \cR$, we define $\lambda + \lambda' = (\lambda_1 + \lambda_1', \lambda_2+\lambda_2',\dots)$, and we define $\lambda \sqcup \lambda'$ to be the unique element of $\cR$ such that for each $n \in \N$, we have $\mult_{\lambda \sqcup \lambda'}(n) = \mult_{\lambda}(n) + \mult_{\lambda'}(n)$.

Let $\cR_f$ be the set of finite decreasing sequences $\lambda = (\lambda_1, \dots, \lambda_n)$ of non-negative 
real numbers. 
For $m,n \in \N$ with $m < n$ and two elements $\lambda = (\lambda_1,\dots,\lambda_m)$ and $\nu = (\nu_1,\dots,\nu_n)$ of $\cR_f$, we identify $\lambda$ and $\nu$ if $\lambda_i = \nu_i$ for $i=1,\dots,m$ and $\nu_j = 0$ for $j = m+1,\dots,n$. Similarly as above, we can define $\mult_\lambda$ and $S_c$ for $c \in \N$, and we similarly define an ordering $\leq$ on $\cR_f$.
For $\lambda \in \cR$ and $\lambda' \in \cR_f$, we define $\lambda + \lambda' \in \cR$ and $\lambda \sqcup \lambda' \in \cR$ similarly as above as well.
For $\lambda \in \cR_f$, let $t(\lambda)$ be the largest integer $i$ such that $\lambda_i \neq 0$ and let $S(\lambda) = S_{t(\alpha)}(\lambda)$

A sequence $\lambda \in \cR_f$ is called a \emph{partition} if it is a sequence of non-negative integers and let $\mathcal{P} \subseteq \cR_f$ be the set of all partitions.
For $N \in \N$, we say that $\lambda$ is a \emph{partition of $N$} if $N = S(\lambda)$.

Let $\alpha, \beta \in \mathcal P$ and consider integers $m_0 \geq t(\alpha)$ and $n_1 \geq t(\beta)$. Let $I := I_{m_0,m_1} := \{ (i,0), (j,1) \colon i \in \{1,\dots,m_0\}, j \in \{1,\dots,m_1\}\}$, which we call the \emph{set of indices of $(\alpha,\beta)$}.
Consider a total order $<$ on $I$ such that for all $e \in \{0,1\}$ and $i,j \in \{1,\dots,m_e\}$, we have $(i,e) < (j,e)$ if and only if $i < j$. 
We call the triple $(m_0,m_1,<)$ an \emph{order on the set of indices of $(\alpha,\beta)$}. 
Any other order $(m_0',m_1',<')$ is said to be equivalent to $(m_0,m_1,<)$ if 
for each $i \in \N$ such that $\alpha_i >0$, we have 
$\{j \in \{1,\dots,m_1\} \colon (j,1) < (i,0)\} = \{j \in \{1,\dots,m_1'\} \colon (j,1) <' (i,0)\}$
and 
for each $j \in \N$ such that $\beta_j >0$, we have 
$\{i \in \{1,\dots,m_0\} \colon (i,0) < (j,1)\} = \{i \in \{1,\dots,m_0'\} \colon (i,0) <' (j,1)\}$.
Often, we simply write $<$ for the triple $(m_0,m_1,<)$, since the $m_0$ and $m_1$ were required to define $I = I_{m_0,m_1}$. We also often write the equivalence class of $(m_0,m_1,<)$ by just $<$, especially in situations where $m_0$ and $m_1$ are not important.
Any equivalence class of orders has a `minimal' order $(m_0^<,m_1^<,<)$ in the sense that $m_0^< \leq m_0$ and $m_1^< \leq m_1$ for all $(m_0,m_1,<)$ in the equivalence class.

Let $X \subseteq \N$ be finite. 
Let $\phi_X \colon \N \to \N \setminus X$ be the unique increasing bijection. Define $\alpha' \in \cP$ by setting $\alpha_i' = \alpha_{\phi_X(i)}$ for all $i \in \N$. For a finite set $Y \subseteq N$, we define $\phi_Y$ and $\beta' \in \cP$ similarly. 
Given an order $(m_0,m_1,<)$ on the set of indices of $(\alpha,\beta)$, we uniquely obtain an order $(m_0',m_1',<')$ on the set of indices of $(\alpha',\beta')$ such that $(i,0) <' (j,1)$ if and only if $(\phi_X(i),0) < (\phi_Y(j),1)$. Let $(\tilde m_0, \tilde m_1, \tilde <)$ be another order on the set of indices of $(\alpha,\beta)$ equivalent to $<$. It is shown in \cite[\S1]{waldspurger} that $(\tilde m_0', \tilde m_1', \tilde <')$ is equivalent to $<'$.

Fix an equivalence class $<$ of orders on the set of indices of $(\alpha,\beta)$.
We shall describe two procedures (a), resp. (b). 
For procedure (a), fix a representative $(m_0,m_1,<)$ such that $m_0 \geq 1$. 
Let $a_1 = 1$. Suppose $a_i$ is defined for some $i \in \N$. If $B_i := \{j \in \{1,\dots,m_1\} \colon (a_i,0) < (j,1) \}$ is non-empty, let $b_i = \min B_i$.
For $i \in \N_{\geq2}$, if $A_i := \{j \in \{1,\dots,m_0\} \colon (b_{i-1},1) < (j,0) \}$ is non-empty, let $a_i = \min A_i$. 
This defines integers $a_1,\dots,a_p,b_1,\dots,b_q \in \N$ for some $p,q \in \N$.
Let $\mu_1 = \alpha_{a_1} + \beta_{b_1} + \alpha_{a_2} + \beta_{b_2} + ...$.
Let $(\alpha',\beta') \in \mathcal P \times \mathcal P$ as above for $X = \{a_1,\dots,a_p\}$ and $Y = \{b_1,\dots,b_q\}$. 
For procedure (b), fix a representative $(m_0,m_1,<)$ such that $m \geq 1$. We define $b_1 = 1$, and we recursively define $a_1, b_2, a_2,b_3,\dots$ similarly as in procedure (a) to obtain integers $a_1,\dots,a_p, b_1,\dots,b_q \in \N$ for some $p,q \in \N$, and a pair of partitions $(\alpha',\beta')$. We also define $\nu_1 = \beta_{b_1} + \alpha_{a_1} + \dots$.

It is shown in \cite[\S1]{waldspurger} that each of the objects defined in both procedures are independent of the choice of representative $(m_0,m_1,<)$. 

\subsection{The sets $P(\alpha,\beta,<)$ and $P_{A,B;s}(\alpha,\beta,<)$}
Let $(\alpha,\beta) \in \cP \times \cP$ and consider an order $<$ on the set of its indices. 
We will define a set $P(\alpha,\beta,<)$ by induction on $m_0^< + m_1^<$. First, define $P(\varnothing, \varnothing,<) = \{(\varnothing,\varnothing)\}$. Suppose that $m_0^< + m_1^< > 0$. If $m_0^< > 0$ (resp. $m_1^< > 0$) we obtain $\mu_1, (\alpha',\beta'),<'$ (resp. $\nu_1,(\alpha'',\beta''),<''$) from procedure (a) (resp. (b)). By induction, $P(\alpha',\beta',<')$ and $P(\alpha'',\beta'',<'')$ have been defined. We define $P(\alpha,\beta,<) = P^a(\alpha,\beta,<) \cup P^b(\alpha,\beta,<)$, where
\begin{align}
P^a(\alpha,\beta,<) &= 
\begin{cases}
\varnothing &\text{if } m_0^< = 0,
\\
\{((\mu_1) \sqcup \mu', \nu') \colon (\mu',\nu') \in P(\alpha',\beta',<')\} &\text{if } m_0^< \neq 0,
\end{cases}
\\
P^b(\alpha,\beta,<)
&= 
\begin{cases}
\varnothing &\text{if } m_1^< = 0,
\\
\{(\mu'', (\nu_1) \sqcup \nu'') \colon (\mu'',\nu'') \in P(\alpha'',\beta'',<'')\} &\text{if }m_1 \neq 0.
\end{cases}
\end{align}

Let $A, B \in \R$, $s \in \R_{>0} $. We define a subset $P_{A,B;s}(\alpha,\beta,<)$ of $P(\alpha,\beta,<)$ by recursion on $m_0^< + m_0^<$. 
Let $P_{A,B;s}(\varnothing, \varnothing,<) = \{(\varnothing,\varnothing)\}$.
Suppose $m_0^< + m_1^< > 0$. If the conditions
\begin{enumerate}
\item \label{a:1} $m_0^< > 0$;
\item \label{a:2} $m_1^< = 0$, or $(1,0) < (1,1)$ and $\alpha_1 + A \geq B$, or $(1,1) < (1,0)$ and $\beta_1 + B \leq A$.
\end{enumerate}
are satisfied, consider $\mu_1, (\alpha',\beta'),<'$ from procedure (a) and let 
\begin{align}
P_{A,B;s}^a(\alpha,\beta,<) &= \{((\mu_1) \sqcup \mu', \nu') \colon (\mu',\nu') \in P_{A-s,B;s}(\alpha',\beta',<')\}.	
\end{align}
If neither conditions hold, we set $P_{A,B;s}^a(\alpha,\beta,<) = \varnothing$. 
Similarly, if the conditions
\begin{enumerate}
\item \label{b:1} $m_1^< > 0$;
\item \label{b:2} $m_0^< = 0$, or $(1,0) < (1,1)$ and $\alpha_1 + A \leq B$, or $(1,1) < (1,0)$ and $\beta_1 + B \geq A$.
\end{enumerate}
are satisfied, consider $\nu_1, (\alpha'',\beta''),<''$ from procedure (b) and let 
\begin{align}
P_{A,B;s}^b(\alpha,\beta,<)  &= \{(\mu'', (\nu_1) \sqcup \nu'') \colon (\mu'',\nu'') \in P_{A,B-s;s}(\alpha'',\beta'',<'')\}.
\end{align}
If neither conditions hold, we set $P_{A,B;s}^b(\alpha,\beta,<) = \varnothing$. 
Finally, let $P_{A,B;s}(\alpha,\beta,<) =  P_{A,B;s}^a(\alpha,\beta,<) \cup P_{A,B;s}^b(\alpha,\beta,<)$.

\begin{proposition}\label{prop:PABdiff}
For any $C \in \R$, we have $P_{A+C,B+C;s}(\alpha,\beta,<) = P_{A,B;s}(\alpha,\beta,<)$.
\end{proposition}
For $C \in \R$ and $\mu, \nu \in \mathcal P$, define 
\begin{align}
[C,-\infty[_s &= (C, C-s, C-2s, \dots) \in \cR,
\\
\Lambda_{A,B;s}(\mu,\nu) &= (\mu + [A,-\infty[_s) \sqcup (\nu + [B,-\infty[_s) \in \cR.
\end{align}

\begin{lemma}
\label{lem:max}
Let $(\mu,\nu) \in P(\alpha,\beta,<)$ and $(\boldsymbol{\mu}, \boldsymbol\nu) \in P_{A,B;s}(\alpha,\beta,<)$. Then
\begin{enumerate}[(a)]
\item $\Lambda_{A,B;s}(\mu,\nu) \leq \Lambda_{A,B;s}(\boldsymbol\mu,\boldsymbol\nu)$,
\item $\Lambda_{A,B;s}(\mu,\nu) = \Lambda_{A,B;s}(\boldsymbol\mu,\boldsymbol\nu)$ if and only if $(\mu,\nu) \in P_{A,B;s}(\alpha,\beta,<)$.
\end{enumerate}
\end{lemma}

\begin{definition}\label{def:pABs}
Let $(\mu,\nu) \in P_{A,B;s}(\alpha,\beta,<)$. By Lemma \ref{lem:max} we can define
\begin{align}
p_{A,B;s}(\alpha,\beta,<) := \Lambda_{A,B;s}(\mu, \nu).
\end{align}
\end{definition}


\subsection{Some results for ${P_{k,-k;2}(\alpha,\beta,<_{\alpha,\beta,k})}$}\label{subsec:P}
Let $n \in \N$.
A pair $(\alpha,\beta) \in \mathcal P \times \mathcal P$ is called a \emph{bipartition} of $n$ if $S(\alpha) + S(\beta) = n$.
Denote by $\mathcal P_2(n)$ the set of bipartitions of $n$.
Let $(\alpha,\beta) \in \mathcal P_2(n)$, $k \in \Z$ and write $\Lambda = \Lambda_{k,-k;2}(\alpha,\beta)$.
Let $m_0 \in \N$ be the smallest integer such that $m_0 \geq t(\alpha)$ and $m_0 - k \geq t(\beta)$ and let $m_1 = m_0 - k$. Let $r = m_0+m_1$. 
Consider the condition
\begin{enumerate}
\item\label{hypoth:1}
$\Lambda^r := (\Lambda_1,\dots,\Lambda_r) \in \cR_f$
has `no multiplicities', i.e. each term of $\Lambda^r$ has multiplicity $1$.
\end{enumerate}
Let $H(n,k)$ be the set of $(\alpha,\beta) \in \mathcal P_2(n)$ such that $\Lambda_{k,-k;2}(\alpha,\beta)$ satisfies condition \ref{hypoth:1} above and let $H_S(n,k) = \{ \Lambda_{k,-k;2}(\alpha,\beta) \in \mathcal P_2(n) \colon (\alpha,\beta) \in H(n,k)\}$.
\begin{remark}\label{rem:ineq}
Note that $m_0 = \frac{r+k}{2}$ and $m_1 = \frac{r-k}{2}$ and 
\begin{align}
\alpha_{m_0} + k + 2 - 2m_0 = \alpha_{\frac{r+k}{2}} + 2 - r &> -r,
\\
\beta_{m_1} - k + 2 - 2m_1 = \beta_{\frac{r-k}{2}}  + 2 - r &> -r,
\\
\alpha_{m_0+1} + k + 2 - 2(m_0+1) &= -r,
\\
\beta_{m_1+1} - k + 2 - 2(m_1+1) &= -r,
\end{align}
so $\Lambda^r$ consists of the first $m_0$ terms of $\alpha + [k,-\infty[_2$ and the first $m_1$ terms of $\beta + [-k,-\infty[_2$.
\end{remark}
If $(\alpha,\beta) \in H(n,k)$, there exists a unique order $<_{\alpha,\beta,k}$ on the indices of $(\alpha,\beta)$ such that 
for $i < m_0$, $j < m_1$, we have $(i,0) <_{\alpha,\beta,k} (j,1)$ if and only if $\alpha_i + k + 2 -2i > \beta_j  - k + 2 -2j$. 

\begin{lemma}\label{lemma:unique}
Let $k \in \Z$, $(\alpha,\beta) \in H(n,k)$ and consider the order $<_{\alpha,\beta,k}$.
Let $(\alpha',\beta')$ be a bipartition obtained from $(\alpha,\beta)$ via procedure (a) or (b) and let $<'$ be the induced ordering on the indices of $(\alpha',\beta')$. Then
$P_{k,-k;2}(\alpha,\beta,<_{\alpha,\beta,k})$ has a unique element $(\mu,\nu)$,
\end{lemma}

\begin{proof}
Let $m_0$ and $m_1 = m_0 - k$ be as above and consider the representative $(m_0,m_1,<)$ of the order $<_{\alpha,\beta,k}$. 
We will prove the result by induction on $m_0^< + m_1^<$.
Clearly, the result holds if $m_0^< = 0$ or $m_1^< = 0$. Suppose $m_0^< > 0$ and $m_1^< > 0$. We may assume that $(1,0) < (1,1)$, otherwise we consider $(\beta,\alpha,-k)$. 
From procedure (a) we obtain a bipartition $(\alpha'$, $\beta') \in \mathcal P_2(n')$ for some $n' \in \N$. Since $(\alpha,\beta) \in H(n,k)$, we have $\alpha_1 + k > \beta_1 - k > -k$, so $P_{k,-k;2}(\alpha,\beta,<) = P_{k,-k;2}^a(\alpha,\beta,<)$ is in bijection with $P_{k-2,-k;2}(\alpha',\beta',<') = P_{k-1,-k+1;2}(\alpha',\beta',<')$. 
Hence 
it suffices to show that $(\alpha',\beta') \in H(n',k-1)$ 
and that $<'$ is equivalent to $<_{\alpha',\beta',k-1}$. By the induction hypothesis, $P_{k-1,-k+1;2}(\alpha',\beta',<')$ has a unique element $(\mu,\nu)$. For Lemma \ref{lemma:PP}, we will also prove that $<'$ is equivalent to $<_{\alpha',\beta',k}$.

Let $k' \in \{k-1,k\}$. 
Procedure (a) applied to $(\alpha,\beta)$ also creates $\mu_1,a_1,\dots, a_T,b_1,\dots,b_t \in \N$ where 
$T \in \{t,t+1\}$. 
Let $m_0' = m_0 - T$ and $m_1' = m_1 - t$ and let $\phi \colon \{1,\dots,m_0'\} \to \{1,\dots,m_0\} \setminus \{a_1,\dots,a_T\}$ and $\psi \colon \{1,\dots,m_1'\} \to \{1,\dots,m_1\} \setminus \{b_1,\dots,b_t\}$ be the unique increasing bijections. For each $i \in \{1,\dots,m_0'\}$, resp. $j \in \{1,\dots,m_1'\}$, let $h(i) \in \{1,\dots,T\}$, resp. $g(j) \in \{1,\dots,t\}$ be the unique element for which $a_{h(i)} < \phi(i) < a_{h(i)+1}$, resp. $b_{g(j)} < \psi(j) < b_{g(j)+1}$. We have $\phi(i) = h(i) + i$, $\psi(j) = g(j)+j$ and the maps $i\mapsto h(i)$, $j \mapsto g(j)$ are weakly increasing.

Let $i,j \in \N$. 
To show that $(\alpha',\beta') \in H(n',k')$, we prove the following:
\begin{enumerate}[(a)]
\item\label{ineq:1} if $i \leq m_0'$, $j \leq m_1'$ and $(i,0) <' (j,1)$, then
\begin{align}
\alpha_i' + k' + 2 - 2i > \beta_j'- k' + 2 - 2j,
\end{align}
\item\label{ineq:2} if $i \leq m_0'$, $j \leq m_1'$ and $(j,1) <' (i,0)$, then
\begin{align}
\beta_j' - k' + 2 - 2j > \alpha_i' + k' + 2 -2i,
\end{align}
\item\label{ineq:3} if $i \leq m_0'$ and $j > m_1'$, then
\begin{align}
\alpha_i' + k' + 2 - 2i \geq - k' + 2 - 2j,
\end{align}
\item\label{ineq:4} if $i > m_0'$ and $j \leq m_1'$, then
\begin{align}
\beta_j' - k' + 2 - 2j \geq k' + 2 -2i,
\end{align}
\item\label{ineq:5} The inequality in \ref{ineq:4} is strict if $\beta_j' \neq 0$.
\end{enumerate}
Assume for now that these inequalities hold.
Let $\Lambda' = \Lambda_{k'-k';2}(\alpha',\beta')$ and for $m \in \N$, let $\Lambda'^m = (\Lambda'_1,\dots,\Lambda'_m)$.
Let $\bm m_0' \in \N$ be the smallest integer such that $\bm m_0' \geq t(\alpha')$ and $\bm m_1' := \bm m_0' - k' \geq t(\beta')$. 
Note that $\Lambda'^{m_0'+m_1'}$ has no multiplicities by \ref{ineq:1} -- \ref{ineq:4}, so if we have $m_0' + m_1' \geq \bm m_0' + \bm m_1'$, then $\Lambda'^{\bm m_0'+ \bm m_1'}$ has no multiplicities and so $(\alpha',\beta') \in H(n',k')$.

If $T = t$, then $m_0' - m_1' = k \geq k'$. Hence $m_0' \geq t(\alpha')$ and $m_0' - k' \geq m_1' \geq t(\beta')$, so we have $m_0' \geq \bm m_0'$ and $m_1' \geq \bm m_1$, so $(\alpha',\beta') \in H(n',k')$.

Suppose $T = t+1$.
Then $m_0' - m_1' = k-1$. 
If $k' = k -1$, then similarly as above, we have $m_0' \geq \bm m_0'$ and $m_1' \geq \bm m_1'$ and so $(\alpha',\beta') \in H(n',k')$. 

Suppose $T = t+1$ and $k' = k$. 
If $\beta_{m_1'} = 0$, then $m_0' \geq t(\alpha')$ and $m_0' - k' = m_1' - 1 \geq t(\beta')$, so $m_0 \geq \bm m_0'$ and $m_1' > \bm m_1'$, so $(\alpha',\beta') \in H(n',k')$. 

Suppose that $T = t+1$, $k' = k$, and $\beta_{m_1'} \neq 0$. Then $m_0' + 1 \geq t(\alpha')$ and $(m_0' + 1) - k' = m_1' \geq t(\beta')$, so $m_0' + 1 \geq \bm m_0'$ and $m_1' \geq \bm m_1'$. 
Now by \ref{ineq:1} -- \ref{ineq:5}, it follows that $\Lambda'^{m_0'+m_1'+1}$ has no multiplicities and since $\bm m_0' + \bm m_1' \leq m_0' + m_1' + 1$, we then have $(\alpha',\beta') \in H(n',k')$.

We conclude that $(\alpha',\beta') \in H(n',k')$, and so we can define $<_{\alpha',\beta',k'}$.
We show that $(m_0',m_1',<')$ and $(\bm m_0',\bm m_1',<_{\alpha,\beta,k'})$ are equivalent. Suppose $i \in \N$ such that $\alpha_i' > 0$. Suppose $j \leq m_1'$ such that $(j,1) <' (i,0)$. Then by \ref{ineq:2}, we have $(j,1) <_{\alpha',\beta',k'} (i,0)$, and by Remark \ref{rem:ineq}, we have $j \leq \bm m_1'$. Suppose $j \leq \bm m_1'$ and $(j,1) <_{\alpha',\beta',k'} (i,0)$. Then $\beta_j' - k' + 2 - 2j > \alpha_i' + k' + 2 - 2i$. By \ref{ineq:3}, we have $j \leq m_1'$ and by \ref{ineq:1}, we have $(j,1) <' (i,0)$. 
Similarly, we can show that for $j \in \N$ such that $\beta_j' > 0$, we have $i  \leq m_0'$ and $(i,0) <' (j,1)$ if and only if $i \leq \bm m_0'$ and $(i,0) <_{\alpha',\beta',k'} (j,1)$. 
Thus $<'$ and $<_{\alpha',\beta',k'}$ are equivalent.

It remains to prove \ref{ineq:1} - \ref{ineq:5}
Let $i \in \{1,\dots,m_0'\}$, $j \in \{1,\dots,m_1'\}$.
Suppose that $(i,0) <' (j,1)$, i.e. $(\phi(i),0) < (\psi(j),1)$. 
Note that $(\phi(i),0) < (b_{h(i)},1)$, and since $(\alpha,\beta) \in H(n,k)$, we have
\begin{align}\label{eq:alphabeta}
a_{\phi(i)} + k + 2 - 2\phi(i) > \beta_{b_{h(i)}} - k + 2 - 2b_{h(i)}.
\end{align}
Since $(\phi(i),0) < (\psi(j),1)$, we have $b_{h(i)} \leq \psi(j)$,
so $\beta_{b_{h(i)}} \geq \beta_{\psi(j)} = \beta_j'$. We have $\alpha_{\phi(i)} = \alpha_i'$, so 
\begin{align}
\alpha_i' + k' + 2 - 2i > \beta_j'- k' + 2 - 2j + 2X,
\end{align}
where
\begin{align}
X = k' - k + j - b_{h(i)} + \phi(i) - i.
\end{align}
We have $\phi(i) - i = h(i)$. Since $b_{h(i)} \leq \psi(j)$, we have $h(i) \leq g(j)$. Furthermore, $q\mapsto b_q$ is strictly increasing, so $q \mapsto q - b_q$ is weakly decreasing, so $g(j) - b_{g(j)} \leq h(i) - b_{h(i)}$. Hence we have
\begin{align}
X 
&= k' - k + j + h(i) - b_{h(i)}
\geq -1 + j + g(j) - b_{g(j)}
= -1 + \psi(j) - b_{g(j)} 
\geq 0,
\end{align}
where the last inequality follows from the fact that $\psi(j) > b_{g(j)}$. Thus we have proved \ref{ineq:1}.

Suppose that $(j,1) <' (i,0)$, i.e. $(\psi(j),1) < (\phi(i),0)$. 
We have $(b_{g(j)},1) < (a_{g(j)+1},0)$, and since $(\alpha,\beta) \in H(n,k)$, we have 
\begin{align}
\beta_{\psi(j)} - k + 2 - 2\psi(j) > \alpha_{a_{g(j)+1}} + k + 2 - 2a_{g(j)+1}.
\end{align}
Since $(\psi(j),1) < (\phi(i),0)$, we have $a_{g(j)+1} < \phi(i)$, so $h(i) \geq g(j)+1$ and so
\begin{align}
\beta_j' - k + 2 - 2j > \alpha_i' + k' + 2 -2i + 2Y,
\end{align}
where
\begin{align}
Y = 
k - k' + i - a_{g(j)+1} + \psi(j) - j
= k - k' + i - a_{g(j)+1} + g(j).
\end{align}
Note that $k-k' \geq 0$.
Similarly as before, the map $p \mapsto p - a_p$ is weakly decreasing, so $(g(j)+1) - a_{g(j)+1} \geq h(i) - a_{h(i)}$, and so
\begin{align}
g(j) - a_{g(j)+1} = (g(j)+1) - a_{g(j)+1} - 1 \geq h(i)-a_{h(i)} - 1 = \phi(i) - i - a_{h(i)} - 1 \geq -i,
\end{align}
where the last inequality holds since $\phi(i) > a_{h(i)}$.
Thus we have $Y \geq 0 + i - i \geq 0$ and \ref{ineq:2} follows.

Let $i \leq m_0'$ and $j > m_1'$. 
Note that $m_0' - m_1' = k$ if $T=t$ and $m_0' - m_1' = k-1$ if $T = t+1$, so $m_0' - m_1' \leq k'+1$. Thus
\begin{align}
2(k' + j - i) \geq 2(k' + m_1' + 1 - m_0') \geq 2(k' - k') \geq 0.
\end{align}
Thus $k'+ 2 - 2i \geq -k' + 2 - 2j$, and \ref{ineq:3} follows.
Similarly, we can show that $k' + 2 - 2j \geq k' + 2 -2i$, and so \ref{ineq:4} and \ref{ineq:5} follow.
\end{proof}

We state a consequence of \cite[(7), p.387]{waldspurger} without proof.
\begin{proposition}\label{prop:symbolterm}
Let $k \in \Z$ and let $(\alpha,\beta) \in H(n,k)$ and let $(\mu,\nu) \in P_{k,-k;2}(\alpha,\beta,<_{\alpha,\beta,k})$. Then the largest term of $\Lambda_{k,-k;2}(\mu,\nu)$ is strictly larger than the other terms of $\Lambda_{k,-k;2}(\mu,\nu)$.
\end{proposition}
We can repeatedly use Proposition \ref{prop:symbolterm} to show that $(\mu,\nu) \in H(n,k)$. However, this is rather heavy on the notation, and it will later follow that $(\mu,\nu) \in H(n,k)$ from Theorem \ref{thm:algorithm}, so it is not necessary to prove this now.

\begin{lemma}\label{lemma:PP}
Let $n \in \N$, $k \in \Z$ and $(\alpha,\beta) \in H(n,k)$. Write $<$ for $<_{\alpha,\beta,k}$. We have
\begin{align}
P_{k/2,-k/2;1/2}(\alpha,\beta,<) = P_{k,-k;2}(\alpha,\beta,<).
\end{align}
\end{lemma}
\begin{proof}
Suppose there exist $h,k \in \Z$ such that $(\alpha,\beta) \in H(n,k)$ and $(\alpha,\beta) \in H(n,h)$
 and such that $<_{\alpha,\beta,h}$ and $<_{\alpha,\beta,k}$ are equivalent. Denote $<_{\alpha,\beta,k}$ by $<$. We shall show that
\begin{align}\label{eq:assertion}
P_{h,-h;2}(\alpha,\beta,<) = P_{h/2,-k/2;1/2}(\alpha,\beta,<) = P_{k,-k;2}(\alpha,\beta,<).
\end{align}
This obviously holds if $m_0^< = 0$ or $m_1^< = 0$, so we assume that $m_0^< > 0$ and $m_1^< > 0$.
We may assume that $(1,0) < (1,1)$, otherwise we consider $(\beta,\alpha,-k)$. From applying procedure (a) to $(\alpha,\beta)$, we obtain $\mu_1$ and $(\alpha',\beta') \in \mathcal P_2(n')$ for some $n' \in \N$. For $j \in \{k,h\}$, we have $(\alpha,\beta) \in H(n,j)$ and so $\alpha_1 + j > \beta_1 - j > -j$. Thus we have 
\begin{align}\label{eq:Paj}
P_{j,-j;2}(\alpha,\beta,<) = P_{j,-j;2}^a(\alpha,\beta,<)
\end{align}
We show that
\begin{align}\label{eq:Pa}
P_{h/2,-k/2;1/2}(\alpha,\beta,<) = P^a_{h/2,-k/2;1/2}(\alpha,\beta,<).
\end{align}
Since $(1,0) < (1,1)$, we have to show that $\alpha_1 + h/2 > -k/2$, i.e. $\alpha_1 > -(h+k)/2$. If $h+k \geq 0$, this holds. Suppose $h + k < 0$. We have $\alpha_1 + k > -h$, so $\alpha_1 > -h-k > -(h+k)/2$ since $h + k < 0$. Thus we have \eqref{eq:Pa} and so $P_{h/2,-k/2;1/2}(\alpha,\beta,<)$ is in bijection with $P_{(h-1)/2,-k/2;1/2}(\alpha',\beta',<')$.

We showed in the proof of Lemma \ref{lemma:unique} that $(\alpha',\beta') \in H(n',j)$ and that $<'$, $<_{\alpha',\beta',j}$ are equivalent for $j=h,h-1,k,k-1$. By the induction hypothesis, \eqref{eq:assertion} holds $(\alpha',\beta')$ for $(h-1,k)$ and $(h,k-1)$. Applying the first equality of \eqref{eq:assertion} for $(h-1,k)$ (resp. the second equality of \eqref{eq:assertion} for $(h,k-1)$) gives the first (resp. second) equality in the following:
\begin{align}\label{eq:assertionprime}
P_{h-1,-h+1;2}(\alpha',\beta',<') = P_{(h-1)/2,-k/2;1/2}(\alpha',\beta',<') = P_{k-1,-k+1;2}(\alpha',\beta',<') =: P'.
\end{align}
By \eqref{eq:Paj} and \eqref{eq:Pa}, any element of the three sets in \eqref{eq:assertion} is of the form $({\mu_1} \sqcup \mu, \nu)$ for some $(\mu,\nu) \in P'$, and so \eqref{eq:assertion} follows. The lemma then follows from \eqref{eq:assertion} for $h = k$.
\end{proof}

\section{The generalised Springer correspondence for $\SO(N)$} \label{sec:gsc}

We briefly describe the generalised Springer correspondence for $G = \SO(N)$ in terms of symbols. The notation will be the same as in \cite[\S 4]{waldspurger} and is slightly different than the notation originally used in \cite{lusztig2}. For convenience, we repeat some definitions and results from the introduction.

Let $N \in \Z_{\geq0}$.
A partition $\lambda$ of $N$ is called \emph{orthogonal} if each even part of $\lambda$ occurs with even multiplicity. We write $\Port(N)$ for the set of orthogonal partitions of $N$. Let $\jordbp(\lambda) = \{i \in \N \colon i \text{ is odd, } \mult_\lambda(i) \neq 0\}$. 
Let $F_2$ be the field with two elements.
Let $F_2[\Delta(\lambda)] = \{\pm1\}^{\Delta(\lambda)}$ be the set of maps $\eps \colon \jordbp(\lambda) \to \{\pm1\} \colon \lambda_i \mapsto \eps_{\lambda_i}$, considered as an $F_2$-vector space. 
Let $F_2[\Delta(\lambda)]'$ be the quotient of $F_2[\Delta(\lambda)]$ by the line spanned by the sum of the canonical basis of this $F_2$-vector space.
Denote by $\PPort(N)$ the set of pairs $(\lambda,[\eps])$ where $\lambda \in \Port(N)$ and $[\eps] \in F_2[\jordbp(\lambda)]'$ is the image of $\eps \in F_2[\Delta(\lambda)]$ under the quotient map.
We define $\PPPort(N)$ to be the set of pairs $(\lambda,\eps)$ with $\lambda \in \Port(N)$ and $\eps \in F_2[\Delta(\lambda)]$.

Let $\cU$ be the set of unipotent classes of $G$.
It is well-known that $\Port(N)$ is in $1$--$1$ correspondence with $\cU$, except the orthogonal partitions of $N$ that only have even parts correspond to precisely two (degenerate) unipotent classes.
This correspondence has the following property.
Suppose $C, C' \in \cU$ are parametrised by $\lambda,\lambda' \in \Port(N)$ respectively. Then $C \subseteq \bar C'$ (i.e. $C \preceq C'$ where $\preceq$ is the closure ordering), if and only if $\lambda \leq \lambda'$.
We say that $\lambda$ is \emph{degenerate} if $\lambda$ only has even parts and we call $\lambda$ \emph{non-degenerate} otherwise.
If $C \in \cU$ is parametrised by $\lambda \in \Port(N)$, then $F_2[\Delta(\lambda)]'$ is in $1$--$1$ correspondence with
the set of irreducible $G$-equivariant local systems on $C$. 
Note that if $C$ is degenerate, then $F_2[\Delta(\lambda)]'$ only contains the empty map, and note that there exist no non-trivial $G$-equivariant local systems on $C$.
Let $\mathcal N_G$ be the set of pairs $(C,\mathcal E)$ where $C$ is a unipotent class in $G$ and $\mathcal E$ is an irreducible $G$-equivariant local system on $C$. 
We obtain a surjective map $\mathcal N_G \to \PPort(N)$ such that the preimage of $(\lambda,\eps)$ has one element if $\lambda$ is non-degenerate, and two elements if $\lambda$ is non-degenerate. We denote the preimage of $(\lambda,[\eps])$ by $\{(C_\lambda^+,\E_{[\eps]}^+),(C_\lambda^-,\E_{[\eps]}^-)\}$. If $\lambda$ is non-degenerate, we write $(C_\lambda,\E_{[\eps]}) = (C_\lambda^+,\E_{[\eps]}^+) = (C_\lambda^-,\E_{[\eps]}^-)$, i.e. we may drop the $+$ and $-$ from the notation. 

Suppose $C \in \cU$ is parametrised by $\lambda \in \Port(N)$ and let $u \in C$. Then $F_2[\Delta(\lambda)]'$ is also in $1$--$1$ correspondence with the set of irreducible representations of the component group $A(u) = Z_G(u) / Z_G^\circ(u)$ of $u$ in $G$.

Let $(\lambda,[\eps]) \in \PPort(N)$. 
Let $\lambda' = \lambda + [0,-\infty[_{1} = (\lambda_1,\lambda_2-1,\lambda_3-2,\dots)$. There exist sequences $z,z' \in \cR$ with integer terms such that $\lambda' = (2z_1,2z_2,\dots) \sqcup (2z_1'-1,2z_2'-1,\dots)$. 
Let $A^\# = z' + [0,-\infty[_1$ and $B^\# = z + [0,-\infty[_1$. Then $A_1^\# \geq B_1^\# \geq A_2^\# \geq B_3^\# \geq \dots$.
A finite subset of $\Z$ is called an interval if it is of the form $\{i,i+1,i+2,\dots,j\}$ for some $i,j \in \Z$.
Let $C$ be the collection of intervals $I$ of $(A \cup B) \setminus (A \cap B)$, with the property that for any $i \in (A \cup B) \setminus ((A \cap B)$ such that $i \notin I$, we have that $I \cup \{i\}$ is not an interval. 
There is an obvious ordering on $C$: for $C', C'' \in C$, $C'$ is larger than $C''$ if any element of $C'$ is larger than any element of $C''$.
Then $C$ is in bijection with $\jordbp(\lambda)$. Let $\jordbp(\lambda) \to C \colon i \mapsto C_i$ be the unique increasing bijection.
Let $t = t(\lambda)$. 
For $i \in \{1,\dots,t\}$, write $\eps(i) = \eps_{\lambda_i}$. For $u \in \{\pm1\}$, define
\begin{align}
J^u &= \{ i \in \{1,\dots,t\} \colon \eps(i)(-1)^{i} = u\}.
\end{align}
Let $M(\lambda,\eps) = M = |J^1| - |J^{-1}|$ and let $w(\lambda,\eps) = w = \sgn(M+1/2)$ so that $w= 1$ if $M = 0$ and $w = \sgn M$ if $M \neq 0$. 
Define
\begin{align}
A_{\lambda,\eps} &=
(A^\# \setminus \bigcup_{i \in \jordbp(\lambda); \eps_i = -w} (A^\# \cap C_i)) \cup ( \bigcup_{i \in \jordbp(\lambda); \eps_i = -w} (B^\# \cap C_i)),
\\
B_{\lambda,\eps} &=
(B^\# \setminus \bigcup_{i \in \jordbp(\lambda); \eps_i = -w} (B^\# \cap C_i)) \cup ( \bigcup_{i \in \jordbp(\lambda); \eps_i = -w} (A^\# \cap C_i)).
\end{align}
If $M \neq 0$ and $-\eps$ is the other representative of $[\eps]$, we have $A_{\lambda,\eps} = A_{\lambda,-\eps}$ and $B_{\lambda,\eps} = B_{\lambda,-\eps}$, so we define the \emph{symbol} of $(\lambda,[\eps])$ to be the ordered pair $S_{\lambda,[\eps]} = (A_{\lambda,\eps},B_{\lambda,\eps}) =  (A_{\lambda,-\eps},B_{\lambda,-\eps})$.
If $M = 0$, $A_{\lambda,\eps} = B_{\lambda,-\eps}$ and $B_{\lambda,\eps} = A_{\lambda,-\eps}$, we define the \emph{symbol} of $(\lambda,[\eps)]$ to be the unordered pair $S_{\lambda,[\eps]} = \{A_{\lambda,\eps},B_{\lambda,\eps}\}$.
Let $X$ be a set, $d \in \Z_{\geq 0}$ and $(x,y) \in X \times X$. Then let
\begin{align}\label{eq:k}
(x,y)_k
=
\begin{cases}
(x,y) \in X \times X &\text{if } k > 0,
\\
\{x,y\} \subseteq X &\text{if } k = 0.
\end{cases}
\end{align}
Following this notation, we see that $S_{\lambda,[\eps]} = (A_{\lambda,\eps},B_{\lambda,\eps})_k$ where $k = |M|$.

We define $p_{\lambda,[\eps]} = p_{\lambda,\eps} = A_{\lambda,\eps} \sqcup B_{\lambda,\eps} = A_{\lambda,\eps}^\# \sqcup B_{\lambda,\eps}^\#$.
Two pairs $(A,B), (A',B') \in \mathcal R \times \mathcal R$ are \emph{similar} if $A\sqcup B = A'\sqcup B'$. 
For any $(\lambda,[\eps]),(\lambda',[\eps']) \in \PPort(N)$, their symbols are similar if and only if $\lambda = \lambda'$.

\begin{remark}\label{rem:symbol}
Let $t = t(\lambda)$, $s = \lfloor t/2 \rfloor$ and $k = |M|$.
Let $A' = (A_{\frac{t+d}{2}}+s,A_{\frac{t+d}{2}-1}+s,\dots,A_1+s)$, $B'=(B_{\frac{t-d}{2}}+s,B_{\frac{t-d}{2}-1}+s,\dots,B_1+s)$ and note that these are increasing sequences.
Then $(A',B')_k$ is the usual symbol in the literature.
The \emph{defect} of $(A,B)$ is defined to be $|\#A'-\#B'|$ and this is equal to $k = |M|$. 
\end{remark}

For $n \in \N$, let $W(B_n) = W(C_n) = S_n \ltimes (\mathbb Z / 2 \mathbb Z)^n$ and $W(D_n) = S_n \ltimes (\mathbb Z / 2 \mathbb Z)^{n-1}$. 
The set $W(B_n)^\vee$ of representations of $W(B_n)$ is in bijection with $\mathcal P_2(n)$.
For $(\alpha,\beta) \in \mathcal P_2(n)$, we write $\rho_{(\alpha,\beta)}$ for the corresponding representation of $W_n$.

Define $\theta\colon \mathcal P_2(n) \to \mathcal P_2(n)$ by $\theta(\alpha,\beta) = (\beta,\alpha)$ for $(\alpha,\beta) \in \mathcal P_2(n)$. Let $\mathcal P_2(n) / \theta$ be the set of $\theta$-orbits
and let $c_{(\alpha,\beta)}$ be the size of the $\theta$-orbit of $(\alpha,\beta)$.
We identify $\mathcal P_2(n)$ with the set of unordered pairs $\{\alpha,\beta\}$ where $(\alpha,\beta) \in \mathcal P_2(n)$.
Then $W(D_n)^\vee$ is parametrised as
\begin{align}
W(D_n)^\vee = \{\rho_{\{\alpha,\beta\},i} \colon \{\alpha,\beta\} \in \mathcal P_2(n) / \theta, i \in \{1,2/c_{(\alpha,\beta)}\}\}.
\end{align}
We sometimes write $\rho_{(\alpha,\beta),i}$ or $\rho_{(\beta,\alpha),i}$ if it is clear from the context that we are talking about $\rho_{\{\alpha,\beta\},i}$.
When $c_{(\alpha,\beta)} = 1$, we also write $\rho_{\{\alpha,\beta\}} = \rho_{\{\alpha,\beta\},1}$.

Let $k \in \N$ such that $k \equiv N \mod 2$ and $k^2\leq N$. Let $(\alpha,\beta) \in \mathcal P_2((N - k^2)/2)$. We define $A_{\alpha,\beta;k} = \alpha + [k,-\infty[_2$ and $B_{\alpha,\beta;k} = \beta + [-k,-\infty[_2$ and define the \emph{symbol of} $(\alpha,\beta)$ to be $S_{\alpha,\beta;k} = (A_{\alpha,\beta;k},B_{\alpha,\beta;k})_k$.

Let $W$ be the Weyl group of $G$. Then $W = W(B_{(N-1)/2})$ (resp. $W = W(D_{N/2})$) if $N$ is odd (resp. even). Let 
\begin{align}
\mathcal W = W^\vee \sqcup \bigsqcup_{k \in \Z_{\geq2}, k^2\leq N, k \equiv N \mod 2} W(B_{(N-k^2)/2})^\vee.
\end{align} 
The Weyl groups $W$ and $W(B_{(N-k^2)/2})$ with $k \in \Z_{\geq 2}$, $k \equiv N \mod 2$ and $k^2 \leq N$ are called the \emph{relative Weyl groups} of $G$. Each relative Weyl group is realised as the the quotient $N_G(L)/L$ where $L$ is a Levi subgroup of a standard parabolic $G$ such that $L$ has a cuspidal pair and $N_G(L)$ is the normaliser of $L$ in $G$.

The generalised Springer correspondence gives a bijection 
\begin{align}
\GSpr \colon \mathcal N_G \to \mathcal W.
\end{align}
Using the parametrisations of $\mathcal N_G$ and $\mathcal W$ described above, we rephrase the generalised Springer correspondence as follows.

\begin{theorem}[Generalised Springer correspondence for $\SO(N)$]\label{thm:gscSO}
Let $N \in \N$.
\begin{enumerate}
\item Suppose $N = 2n+1$ is odd. For each $(\lambda,[\eps]) \in \PPort(N)$, there exists a unique odd integer $k(\lambda,[\eps]) = k \in \N$ and a unique pair $(\alpha,\beta) \in \mathcal P_2((N-k^2)/2)$ such that $(A_{\lambda,\eps},B_{\lambda,\eps}) = (A_{\alpha,\beta;k},B_{\alpha,\beta;k})$. Conversely, for each odd $k \in \N$ such that $k^2 \leq N$, and for each $(\alpha,\beta) \in \mathcal P_2((N-k^2)/2)$, there exists a unique $(\lambda,[\eps]) \in \PPort(N)$ such that $(A_{\lambda,[\eps]},B_{\lambda,[\eps]}) = (A_{\alpha,\beta;k},B_{\alpha,\beta;k})$. Thus we have a bijection
\begin{align}
\Phi_N \colon \PPort(N) \to \bigsqcup_{k \in \N \mathrm{odd},\, k^2 \leq N} \mathcal P_2((N-k^2)/2).
\end{align}
\item Suppose $N = 2n$ is even. Let $(\lambda,[\eps]) \in \PPort(N)$. Then one of the following is true:
\begin{itemize}
\item 
There exists a unique $\{\alpha,\beta\} \in {\mathcal P}_2(N/2) / \theta$ with $\alpha \neq \beta$ such that $\{A_{\lambda,[\eps]},B_{\lambda,[\eps]}\} = \{A_{\alpha,\beta;0},B_{\alpha,\beta;0}\}$,
\item There exists a unique even integer $k(\lambda,[\eps]) = k \in \N$ and a unique pair $(\alpha,\beta) \in \mathcal P_2((N-k^2)/2)$ such that $(A_{\lambda,[\eps]},B_{\lambda,[\eps]}) = (A_{\alpha,\beta;k},B_{\alpha,\beta;k})$. 
\end{itemize}
Conversely, we have
\begin{itemize}
\item For each $\{\alpha,\beta\} \in \mathcal P_2(N/2)/\theta$, there exists a unique $(\lambda,[\eps]) \in \PPort(N)$ such that $\{A_{\lambda,[\eps]},B_{\lambda,[\eps]}\} = \{A_{\alpha,\beta;0},B_{\alpha,\beta;0}\}$,
\item For each even $k \in \N$ such that $k^2 \leq N$ and for each $(\alpha,\beta) \in \mathcal P_2((N-k^2)/2)$, there exists a unique $(\lambda,[\eps]) \in \PPort(N)$ such that $(A_{\lambda,[\eps]},B_{\lambda,[\eps]}) = (A_{\alpha,\beta;k},B_{\alpha,\beta;k})$.

\end{itemize}
Thus we have a bijection
\begin{align}
\Phi_N \colon \PPort(N) \to (\mathcal P_2(N/2) / \theta) \sqcup \bigsqcup_{k \in \N \mathrm{even},\,k^2\leq N} \mathcal P_2((N-k^2)/2).
\end{align}
\end{enumerate}
\end{theorem}

\begin{remark}\label{rem:parity}
Note that in \Cref{thm:gscSO}, we have $N \equiv t(\lambda) \equiv k(\lambda,[\eps]) \mod 2$. Furthermore, it is well-known that $k(\lambda,[\eps])$ is equal to the defect of $(\lambda,[\eps])$ (see \cite[\S13]{lusztig2}), hence $k(\lambda,[\eps]) = |M| = |J^1| - |J^{-1}|$ by \Cref{rem:symbol}.
\end{remark}

\begin{remark}\label{rem:symbolodd}
Let $(\lambda,[\eps]) \in \PPort(N)$ such that $\lambda$ only has odd parts. 
Let $t = t(\lambda)$, $t^+ = \lceil t/2 \rceil$ and $t^- = \lfloor t/2 \rfloor$.
Then
\begin{align}\label{eq:gscodd0}
A^\# &= \left( \frac{\lambda_{2i-1} + 2i-1 }{2} \colon i = 1,\dots, t^+ \right) \sqcup (-t - 2i \colon i \in \Z_{\geq0}),
\\
B^\# &= \left( \frac{\lambda_{2i} + 2i-1 }{2} \colon i = 1,\dots, t^- \right) \sqcup (-t - 2i \colon i \in \Z_{\geq0}),
\end{align}
and $A_1^\# > B_1^\# > A_2^\# > B_2^\# > \dots > A_{t^+}^\# > B_{t^+}^\#$. Hence the largest $t$ terms of $p_{\lambda,\eps} = A_{\lambda,\eps} \sqcup B_{\lambda,\eps} = A^\# \sqcup B^\#$ are all distinct.
Suppose $(\alpha,\beta)_k = \Phi_N(\lambda,[\eps])$ with $k = k(\lambda,[\eps])$. Note that $t \equiv k \mod 2$ and let $m_0 = \frac{t+k}{2}$ and $m_1 = \frac{t-k}{2}$. 
It is easy to check that $m_0$ and $m_1$ are the same as in \S\ref{subsec:P}, i.e.
$m_0$ is the smallest integer such that $m_0 \geq t(\alpha)$ and $m_0 - k \geq t(\beta)$, so 
$p_{\lambda,\eps} = \Lambda_{k,-k;2}(\alpha,\beta) \in H_S((N-k^2)/2,k)$, i.e.  $(\alpha,\beta) \in H((N-k^2)/2,k)$.
\end{remark}

\section{Green functions and multiplicities}\label{sec:green}

\subsection{Green functions and the Springer correspondence}
Recall that for $N \in \N$ odd (resp. even), the Weyl group of $\SO(N)$ is $W(B_{(N-1)/2})$ (resp. $W(D_{N/2})$) and the other relative Weyl groups are of the form $W(B_{(N-k^2)/2})$ for some odd (resp. even) integer $k \in \Z_{\geq0}$ such that $k^2 \leq N$.
In other words, for each $n \in \Z_{\geq0}$, we can interpret $W(B_n)$ as a relative Weyl group of $\SO(2n+k^2)$ for any $k \in \N$, and $W(D_n)$ as the Weyl group of $\SO(2n)$.
Let $n \in \Z_{\geq0}$ and let $W_{n}$ be $W(B_{n})$ or $W(D_{n})$.
Let $t$ be an indeterminate and let $d = 1$ if $W_n = W(D_n)$ and $d = 2$ if $W_n = W(B_n)$.  The Poincaré polynomial of $W_{n}$ is given by
\begin{align}
P_{W_{n}}(t) 
= 
\frac{t^{dn}-1}{t-1}\cdot\prod_{i=1}^{n-1} \frac{t^{2i}-1}{t-1}.
\end{align}
Let $V$ be the reflection representation of $W$. For a class function $f$ of $W_{n}$, define $R(f) \in \Z[t]$ by
\begin{align}
R(f) = \frac{(t-1)^{\dim V}P_W(t)}{|W|}\sum_{w \in W_{n}} \frac{\det_V(w)f(w)}{\det_V(t\cdot\text{id}_V-w)}.
\end{align}
If $f$ is an irreducible character of $W_n$, then $R(f)$ is called the fake degree of $f$. For $\rho \in W^\wedge$, let $\chi(\rho)$ denote its character.

Let $N^*$ be the number of reflections of $W_{n}$. 
Let $\Omega = (\omega_{\rho,\rho'})_{\rho,\rho'\in W_n^\wedge}$ be the $|W_{n}| \times |W_{n}|$ matrix over $\Z[t]$ defined by
\begin{align}
\omega_{\rho,\rho'} = t^{N^*} R(\chi(\rho) \otimes \chi(\rho') \otimes \overline{\det}_V).
\end{align}

Let $(\alpha,\beta) \in \cP_2(n)$.
If $W_n = W(B_n)$, let $\rho = \rho_{(\alpha,\beta)}$, and if $W_n = W(D_n)$, let $i \in \{1,2/c_{(\alpha,\beta)}\}$ and let $\rho = \rho_{\{\alpha,\beta\},i}$.
Let $k \in \Z_{\geq0}$, $N = 2n + k^2$, and $r \in \N$ such that $r \geq t(\alpha) + t(\beta) + k$ (so $r$ is as in Remark \ref{rem:ineq}).
If $W_n = W(D_n)$, we assume that $k = 0$.
Let $\Lambda := \Lambda_{\rho;k} := \Lambda_{k,-k;2}(\alpha,\beta)$ and $\Lambda^{(0)} := [k,-\infty[_2 \sqcup [-k,-\infty[_2$.
Let $(C,\E) = \GSpr^{-1}(\rho)$ and suppose $C$ is parametrised by an orthogonal partition $\lambda$ of $N$. We define
\begin{align}
a_k(C) := a_k(\lambda) := a_k(\rho) := a_k(\alpha,\beta)
:=
\sum_{1\leq i< j \leq r} \min(\Lambda_i,\Lambda_j) - \sum_{1\leq i< j \leq r} \min(\Lambda^{(0)}_i,\Lambda^{(0)}_j).
\end{align}
We will usually drop $k$ from the notation. 
Note in particular that $a_k(\rho)$ is well-defined when $W_n = W(D_n)$, since in this case, we assumed that $k=0$.
Also note that $a_k(\rho)$ does not depend on $r$. 
We write $\rho \sim \rho'$ if the symbols $\Lambda_{\rho;k}$ and $\Lambda_{\rho';k}$ are similar.
Then $a_k(\rho) = a_k(\rho')$ if $\rho \sim \rho'$, so $a_k(C)$ and $a_k(\lambda)$ are well-defined.
Note that $a_k$ is the same as in \cite[(1.2.2)]{shoji2001green} and if $k = 1$, then $a_k$ is the same as the function $b$ 
in \cite[4.4]{LUSZTIG1986146}. 

Define a total order $\prec_k$ on $W_{n}^\wedge$ such that for $\rho,\rho' \in W_{n}^\wedge$ we have $a_k(\rho) \geq a_k(\rho')$ if $\rho \prec_k \rho'$, and such that each similarity class forms an interval. 
Again, we will usually drop $k$ from the notation.
The order $\prec_k$ uniquely defines a total order $\prec_k$ on $\mathcal P_2(n)$.
In \cite{shoji2001green} (for type $B$) and in \cite{SHOJI2002563} (for type $D$) the following theorem is proved:

\begin{theorem}\label{theorem:shoji}
Let $n \in \N$, $k \in \Z_{\geq 0}$, let $W_{n} = W(B_n)$ or $W_n = W(D_n)$, and write $\prec$ for $\prec_k$.
If $W_{n}$ is of type $D_{n}$, we assume that $k = 0$.
There exist unique $|W_{n}|\times|W_{n}|$ matrices $P = P^{(k)}$ and $\Lambda = \Lambda^{(k)}$ over $\Q(t)$ such that
\begin{align}
P\lambda P^t &= \Omega, &&
\\
\lambda_{\rho,\rho'} &= 0 && \text{if } \rho \not\sim \rho',
\\
p_{\rho,\rho'} &= 0 && \text{unless $\rho' \prec \rho$ and $\rho \not\sim \rho'$, or $\rho = \rho'$,}
\\
p_{\rho,\rho} &= t^{a(\rho)}.
\end{align}
Furthermore, the entries of $P$ and $\Lambda$ lie in $\Z[t]$. The polynomials $p_{\rho,\rho'}$ are called \emph{Green functions of $W_n$ for $k$}. 
\end{theorem}

Except when $W_n = W(B_n)$ and $k = 0$, we can view $W_{n}$ as a relative Weyl group of $\SO(N) = \SO(2n+k^2)$. 
Regardless, we considered the Green functions of $W_n = W(B_n)$ for $k = 0$ in \Cref{theorem:shoji}, as we will later see from \Cref{prop:shojigreenD} that they are related to the Green functions of $W(D_n)$ for $k = 0$.

The matrix $P^{(k)}$ is closely related to the solutions of a similar matrix equation considered in \cite[\S24]{lusztig1986character}. We will briefly discuss this relation. 
Let $\rho \in W_n^\wedge$ and let $(C,\E) = \GSpr^{-1}(\rho) \in \cN_{\SO(N)}$. Define the \emph{Springer support} of $\rho$ to be $\supp \rho := C$.
Fix a maximal torus $T$ of $G$.
Recall that $W_{n} \cong N_G(L)/L$ for some Levi subgroup $L$ of $G$ with a cuspidal pair.
Let $D$ be the diagonal $|W_{n}|\times|W_{n}|$ matrix over $\Q(t)$ such that for all $\rho \in W^\wedge$, $D_{\rho,\rho} = t^{a(\rho)}$. 
Let $\tilde \Omega = D^{-1} \Omega D^{-1}$.
Let $\rho,\rho' \in W^\wedge$, $C =\supp \rho$ and $C' = \supp \rho'$. 
The following result can be verified using \cite[Cor. 6.1.4.]{collingwood}.

\begin{proposition}\label{prop:a-function}
We have $2a(\rho) = \dim G - \dim Z_L^\circ - \dim C$, where $Z_L^\circ$ is the connected component of the centre of $L$.
\end{proposition}

Denote by $\tilde \omega_{\rho,\rho'}$ the polynomial $\Omega_{i',i}$ in \cite[\S24.7]{lusztig1986character}, where $i = \GSpr^{-1}(\rho)$ and $i' = \GSpr^{-1}(\rho')$.
By \Cref{prop:a-function}, there exists a $c \in \Z$ such that for all $\rho,\rho' \in W_n^\wedge$, we have
\begin{align}
\omega_{\rho,\rho'} = t^{a(\rho) + a(\rho') + c} \tilde \omega_{\rho,\rho'}.
\end{align}
Note that $\prec$ is compatible with the closure order on $\cU$: if $C$ is strictly smaller than $C'$ in the closure order, then $a(\rho) > a(\rho')$ by \Cref{prop:a-function}, hence $\rho \prec \rho'$ and $\rho \not\sim \rho'$, and if $C = C'$, then $\Lambda_{\rho,k} = \Lambda_{\rho',k}$ by the generalised Springer correspondence, and so $\rho \sim \rho'$.
Thus we can apply \cite[Theorem 24.8(b)]{lusztig1986character} with this order $\prec$, and by taking the transpose of the matrix equation in \cite[Theorem 24.8(b)]{lusztig1986character}, this theorem becomes:

\begin{theorem}\label{theorem:lusztig}
Let $n,k,W_n,\prec$ be as in \Cref{theorem:shoji}, except we assume that $k \neq 0$ if $W_n = W(B_n)$.
Then there exist unique $|W_{n}|\times|W_{n}|$ matrices $\tilde P = \tilde P^{(k)}$ and $\tilde \Lambda = \tilde \Lambda^{(k)}$ over $\Q(t)$ such that
\begin{align}
\tilde P \tilde\lambda \tilde P^t &= \tilde \Omega, &&
\\
\tilde\lambda_{\rho,\rho'} &= 0 && \text{if } \rho \not\sim \rho', \label{LScondition1}
\\
\tilde p_{\rho,\rho'} &= 0 && \text{unless $\rho' \prec \rho$ and $\rho \not\sim \rho'$, or $\rho = \rho'$,} \label{LScondition2}
\\
\tilde p_{\rho,\rho} &= 1.
\end{align}
Furthermore, the entries of $\tilde P, \tilde \Lambda$ lie in $\Z[t]$.
\end{theorem}

Consider $P = P^{(k)}$ and $\Lambda$ from \Cref{theorem:shoji}.
Then it is easy to see that $\tilde P = D^{-1}P$ and $\tilde \Lambda = t^{-c}\Lambda$ are the solutions to the matrix equation in \Cref{theorem:lusztig}.
Thus we have now related $P$ in \Cref{theorem:shoji}, which has a combinatorial interpretation as shown in \cite{shoji2001green}, to $\tilde P$ in \Cref{theorem:lusztig}, which has a geometric interpretation as shown in \cite{lusztig1986character}. We shall discuss the geometric interpretation later in \S\ref{subsec:mult}.

\begin{remark}\label{rem:greenindex}
If $W_n = W(B_n)$ and $\rho = \rho_{\bm\alpha}$ and $\rho' = \rho_{\bm\alpha'}$ for $\bm\alpha,\bm\alpha' \in \cP_2(n)$, we will also write $p_{\bm\alpha,\bm\alpha'} = p_{\rho,\rho'}$ and $\tilde p_{\bm\alpha,\bm\alpha'} = \tilde p_{\rho,\rho'}$.
If $W_n = W(D_n)$ and $\rho = \rho_{\bm\alpha,i}$ and $\rho' = \rho_{\bm\alpha',i'}$ for $\bm\alpha,\bm\alpha'\in\cP_2(n)/\theta$ and $i \in \{1,1/c_{\bm\alpha}\}$, $i' \in \{1,1/c_{\bm\alpha'}\}$,
we will similarly write $p_{\bm\alpha^i,\bm\alpha'^{i'}} = p_{\rho,\rho'}$ and $\tilde p_{\bm\alpha^i,\bm\alpha'^{i'}} = \tilde p_{\rho,\rho'}$. We will furthermore drop the $i$ (or $i'$) from the notation if $c_{\bm\alpha} = 1$ (or $c_{\bm\alpha'} = 1$).
\end{remark}

\subsection{Symmetric polynomials}\label{subsec:sympol}
The matrix of Green functions $P^{A_{n-1}}$ for $S_n$ has an interpretation in terms of symmetric functions, see for instance \cite[Ch. III \S7]{macdonald1998symmetric}, where it is shown that $P^{A_{n-1}}$ is equal to the Kostka matrix, which is the transition matrix between Schur functions and Hall-Littlewood functions. 
Shoji extended these results to complex reflection groups in \cite{shoji2001green} and \cite{SHOJI2002563}. In particular, the connection between the Green functions of type $B$/$C$ and $D$ and symmetric functions were described. The result \cite[p.685]{shoji2001green} allows one to compute the Green functions for Weyl groups of type $B$ under some conditions. Waldspurger proved a certain generalisation \cite[Proposition 4.2]{waldspurger} of this result for the relative Weyl groups of $\Sp(2n)$. \Cref{prop:multshoji} and \Cref{lemma:QR} together are a direct analogue of \cite[Proposition 4.2]{waldspurger} for $\SO(N)$.

We briefly describe the symmetric functions used by Shoji, but we shall not include their definitions here. 
The main purpose of this subsection is to highlight the main ideas in \cite{shoji2001green} are relevant for Waldspurger's proof of \cite[Proposition 4.2]{waldspurger} and the proof of \Cref{prop:multshoji}.

Let $(\alpha,\beta) \in \mathcal P \times \mathcal P$ and let $\bm m = (m_0,m_1) \in \N \times \N$ such that $m_0 \geq t(\alpha)$ and $m_1 \geq t(\beta)$ and let $k = m_0 - m_1$.
For $d=1,2$ and $j=1,\dots,m_d$ define indeterminates $x_j^{(d)}$ and write $x^{(d)} = (x_j^{(d)})_j$ and $x = (x_j^{(d)})_{d,j}$,
Consider the following symmetric polynomials in $\Z[x]$ and $\Z[x,t]$:
\begin{enumerate}[(I)]
\item\label{sym1} (Schur functions)
 $s_{(\alpha,\beta)}(x) = s_{\alpha}(x^{(0)})s_{\beta}(x^{(1)}) \in \Z[x]$ as in \cite[\S2.1]{shoji2001green}, where $s_{\alpha}(x^{(0)}) \in \Z[x^{(0)}]$ and $s_{\beta}(x^{(1)}) \in \Z[x^{(1)}]$ are the usual Schur functions attached to the partitions $\alpha$ and $\beta$,
\item\label{sym2} (Monomial symmetric functions) $m_{(\alpha,\beta)}(x) \in \Z[x]$ as in \cite[\S2.1]{shoji2001green},
\item\label{sym3} (Complete symmetric functions) $h_{(\alpha,\beta)}(x) \in \Z[x]$ as in \cite[\S6.7]{shoji2001green},
\item\label{sym4} $q_{(\alpha,\beta)}(x,t) \in \Z[x,t]$ (\cite[\S2.4]{shoji2001green}),
\item\label{sym5} $R_{(\alpha,\beta)}(x,t) = R_{(\alpha,\beta),<}(x,t) \in \Z[x,t]$ as in \cite[(3.2.1)]{shoji2001green}, where $<$ stands for the order $(m_0,m_1,<)$ on the set of indices of $(\alpha,\beta)$, 
\item\label{sym6} (Hall-Littlewood functions) $P_{(\alpha,\beta)}'(x,t) = P_{(\alpha,\beta),k}'(x,t) \in \Z[x,t]$ as in \cite[Theorem 4.4]{shoji2001green}. We write $P'$ instead of $P$ to the notation to avoid confusion with $P$ in Theorem \ref{theorem:shoji}). Note that in \emph{loc. cit.}, the Hall-Littlewood functions are indexed by certain symbols that are equivalent to $\Lambda_{k,-k;2}(\alpha,\beta)$, so we can also index them by $(\alpha,\beta),k$,
\item\label{sym7} 
$Q_{(\alpha,\beta)}(x) = Q_{(\alpha,\beta),k}(x) \in \Z[x,t]$ as in \cite[Theorem 4.4 (ii)]{shoji2001green}. 
\end{enumerate}

The polynomials above all satisfy a `stability property'. The polynomials \ref{sym1} -- \ref{sym5}, can be defined replacing $m_0$ with $m_0+1$ (resp. $m_1$ with $m_1+1$). Evaluating $x_{m_0+1}^{(0)} = 0$ (resp. $x_{m_1+1}^{(1)} = 0$) yields the same polynomial as in the definition for $(m_0,m_1)$. 
The polynomials $P'_{(\alpha,\beta),k}$ and $Q_{(\alpha,\beta),k}$ depend on $k = m_0 - m_1$, and they can be defined when we replace $(m_0,m_1)$ by $(m_0+1,m_1+1)$, and evaluating $x_{m_0+1}^{(0)}=0$ and $x_{m_1+1}^{(1)}=0$ yields the same polynomial as in the definition for $(m_0,m_1)$.
As such, the polynomials above can be viewed as functions in infinitely many variables.
These functions form bases for certain $\Z$ and $\Q(t)$-modules as follows. 
Denote by
\begin{align}
\Xi_{(m_0,m_1)} = 
\Z[x_i^{(0)} \colon i = 1,\dots,m_0]^{S_{m_0}}
\otimes 
\Z[x_i^{(1)} \colon i = 1,\dots,m_1]^{S_{m_1}}
\end{align}
the ring of symmetric polynomials with variables $x$ with respect to $S_{m_0} \times S_{m_1}$. Then $\Xi_{(m_0,m_1)}$ has a graded ring structure $\Xi_{(m_0,m_1)} = \bigoplus_{i \geq 0} \Xi_{(m_0,m_1)}^i$, where each $\Xi_{(m_0,m_1)}^i$ consists of homogeneous symmetric polynomials of degree $i$ and the zero polynomial. Consider the inverse limit 
\begin{align}
\Xi^i = \varprojlim_{(m_0,m_1)} \Xi_{(m_0,m_1)}^i
\end{align}
with respect to the obvious restrctions $\Xi_{(m_0+\ell,m_1+\ell)}^i \to \Xi_{(m_0,m_1)}^i$ where $\ell \in \N$. Define 
\begin{align}
\Xi = \bigoplus_{i\geq0} \Xi^i.
\end{align}
For $n, m_0, m_1 \in \N$ with $m_0, m_1 \geq n$, we have that $(s_{(\alpha,\beta)}(x))_{(\alpha,\beta) \in \mathcal P_2(n)}$ is a basis for the $\Z$-module $\Xi_{(m_0,m_1)}^n$, as well as the $\Q(t)$-vector space $\Xi_{\Q,(m_0,m_1)}^n[t] := \Q(t) \otimes_{\Z} \Xi_{m_0,m_1}^n$. 
By the stability property, $(s_{(\alpha,\beta)}(x))_{(\alpha,\beta) \in \mathcal P_2(n)}$ is also a $\Z$-basis (resp. $\Q(t)$-basis) of $\Xi^n$ (resp. $\Xi_{\Q}^n[t] := \Q(t) \otimes_{\Z} \Xi^n$).
As polynomials (resp. functions in infinitely many variables), the functions in \ref{sym2} -- \ref{sym7} are shown in \cite{shoji2001green} to be bases of the $\Q(t)$-vector space $\Xi_{\Q,(m_0,m_1)}^n[t]$ (resp. $\Xi_{\Q}^n[t]$), where $(\alpha,\beta)$ runs through $\mathcal P_2(n)$.  
We want to study the transition matrices between some of these bases.
By the stability properties, we see that the transition matrices do not depend on $m_0$ and $m_1$ as long as $m_0\geq n$, $m_1 \geq n$ and $m_0 - m_1 = k$, and that the transition matrix between any two bases in $\Xi_{\Q,(m_0,m_1)}^n[t]$ is the same as the transition matrix between between these two bases in $\Xi_{\Q}^n[t]$.
Given any two bases $X$ and $Y$ of a vector space, denote by $M(X,Y)$ the transition matrix from $Y$ to $X$. 

\begin{theorem}[{\cite[Theorem 5.4]{shoji2001green}}]\label{thm:kostka}
Let $n \in \N$, $k \in \Z_{\geq0}$ 
and let $P = P^{(k)}$ be as in Theorem \ref{theorem:shoji} for the Weyl group $W(B_n)$. Then $P = K^{B_n}(t^{-1})T_k$, where $T_k$ is the diagonal $|W(B_n)| \times |W(B_n)|$ matrix with entries $(T_k)_{\rho,\rho} = t^{a_k(\rho)}$ and $K^{B_n}(t) = M(s,P')$ is the transition matrix between the Hall-Littlewood functions $P'_{(\alpha,\beta),k}$ and the Schur functions $s_{(\alpha,\beta)}$, where $(\alpha,\beta) \in \mathcal P_2(n)$. We call $K^{B_n}(t)$ a \emph{Kostka matrix} of type $B$, and the entries of $K^{B_n}(t)$ are called \emph{Kostka polynomials}.
\end{theorem}

Note that the Kostka polynomials in \Cref{thm:kostka} depend on $k$, but we drop it from the notation for simplicity.

The Green functions of $W(D_n)$ for $k = 0$ are described in \cite{SHOJI2002563}.
Let $T_0'$ be the diagonal $|W(D_n)|\times|W(D_n)|$ matrix with diagonal entries $(T_0')_{\rho,\rho} = t^{a_0(\rho)}$.
Similar to the type $B$ case, it is shown in \emph{loc. cit.} that the matrix of $P = P^{(0)}$ of Green functions of $W(D_n)$ for $k=0$ is equal to the transition matrix $K^{D_n}(t^{-1})T_0' $, where $K^{D_n}(t)$ is a transition matrix between certain Schur functions and Hall-Littlewood functions `of type $D$', defined slightly different than in \ref{sym1} and \ref{sym6}.
Using \cite[Proposition 4.9]{SHOJI2002563}, we find the following result for $K^{D_n}(t)$.
If $n$ is even, let $K^{A_{n/2-1}}(t)$ be the Kostka matrix of $S_{n/2}$ as in \cite[III.6]{macdonald1998symmetric}.
From \Cref{thm:kostka}, we see that $K^{B_n}$ depends on $k$.
For the following, we assume that $k=0$.  
We will index the Kostka polynomials of type $A_{n/2-1}$ and $B_n$ by partitions of $n/2$ and bipartitions of $n$, respectively, and we will index the Kostka polynomials of type $D_n$ by $\bm \alpha^i$ where $\bm \alpha \in \cP_2(n)/\theta$ and $i \in \{1,2/c_{\bm\alpha}\}$, and we drop the $i$ if $c_{\bm\alpha} = 2$ (cf. \Cref{rem:greenindex}).

\begin{proposition}\label{prop:shojigreenD}
Let $\bm\alpha,\bm\beta \in \mathcal P_2(n)$.
\begin{enumerate}[wide,ref={\theproposition(\arabic*)}]
\crefalias{enumi}{proposition}
\item\label{prop:shojigreenD1} Suppose $\theta(\bm\alpha) \neq \bm\alpha$. 
Then for $j=1,2$, we have
\begin{align}
K_{\bm\beta^j,\bm\alpha}^{D_n}(t) = \frac{c_{\bm\beta}}{2} (K_{\bm\beta,\bm\alpha}^{B_n}(t) + K_{\bm\beta,\theta\bm\alpha}^{B_n}(t)).
\end{align}
\item\label{prop:shojigreenD2} Suppose $\theta(\bm\alpha) = \bm\alpha$ and $\theta(\bm\beta) \neq \bm\beta$. Then for $i = 1,2$ we have
\begin{align}
K_{\bm\beta,\bm\alpha^i}^{D_n}(t) = K_{\bm\beta,\bm\alpha}^{B_n}(t) = K_{\theta\bm\beta,\bm\alpha}^{B_n}(t).
\end{align}
\item\label{prop:shojigreenD3} Suppose $\theta(\bm\alpha) = \bm\alpha = (\alpha,\alpha)$ and $\theta(\bm\beta) = \bm\beta = (\beta,\beta)$. Then $n$ is even and for $i=1,2$ and $j=1,2$, we have
\begin{align}\label{eq:shojigreenD3}
K_{\bm\beta^j,\bm\alpha^i}^{D_n}(t) = (-1)^{i+j}K_{\beta,\alpha}^{A_{n/2-1}}(t^2) + K_{\bm\beta,\bm\alpha}^{B_n}(t).
\end{align}
\end{enumerate}
\end{proposition}

\begin{remark}
It may not seem obvious from \cite[Proposition 4.9]{SHOJI2002563} that we have to evaluate $K_{\bm\beta,\bm\alpha}^{A_{n/2-1}}$ in $t^2$ in \Cref{prop:shojigreenD3}, as it comes from a subtlety involved in defining the Hall-Littlewood functions.
In the notation of \emph{loc. cit.}, when $j = 1$, then the $P_{z}^j$ are the usual Hall-Littlewood functions of $S_{n/2}$.
Since $h_j = h_1 = 2$, the term $t^{h_j}$ in \cite[(3.3.7)]{SHOJI2002563} is $t^2$ and so the $P_z^j$ are the Hall-Littlewood functions of $S_{n/2}$ evaluated in $t^2$ rather than $t$. 
As such, the Kostka matrix of $S_{n/2}$ that appears in \cite[Proposition 4.9]{SHOJI2002563} should also be evaluated in $t^2$. 
Similarly, when $j = 0$, we have $h_j = 1$, so the Kostka matrices of type $B$ appearing in \cite[Proposition 4.9]{SHOJI2002563} are simply evaluated in $t$.

Furthermore, we will only need \Cref{prop:shojigreenD1} later, but we include the other two statements as well for completeness. In this case, we thus have a description of the Green functions of $W(D_n)$ for $k=0$ in terms of Green functions of $W(B_n)$ for $k = 0$, which are described by \Cref{theorem:shoji} for $W_n = W(B_n)$ and $k = 0$.
\end{remark}

The following is a direct consequence of \cite[(3.3.2), \S3.9]{SHOJI2002563}.
Suppose $\bm\alpha,\bm\beta \in \mathcal P_2(n)/\theta$ and $i \in \{1,2/c_{\bm\alpha}\}$, $j \in \{1,2/c_{\bm\beta}\}$.
Then 
\begin{align}\label{eq:thetadegenerate}
K_{\theta\bm\beta,\theta\bm\alpha}^{B_n}(t) = K_{\bm\beta,\bm\alpha}^{B_n}(t).
\end{align}
Note that the condition $k=0$ is important here.

\subsection{Multiplicities}\label{subsec:mult}
We recall some notions regarding certain multiplicites from \cite{lusztig1986character}.
Let $N \in \N$, $G = \SO(N)$, and $(C,\E), (C',\mathcal E') \in \mathcal N_G$. 
Consider the intersection cohomology complex $\IC(\bar C',\E')$ on $\bar C'$ associated to $(C',\E')$. 
For each $m \in \N$, the restriction of $H^{2m}(\IC(\bar C',\mathcal E'))$ to $C$ is a 
direct sum of irreducible $G$-equivariant local systems on $C$.
Let $\mult(C,\mathcal E;C',\mathcal E')$ be the multiplicity of $\mathcal E$ in $\bigoplus_{m\in\Z}H^{2m}(\IC(\bar C',\mathcal E'))|_C$.
Let $\rho = \GSpr(C,\mathcal E)$ and $\rho' = \GSpr(C',\mathcal E')$. Suppose $(C,\mathcal E)$ and $(C',\mathcal E')$ are parametrised by $(\lambda,[\eps]), (\lambda',[\eps']) \in \PPort(N)$, respectively. If $k(\lambda,[\eps]) \neq k(\lambda',[\eps'])$, then $\mult(C,\mathcal E;C',\mathcal E') = 0$. If $k(\lambda,[\eps]) = k(\lambda',[\eps'])$, then $\rho = \GSpr(C,\mathcal E)$ and $\rho' = \GSpr(C',\mathcal E')$ are representations of the same relative Weyl group $W_{\text{rel}}$. In this case, let $\tilde P = \tilde P^{(k(\lambda,[\eps]))}$ be as in Theorem \ref{theorem:lusztig} for $W_{\text{rel}}$. For each $m \in \N$, the coefficient of $t^{m}$ in $\tilde p_{\rho',\rho}(t)$ is equal to the multiplicity of $\mathcal E$ in $H^{2m}(\IC(\bar C',\mathcal E'))|_C$ by \cite[(24.8.2)]{lusztig1986character} (recall that $\tilde P$ in \Cref{theorem:lusztig} is the transpose of the matrix $\Pi$ in \cite[(24.8.2)]{lusztig1986character}). 
Let $\mult(C,\mathcal E;C',\mathcal E')$ be the multiplicity of $\mathcal E$ in $\bigoplus_{m\in\Z}H^{2m}(\IC(\bar C',\mathcal E'))|_C$. Then we have $\mult(C,\mathcal E;C',\mathcal E') = \tilde p_{\rho',\rho}(1)$. 

We want to give a combinatorial description of $\mult(C,\mathcal E;C',\mathcal E')$.
Let $(\alpha,\beta) \in \mathcal P \times \mathcal P$ and let $(m_0,m_1,<)$ be any order on the set of indices $I = I_{m_0,m_1}$ of $(\alpha,\beta)$. Let $J = \{((i,e),(j,f)) \in I^2 \colon (i,e) < (j,f), e\neq f\}$. Let $X = \N^J$ and write its elements as $x = (x_{(i,e),(j,f)})_{((i,e),(j,f))\in J}$ with $x_{(i,e),(j,f)} \in \N$. For $x \in X$, we define $(\alpha[x],\beta[x])\in \Z^{m_0} \times \Z^{m_1}$ by
\begin{align}
\alpha[x]_i
&= \alpha_i + \sum_{j \in \{1,\dots,m_1\},(i,0)<(j,1)} x_{(i,0),(j,1)} - \sum_{j\in\{1,\dots,m_1\},(j,1)<(i,0)} x_{(j,1),(i,0)},
\\
\beta[x]_j
&= \beta_i + \sum_{i \in \{1,\dots,m_0\},(j,1)<(i,0)} x_{(j,1),(i,0)} - \sum_{i\in\{1,\dots,m_0\},(i,0)<(j,1)} x_{(i,0),(j,1)},
\end{align}
for $i=1,\dots,m_0$, $j=1,\dots,m_1$. For each $(\mu,\nu) \in \Z^{m_0} \times \Z^{m_1}$, let $X(\alpha,\beta,<;\mu,\nu) = \{x \in X \colon \alpha[x] = \mu, \beta[x] = \nu\}$.

\begin{remark}
\begin{enumerate}
\item For all $x \in X$, we have $S(\alpha[x]) + S(\beta[x]) = S(\alpha) + S(\beta)$. Hence $X(\alpha,\beta,<;\mu,\nu)$ is empty if $S(\mu) + S(\nu) \neq S(\alpha) + S(\beta)$.
\item For all $(\mu,\nu) \in \Z^{m_0} \times \Z^{m_1}$, $X(\alpha,\beta,<;\mu,\nu)$ is finite. 
\end{enumerate}
\end{remark}

For each $\mu \in \Z^{m_0}$ and $w \in S_{m_0}$, we define $\mu[w] \in \Z^{m_0}$ by $\mu[w]_i = \mu_{wi} + i -wi$ for $i = 1,\dots,m_0$. 
We similary define $\nu[v]$ for $\nu \in \Z^{m_1}$ and $v \in S_{m_1}$.
Let $(\mu,\nu) \in \mathcal P \times \mathcal P$. If $t(\mu) > m_0$ or $t(\nu) > m_1$, we define $\mult(\alpha,\beta,<;\mu,\nu) = 0$. Otherwise, we can consider $(\mu,\nu)$ as an element of $\Z^{m_0} \times \Z^{m_1}$ and we define
\begin{align}
\mult(\alpha,\beta,<;\mu,\nu)
:= \sum_{w \in S_{m_0}, v \in S_{m_1}} \sgn(w) \sgn(v) | X(\alpha,\beta,<;\mu[w],\nu[v])|.
\end{align}
It is shown on \cite[\S3.1, p. 412]{waldspurger} that $\mult(\alpha,\beta,<;\mu,\nu)$ is independent of the choice of representative for $<$.

Recall $p_{A,B;s}(\alpha,\beta,<)$ from \Cref{def:pABs}.
Let $Q_s(\alpha,\beta,<)$ be the set of pairs $(\mu,\nu) \in \mathcal P \times \mathcal P$ such that for all $A,B,s \in \R$ with $s>0$, we have
\begin{align}
\Lambda_{A,B;s}(\mu,\nu) \leq p_{A,B;s}(\alpha,\beta,<).
\end{align} 
We have the following result \cite[Prop. 3.1]{waldspurger}:

\begin{proposition}\label{prop:mult}
Let $(\mu,\nu) \in \mathcal P \times \mathcal P$ and $A,B,s \in \R$ with $s>0$.
\begin{enumerate}[ref={\theproposition(\arabic*)}]
\crefalias{enumi}{proposition}
\item Suppose that $\mult(\alpha,\beta,<;\mu,\nu) \neq 0$. Then $(\mu,\nu) \in Q_s(\alpha,\beta,<)$.
\item\label{prop:mult2} Suppose that $\Lambda_{A,B;s}(\mu,\nu) = p_{A,B;s}(\alpha,\beta,<)$. Then $\mult(\alpha,\beta,<;\mu,\nu) \neq 0$ if and only if $(\mu,\nu) \in P_{A,B;s}(\alpha,\beta,<)$.
\item \label{prop:mult3} Suppose that $(\mu,\nu) \in P_{A,B;s}(\alpha,\beta,<)$. Then $\mult(\alpha,\beta,<;\mu,\nu) = 1$.
\end{enumerate}
\end{proposition}

\begin{proposition}[cf. {\cite[Proposition 4.2]{waldspurger}}]\label{prop:multshoji}
Let $N \in \N$, let $(\lambda,[\eps]), (\lambda',[\eps']) \in \Port(N)$ such that $\lambda$ only has odd parts and such that $k := k(\lambda,[\eps]) = k(\lambda',[\eps'])$.
Let $(\alpha,\beta)_k = \Phi_N(\lambda,[\eps])$, $(\alpha',\beta')_k = \Phi_N(\lambda',[\eps'])$ and write $<$ for $<_{\alpha,\beta,k}$.
\begin{enumerate}[wide,ref={\theproposition(\arabic*)}]
\crefalias{enumi}{proposition}
\item\label{shojiB}
If $k > 0$, then
\begin{align}
\mult(C_\lambda,{\mathcal E}_{[\eps]};C_{\lambda'},\mathcal E_{[\eps']}) = \mult(\alpha,\beta,<;\alpha',\beta').
\end{align}
\item\label{shojiD} 
If $k = 0$ (note that $N$ is even), then
\begin{align}
\mult(C_\lambda,{\mathcal E}_{[\eps]};C_{\lambda'}^\pm,\mathcal E_{[\eps']}^\pm) 
&= 
\frac{c_{(\alpha',\beta')}}{2}(\mult(\alpha,\beta,<;\alpha',\beta') +  \mult(\alpha,\beta,<;\beta',\alpha'))
\\
&=
\begin{cases}
\mult(\alpha,\beta,<;\alpha',\beta') &\text{if } \alpha' = \beta', 
\\
\mult(\alpha,\beta,<;\alpha',\beta') +  \mult(\alpha,\beta,<;\beta',\alpha') &\text{if } \alpha' \neq \beta'. 
\end{cases}
\end{align}
\end{enumerate}
\end{proposition}

\subsection{Proof of Proposition \ref{prop:multshoji}}
We first prove some preliminary results.

Let $m = m_0 + m _1$ and let $e_1,\dots,e_m$ be the unit vectors in $\Z^m$. For $i \neq j$ define an operator $R_{ij}$ on $\Z^m$ by $R_{ij}\lambda = \lambda + e_i - e_j$. A \emph{raising operator} (resp. \emph{lowering operator}) $R$ on $\Z^m$ is a product of $R_{ij}$ with $i < j$ (resp. $i > j$). Given an identification of $\Z^m$ with $\Z^{m_0} \times \Z^{m_1}$, we write $R_{ij}$ as $R_{\bm\mu,\bm\nu}$ for $\bm\mu,\bm\nu \in \Z^{m_0} \times \Z^{m_1}$ correspond to $i,j \in \Z^m$, respectively.

We define $s_{(\mu,\nu)}$, $h_{(\mu,\nu)}$ and $q_{(\mu,\nu)}$ for $(\mu,\nu) \in \Z^{m_0} \times \Z^{m_1}$ as follows (so far, they were only defined for $(\mu,\nu) \in \cP \times \cP$). 
If any of the $\mu_i$ or $\nu_i$ is negative, we set $q_{(\mu,\nu)} = 0$ and $h_{(\mu,\nu)} = 0$. 
Otherwise, pick  $w \in S_{m_0}$, $v \in S_{m_1}$ such that $w(\mu), v(\nu) \in \mathcal P$ and we set  $q_{(\mu,\nu)} = q_{(w(\mu),v(\nu))}$ and $h_{(\mu,\nu)} =  h_{(w(\mu),v(\nu))}$. 
We set $s_{(\mu,\nu)} = 0$ if $\mu_i - i$ for $i=1,\dots,m_0$ are not all distinct, or if $\nu_i - i$ for $i=1,\dots,m_1$ are not all distinct. Otherwise, there exist $(\alpha,\beta) \in \cP \times \cP$ and $w \in S_{m_0}$, $v \in S_{m_1}$ such that $\mu = \alpha[w]$ and $\nu = \beta[v]$ are partitions and we set
\begin{align}
s_{(\mu,\nu)} = \sgn(w)\sgn(v) s_{(\alpha,\beta)}.
\end{align}
For a raising or lowering operator $R$ on $\Z^m$, we define
\begin{align}
R q_{(\mu,\nu)} = q_{R(\mu,\nu)},
\\
R h_{(\mu,\nu)} = h_{R(\mu,\nu)},
\\
R s_{(\mu,\nu)} = s_{R(\mu,\nu)}.
\end{align}
Note that there may exist raising or lowering operators $R, R'$ such that $R'q_{(\mu,\nu)} = 0$ but $(RR')q_{(\mu,\nu)} \neq 0$. The same problem can occur for $h_{(\mu,\nu)}$ and $s_{(\mu,\nu)}$.
\begin{remark}\label{rem:raising}
For any raising or lowering operator $R$, the matrix of the linear map $q_{(\alpha,\beta)} \mapsto Rq_{(\alpha,\beta)}$ in the basis $(q_{(\alpha,\beta)})$ is the same as the matrix of the linear map $h_{(\alpha,\beta)} \mapsto Rh_{(\alpha,\beta)}$ in the basis $(h_{(\alpha,\beta)})$, but it is not generally the same as the matrix of $s_{(\alpha,\beta)} \mapsto Rs_{(\alpha,\beta)}$ in the basis $(s_{(\alpha,\beta)})$.
\end{remark}

The following result follows from the proof of \cite[Corollary 6.8]{shoji2001green}. However, one of the steps in the proof that we present is different than what was done in \emph{loc. cit.}, where we use ideas from \cite[\S2]{garsia1992orthogonality}.

\begin{lemma}\label{lemma:multshoji}
Let $k \in \Z_{\geq 0}$, $n = \frac{N-k^2}{2}$. Let $\bm\alpha = (\alpha,\beta) \in \mathcal P_2(n)$ 
Let $m_0,m_1 \in \N$ such that $m_1 = m_0 - k \geq n$ and let $(m_0,m_1,<)$ be any order on the set of indices of $(\alpha,\beta)$ such that $\alpha_i + k + 2 -2i \geq \beta_j - k + 2 - 2j$ if $(i,0) < (j,1)$.
Suppose that $Q_{\bm\alpha,k} = R_{\bm\alpha,<}$ in $\Xi_{\Q}^n[t]$. 
Suppose $P = P^{(k)}$ is as in Theorem \ref{theorem:shoji} for $W_n := W(B_n)$.
Then
\begin{align}
p_{(\alpha',\beta'),(\alpha,\beta)}(1) = \mult(\alpha,\beta,<;\alpha',\beta').
\end{align}
\end{lemma}
\begin{proof}
Recall that the transition matrices between the bases of symmetric functions in \ref{sym1} - \ref{sym7} in the space $\Xi_{\Q}[t]$ are the same as in the space $\Xi_{\Q,(m_0,m_1)}^n[t]$, as $m_0 = m_1 + k$ and $m_0, m_1 \geq n$.
Thus we can consider the symmetric functions in \ref{sym1} - \ref{sym7} as elements of $\Xi_{\Q,(m_0,m_1)}^n[t]$.
Write $K(t)$ for $K^{B_n}(t) = M(s,P')$. 
By Theorem \ref{thm:kostka}, we have $K(t^{-1})T_k = P(t)$. 
By \cite[(5.6.1)]{shoji2001green}, we have
\begin{align}
M(Q,q) = (M(P',m)^t)^{-1} = K(t)^t (M(s,m)^t)^{-1}.
\end{align}
Let $I = I_{m_0,m_1}$ be the set of indices of $(\alpha,\beta)$ and let $J = \{((i,e),(j,f)) \in I^2 \colon (i,e) < (j,f), e\neq f\}$ as in \S\ref{subsec:mult} and define $J' = \{ ((i,e),(j,e)) \in I^2 \colon (i,e) < (j,e), e=0,1\}$. 
We can identify $I$ with $\{1,\dots,m_0+m_1\}$, respecting $<$ and the usual order on $\{1,\dots,m_0+m_1\}$. We identify $\Z^{m_0+m_1}$ with $\Z^{m_0} \times \Z^{m_1}$ accordingly, and we define raising and lowering operators on $\Z^{m_0} \times \Z^{m_1}$ as before.
Then 
\begin{align}\label{eq:sh}
s_{(\alpha,\beta)}(x) = \prod_{(a,b) \in J'} (1 - R_{a,b}) h_{(\alpha,\beta)}(x)
\end{align}
by \cite[(6.7.3)]{shoji2001green}.
Using $Q_{(\alpha,\beta)} = R_{(\alpha,\beta)}$ and \text{\cite[(3.7.1)]{shoji2001green}}, we find
\begin{align}\label{eq:Qq}
Q_{(\alpha,\beta)}(x,t)
&= 
\left(\prod_{(a,b) \in J'} (1 - R_{a,b}) \prod_{(a,b) \in J} (1 - tR_{a,b})^{-1}\right) q_{(\alpha,\beta)}(x,t).
\end{align}
Define a basis $(H_{(\alpha,\beta)})_{(\alpha,\beta) \in \mathcal P_2(n)}$ of $\Xi_{\Q,(m_0,m_1)}^n[t]$ by requiring that $M(H,h) = M(Q,q)$.
By Remark \ref{rem:raising}, we have
\begin{align}\label{eq:Hh}
H_{(\alpha,\beta)}(x)
&= 
\left(\prod_{(a,b) \in J'} (1 - R_{(i,e),(j,f)}) \prod_{(a,b) \in J} (1 - tR_{a,b})^{-1}\right) h_{(\alpha,\beta)}(x).
\end{align}
Note that $K(t)^t = M(Q,q)M(s,m)^t = M(H,h) M(s,h)^{-1} = M(H,s)$.
We want to show that $K_{(\alpha',\beta'),(\alpha,\beta)}$ is the coefficient of $s_{(\alpha',\beta')}$ in 
\begin{align}\label{eq:Hs}
H_{(\alpha,\beta)}(x)
=
\prod_{(a,b) \in J} (1 - tR_{a,b})^{-1} s_{(\alpha,\beta)}
=
\prod_{(a,b) \in J} (\sum_{\ell=0}^\infty (tR_{a,b})^\ell) s_{(\alpha,\beta)}.
\end{align} 
Note that \eqref{eq:Qq} involves raising operators acting on $q_{(\alpha,\beta)}$, whereas in \eqref{eq:Hs}, the raising operators act on $s_{(\alpha,\beta)}$, and we noted that the matrices of these actions are not generally the same in Remark \ref{rem:raising}.
We shall prove \eqref{eq:Hs} using ideas from \cite[\S2]{garsia1992orthogonality}, where a similar result is proved for classical Schur functions, avoiding the use of raising operators.

Let $\bm\alpha = (\alpha^{(0)},\alpha^{(1)}) \in \mathcal P_2(n)$. The complete symmetric functions $h_{(\alpha^{(0)},\alpha^{(1)})}$ 
are defined to be a product 
\begin{align}
h_{(\alpha^{(0)},\alpha^{(1)})}(x) = \prod_{d=0}^1 \prod_{i=1}^{m_d} h_{\alpha_i^{(d)}}^{(d)}(x^{(d)}),
\end{align}
where $h_{\alpha_i^{(d)}}^{(d)}(x^{(d)})$ are the ordinary complete symmetric functions in the variables $x^{(d)} = (x_1^{(d)},\dots,x_{m_d}^{(d)})$ of degree $a_i^{(d)}$.
For $d=0,1$, consider the generating function of $\{h_m^{(d)}(x)\}_{m \in \Z_{\geq0}}$:
\begin{align}
\Omega^{(d)}(x^{(d)},z) = \sum_{m = 0}^\infty h^{(d)}_m(x^{(d)})z^m.
\end{align}
For $d=0,1$ and $i=1,\dots,m_d$, introduce the variables $y_i^{(d)}$.
Then \eqref{eq:sh} is equivalent to the statement that $s_{\bm\alpha}(x) = s_{\alpha^{(0)}}(x^{(0)}) s_{\alpha^{(1)}}(x^{(1)})$ where for $d=0,1$, we have
\begin{align}\label{eq:schurgenerating}
s_{\alpha^{(d)}}(x^{(d)}) 
= 
\left.
\left( \prod_{i=1}^{m_d} \Omega^{(d)}(x^{(d)},y_i^{(d)}) \right)
\left( \prod_{1\leq a < b \leq m_d}  1 - \frac{y_b^{(d)}}{y_a^{(d)}} \right)
\right|_{\prod_{i=1}^{m_d} (y_i^{(d)})^{\alpha_i^{(d)}}},
\end{align}
where the $|_{\prod_{i=1}^{m_d} y_i^{(d)}}$ stands for taking the coefficient of the $(\prod_{i=1}^{m_d} y_i^{(d)})$ term. Let $\Delta_{m_d}(y^{(d)}) = \prod_{1\leq a < b \leq m_d} ( y_a^{(d)} - y_b^{(d)})$, $\delta_{m_d} = (m_d-1, m_d-2, \dots, 0) \in \mathcal P$, and $\Omega_d(x^{(d)},y^{(d)}) =  \prod_{i=1}^{m_d} \Omega^{(d)}(x^{(d)},y_i^{(d)})$. Then we can rewrite \eqref{eq:schurgenerating} to obtain
\begin{align}
s_{\bm\alpha}(x) 
&=
s_{\alpha^{(0)}}(x^{(0)})s_{\alpha^{(1)}}(x^{(1)})
\\
&= 
\left. \Omega_0(x^{(0)},y^{(0)}) \Delta_{m_0}(y^{(0)})  \right|_{(y^{(0)})^{\alpha^{(0)} + \delta_{m_0}}} \cdot
\left. \Omega_1(x^{(1)},y^{(1)}) \Delta_{m_1}(y^{(1)})  \right|_{(y^{(1)})^{\alpha^{(1)} + \delta_{m_1}}}.
\end{align}
For $d=0,1$, since $\Omega_k(x^{(d)},y^{(d)}) \Delta_{m_d}(y^{(d)})$ is an alternating function in $y^{(d)}$, taking the coefficient of $(y^{(d)})^{\alpha^{(d)} + \delta_{m_d}}$ is the same as taking the coefficient of 
\begin{align}\label{eq:defdelta}
\Delta_{\alpha^{(d)}}(y^{(d)}) := \sum_{\sigma \in S_{m_d}} \sgn(\sigma) (\sigma y^{(m_d)})^{\alpha^{(d)}+\delta_{m_d}}
\end{align}
in $\Omega_k(x,y) \Delta_{m_d}(y)$, so we see that
\begin{align}\label{eq:cauchy}
\Omega_0(x^{(0)},y^{(0)})\Omega_1(x^{(1)},y^{(1)}) 
&= 
\sum_{\bm \alpha \in \mathcal P \times \mathcal P}
s_{\alpha^{(0)}}(x^{(0)}) \frac{\Delta_{\alpha^{(0)}}(y^{(0)})}{\Delta_{m_0}(y^{(0)})}
s_{\alpha^{(1)}}(x^{(1)}) \frac{\Delta_{\alpha^{(1)}}(y^{(1)})}{\Delta_{m_1}(y^{(1)})}.
\end{align}
Note that \eqref{eq:Hh} is equivalent to
\begin{align}
H_{\bm \alpha}(x) 
&=
\left.
\left( \prod_{d=0}^1 \Omega_d(x^{(d)},y^{(d)}) \right)
\left( \prod_{((a,e),(b,e)) \in J'} 1 - \frac{y_b^{(e)}}{y_a^{(e)}}\right)
\left(\prod_{((a,e),(b,1-e)) \in J}  1 -t \frac{y_b^{(1-e)}}{y_a^{(e)}}\right)\right|_{y^{\bm\alpha}}
\end{align}
and by \eqref{eq:cauchy}, this becomes
\begin{align}
H_{\bm \alpha}(x) 
&=
\left.
\sum_{\bm \alpha \in \mathcal P \times \mathcal P}
s_{\bm\alpha}(x)
\left( \prod_{d=0}^1 \Delta_{\alpha^{(d)}}(y^{(d)}) \right)
\left( \prod_{((a,e),(b,1-e)) \in J} 1 -t \frac{y_b^{(1-e)}}{y_a^{(e)}}\right)\right|_{y^{\bm\alpha + (\delta_{m_0},\delta_{m_1})}}. 
\label{eq:Hs2}
\end{align}
By \eqref{eq:defdelta} and the definition of the Schur functions $s_{(\mu,\nu)}$ for $(\mu,\nu) \in \Z \times \Z$, we see that \eqref{eq:Hs2} is equivalent to \eqref{eq:Hs}, as desired. 

Thus \eqref{eq:Hs} gives a way to compute $M(H,s)= K(t)^t$ in $\Xi^n$, hence a way to compute $p_{(\alpha',\beta'),(\alpha,\beta)}(1) = K_{(\alpha',\beta'),(\alpha,\beta)}(1)$.
For each $w \in S_{m_0}$, $v \in S_{m_1}$, we want to count the number of ways we can create products $R$ of raising operators $R_{\bm j}$ with $\bm j \in J$ such that $(\alpha'[w],\beta'[v]) = R(\alpha,\beta)$, i.e. we want to determine the size of the set
\begin{align}
R(w,v) := \{(y_{\bm j})_{\bm j \in J} \in X \colon (\alpha'[w],\beta'[v]) = \prod_{\bm j \in J}R_{\bm j}^{y_{\bm j}}(\alpha,\beta)\}.
\end{align}
It holds that
\begin{align}
p_{(\alpha',\beta'),(\alpha,\beta)}(1)
& = \sum_{w \in S_{m_0}, v \in S_{m_1}} \sgn(w)\sgn(v) |R(w,v)|.
\end{align}
Note that $R(w,v) = X(\alpha,\beta,<;\alpha'[w],\beta'[v])$, so
\begin{align}
p_{(\alpha',\beta'),(\alpha,\beta)}(1)
= \mult(\alpha,\beta,<;\alpha',\beta'). 
&\qedhere
\end{align}
\end{proof}

The following lemma and \cite[Proposition 4.2]{waldspurger} are a generalisation of \cite[Prop. 6.2]{shoji2001green}, where it was shown that $Q_{\bm\alpha} = R_{\bm\alpha}$ for symbols of type $B$ and $C$, but with strictly stronger conditions on $\bm\alpha$ than the conditions in \Cref{lemma:QR} and \cite[Proposition 4.2]{waldspurger}.

\begin{lemma}\label{lemma:QR}
Let $k \in \Z_{\geq 0}$, $n = \frac{N-k^2}{2}$ and let $\bm\alpha = (\alpha,\beta) \in \mathcal P_2(n)$ such that $\bm\alpha \in H(n,k)$.
Let $m_0,m_1 \in \N$ such that $m_1 = m_0 - k \geq n$ and let $(m_0,m_1,<)$ be an order on the set of indices of $(\alpha,\beta)$ equivalent to $<_{\alpha,\beta,k}$.
Then $Q_{\bm\alpha,k} = R_{\bm\alpha,<_{\alpha,\beta,k}}$ in $\Xi_{\Q,(m_0,m_1)}^n[t]$.
\end{lemma}

\begin{proof}
We show that $R_{\bm\alpha}$ satisfies the conditions (a) and (b) in \cite[Theorem 4.4 (ii)]{shoji2001green} that uniquely characterise $Q_{\bm\alpha}$. It was shown that $R_{\bm\alpha}$ satisfies condition (a) in \cite[Proposition 3.12 (i)]{shoji2001green}. 
Let $\mathscr A$ be the subring of $\Q(t)$ consisting of functions that have no pole at $t=0$ and let $\mathscr A^*$ be the set of units of $\mathscr A$.
Condition (b) can be formulated as
\begin{enumerate}[(a)]
\setcounter{enumi}{1}
\item\label{Rcondb}
$R_{\bm\alpha}$ can be expressed as
\begin{align}
R_{\bm \alpha}(x,t) &= \sum_{\bm\alpha' \in \mathcal P_2(n)} u_{\bm\alpha,\bm\alpha'}(t)s_{\bm\alpha'}(x),
\end{align}
where $u_{\bm\alpha,\bm\alpha'}(t) \in \mathscr A$ such that $u_{\bm\alpha,\bm\alpha}(t) \in \mathscr A^*$, and such that $u_{\bm\alpha,\bm\alpha'}(t) = 0$ if $\bm\alpha \prec \bm\alpha'$ and $\bm\alpha \not\sim \bm\alpha'$.
\end{enumerate}

By \cite[Proposition 3.14]{shoji2001green}, $R_{\bm\alpha}$ can be expressed as a linear combination of Schur functions with coefficients in $\Z[t]$, i.e.
\begin{align}\label{eq:Rexpression}
R_{\bm\alpha}(x,t) = \sum_{(\alpha',\beta')\in\mathcal P_2(n)} r_{\bm\alpha}(\alpha',\beta')(t) s_{(\alpha',\beta')}(x), \quad  r_{\bm\alpha}(\alpha',\beta')(t) \in \Z[t].
\end{align}
We show that for any $\bm\alpha' = (\alpha',\beta') \in \mathcal P_2(n)$ such that $r_{\bm\alpha}(\alpha',\beta') \neq 0$, we have 
\begin{align}\label{eq:LambdaLambdap}
\Lambda_{k,-k;2}(\alpha',\beta') \leq \Lambda_{k,-k;2}(\alpha,\beta),
\end{align}
which implies that $a(\bm\alpha') \geq a(\bm\alpha)$ and hence $\bm\alpha' \preceq \bm\alpha$ or $\bm \alpha \sim \bm\alpha'$, which in turn implies that $R_{\bm\alpha}$ satisfies condition \ref{Rcondb} as above.
An expression for $R_{\bm\alpha}$ is given in \cite[(3.13.1)]{shoji2001green} in terms of raising operators. Using this, a more explicit expression for $r_{\bm\alpha}(\alpha',\beta')$ is given in \cite[\S4.3]{waldspurger} as follows. 
Recall that we assume that $m_0 = m_1 -k \geq n$.
Let $I = I_{m_0,m_1}$ be the set of indices of $(\alpha,\beta)$.
For $e=0,1$, let $(v_i^e)_{i=1,\dots,m_e}$ denote the canonical basis for $\Z^{m_e}$. 
Let $\nu_0 = \max((t(\alpha),0),(t(\beta),1))$. 
Let $J = \{v_i^e - v_j^f \in \Z^{m_0} \times \Z^{m_1} \colon ((i,e),(j,f)) \in I^2, e \neq f, (i,e) < (j,f), (i,e) \leq \nu_0\}$ and
let $K = \{v_i^e - v_j^f \in \Z^{m_0} \times \Z^{m_1} \colon ((i,e),(j,f)) \in I^2, e = f, (i,e) < (j,f), \nu_0 < (i,e)\}$. 
Let $\Delta_{m_e} = (m_e -1, m_e - 2, \dots, 0)$ for $e=1,2$ and consider $(\alpha + \Delta_{m_0},\beta + \Delta_{m_1}) \in \Z^{m_0} \times \Z^{m_1}$. Consider the usual action of the symmetric groups $S_{m_0}$ and $S_{m_1}$ on $\Z^{m_0}$ and $\Z^{m_1}$, respectiveley. For $e\in\{0,1\}$, $\sigma \in S_{m_e}$, and $x \in \Z^{m_e}$, write $x^\sigma = (x_{\sigma1},\dots,x_{\sigma m_e})$. For $(\alpha',\beta') \in \mathcal P_{m_0} \times \mathcal P_{m_1}$ and $d \in \Z_{\geq0}$, define
\begin{align}\label{eq:rd}
&r_{\bm\alpha}^d(\alpha',\beta') = \sum_{\sigma \in S_{m_0}, \tau \in S_{m_1}} (\sgn \sigma \sgn \tau)\cdot
\\
&|\{ X \subseteq J \cup K \colon |X| = d, (\alpha+\Delta_{m_0},\beta+\Delta_{m_1}) - \sum_{x \in X} x = ((\alpha'+\Delta_{m_0})^\sigma,(\beta'+\Delta_{m_1})^\tau)\}|.
\end{align}
Then 
\begin{align}\label{eq:r}
r_{\bm\alpha}(\alpha',\beta') = v_{\alpha,\beta}(t)^{-1} \sum_{d \in \Z_{\geq0}} (-1)^d r_{\bm\alpha}^d(\alpha',\beta')t^d,
\end{align}
for some non-zero $v_{\alpha,\beta}(t) \in \Q[t]$. We note that we can derive \eqref{eq:r} by keeping track of the raising operators appearing in \cite[(3.13.1)]{shoji2001green}. 
Note in particular that \cite[(3.13.1)]{shoji2001green} requires that $m_0,m_1 \geq n$, which we indeed assume.
Alternatively, we can write \eqref{eq:rd} as
\begin{align}
r_{\bm\alpha}^d(\alpha',\beta') = \sum_{\sigma \in S_{m_0}, \tau \in S_{m_1}} \sgn \sigma \sgn \tau \sum_{X_J \subseteq J \colon |X_J| \leq d } |\mathscr X_K (X_J,\sigma,\tau)|,
\end{align}
where
\begin{align}
\mathscr X_K(X_J,\sigma,\tau) = \{&X_K \subseteq K \colon |X_K| = d - |X_J|, 
\\
&(\alpha+\Delta_{m_0},\beta+\Delta_{m_1}) - \sum_{x \in X_K} x =  ((\alpha'+\Delta_{m_0})^\sigma,(\beta'+\Delta_{m_1})^\tau) + \sum_{x \in X_J} x\}.
\end{align}
Let $(\alpha',\beta') \in \cP_2(n)$ such that $r_{\bm\alpha}(\alpha',\beta') \neq 0$. Then there exist $X_J \subseteq J$, $\sigma \in S_{m_0}$ and $\tau \in S_{m_1}$ such that $\mathscr X_K(X_J,\sigma,\tau) \neq \varnothing$.
Similar to \cite[\S4.3 (11)]{waldspurger}, we can show that if $\mathscr X_K(X_J,\sigma,\tau)$ is non-empty, then we must have
\begin{align}\label{eq:rcondition}
(\alpha + \Delta_{m_0},\beta+\Delta_{m_1}) = ((\alpha'+\Delta_n)^\sigma,(\beta'+\Delta_m)^\tau)) + \sum_{x \in X_J} x.
\end{align}
The result \cite[\S4.3 (11)]{waldspurger} is a priori specific for $\Sp(2n)$, as the order on the indices of $(\alpha,\beta)$ comes from the ordering of the terms of the symbols of $\Sp(2n)$. However, the arguments used in \cite[\S4.3 (11)]{waldspurger} work in the $\SO(N)$ setting as well. In particular, one of the arguments requires the use of \cite[(3.13.1)]{shoji2001green}, and this result is applicable to our situation since $(m_0,m_1,<)$ satisfies the condition on \cite[\S3.8 p.666]{shoji2001green}. 
In this paper, we shall not include the proof of \eqref{eq:rcondition}. 
We shall mainly focus on the part of the proof of this proposition analogous to \cite[\S4.4]{waldspurger}, as this part has some sublte differences compared to the proof in \emph{loc. cit.}. 

Let $r = m_0 + m_1$. We can view $\R^r$ as a root space of type $A_{r-1}$ and we shall use the theory of the root system $A_{r-1}$ to prove that $\Lambda(\alpha',\beta') \leq \Lambda(\alpha,\beta)$.
Recall that $I$ is the set of indices of $(\alpha,\beta)$.
Similar as before, we identify $I$ with $\{1,\dots,r\}$, respecting the order $(m_0,m_1,<)$. This identification defines an isomorphism $\iota \colon \R^{m_0} \times \R^{m_1} \to \R^r$. For $h \in \N$, let $\mathcal P_h$ be the subset of $\mathcal P$ consisting of all the partitions of length at most $h$. Thus $\mathcal P_{m_0} \times \mathcal P_{m_1}$ is also identified with a subset of $\R^r$ via $\iota$. For $i = 1,\dots,m_0$ and $j=1,\dots,m_1$, let $x_i$ and $y_j$ be the elements of $\{1,\dots,r'\}$ that are identified with the elements $(i,0)$ and $(j,1)$ in $I$, respectively. Let $I_a = \{x_1,\dots,x_{m_0}\}$, $I_b = \{y_1,\dots,y_{m_1}\}$, and let $\eps_1,\dots,\eps_r$ be the unit vectors in $\R^r$. Then $\Sigma = \{\eps_i - \eps_j\}_{i\neq j}$ in $\R^r$ forms the root system $A_{r-1}$ with positive roots $\Sigma^+ = \{\eps_i - \eps_j\}_{i<j}$. Denote by $W = S_r$ the Weyl group of $\Sigma$. Let $\Sigma^M = \{\eps_i - \eps_j \colon i\neq j; i,j \in I_a \text{ or } i,j \in I_b\}$ and $\Sigma^{M,+} = \Sigma^M \cap \Sigma^+$. For $e \in \N$, let $\delta_e = \{(e-1)/2, (e-3)/2, \dots, (1-e)/2\}$. Then $\delta := \delta_r$ is the half-sum of the positive roots and $\delta^M = \iota(\delta_{m_0},\delta_{m_1})$ is the half-sum of the roots in $\Sigma^{M,+}$. Let $\mathcal J = \iota(J) \subseteq \Sigma^+ \setminus \Sigma^{M,+}$. 
Let $\bar C^+ = \{x \in \R^r \colon x_1 \geq x_2 \geq \dots\}$ be the fundamental Weyl chamber. For all $x \in \R^r$, there exists a $w \in W$ such that $x^+ := wx \in \bar C^+$.
Define $z \in \R^r$ by
\begin{align}
z_i =
\begin{cases}
k+1 - m_0 &\text{if } i \in I_a,
\\
-k+1 - m_1 &\text{if } i \in I_b.
\end{cases}
\end{align}
Let $(A,B) = (A_{\alpha,\beta,k},B_{\alpha,\beta,k})$, $(A',B') = (A_{\alpha',\beta',k},B_{\alpha',\beta',k}) \in \cR \times \cR$.
Let $\bm\Lambda = ((A_1,\dots,A_{m_0}),(B_1,\dots,B_{m_1}))$, $\bm\Lambda' = ((A_1',\dots,A_{m_0}'),(B_1',\dots,B_{m_1}')) \in \Z^{m_0}\times\Z^{m_1}$. 
Let $\Lambda = \iota(\bm\Lambda)$ and $\underline\Lambda = \iota(\bm\Lambda')$.
Then $\Lambda = \Lambda^+$ by definition of $\iota$.
Note that $\underline\Lambda^+$ and $\Lambda$ each form the largest $r$ terms of $\Lambda_{k,-k;2}(\alpha',\beta')$ and $\Lambda_{k,-k;2}(\alpha,\beta)$ respectively, so it suffices to show that $\underline\Lambda^+ \leq \Lambda$, since we then immediately get \eqref{eq:LambdaLambdap}.

We first describe $\underline\Lambda$.
Note that $(\alpha',\beta')$ can be considered as an element of $\mathcal P_{m_0} \times \mathcal P_{m_1}$, and so we can consider $\iota(\alpha',\beta')$. 
Let $X_J \subseteq J$, $\sigma \in S_{m_0}$ and $\tau \in S_{m_1}$ such that \eqref{eq:rcondition} holds. Since $\sigma \times \tau$ fixes $(\Delta_{m_0} - \delta_{m_0}, \Delta_{m_1} - \delta_{m_1})$, we can replace $\Delta_{m_0}$ and $\Delta_{m_1}$ by $\delta_{m_0}$ and $\delta_{m_1}$ in \eqref{eq:rcondition}, respectively. 
Note that $W^M$ is identified with $S_{m_0} \times S_{m_1}$ via $\iota$.
Let $w^M$ be the element of $W^M$ corresponding to $(\sigma \times \tau)^{-1} \in S_{m_0} \times S_{m_1}$ and let $X = \iota(X_J)$.
Applying $\iota$ to both sides of \eqref{eq:rcondition} and rearranging, we find
\begin{align}
\iota(\alpha',\beta') + \delta^M = w^M(\iota(\alpha,\beta) + \delta^M - \sum_{x \in X} x).
\end{align}
Note that $\underline\Lambda = \iota(\alpha',\beta') + 2\delta^M + z$ and $w^M z = z$, so we have
\begin{align}
\underline\Lambda 
&=
w^M(\iota(\alpha,\beta) + \delta^M - \sum_{x \in X}) + \delta^M + z 
\\
&= w^M(\Lambda - \delta^M - z - \sum_{x \in X} x) + \delta^M + z
\\
&= w^M(\Lambda - \delta^M - \sum_{x \in X} x) + \delta^M.\label{eq:underlineLambda}
\end{align}

Next, we want to prove \eqref{eq:Lambdadelta}, which is a `substitute'  for a certain property in \cite[Proposition 4.2]{waldspurger} that does not hold in our setting. 
In the $\Sp(2n)$ setting of \cite[Proposition 4.2]{waldspurger}, it holds that $\Lambda-\delta \in C^+$, and this is a property that was used in the proof of the proposition.
In our setting, we generally do not have $\Lambda - \delta \in \bar C^+$. 
In fact, let $m_0' \in \N$ be the smallest integer such that $m_0' \geq t(\alpha)$ and $m_1' := m_0' - k \geq t(\beta)$. Let $\bar \Lambda = \Lambda - \delta$ and $r' = m_0' + m_1' = t(\lambda)$. 
Since $(\alpha,\beta) \in H(n,k)$, we have $\Lambda_1 > \Lambda_2 > \dots > \Lambda_{r'}$. 
Thus if $m_0 = m_0'$, and hence $m_1 = m_1'$, we have $\bar\Lambda  \in \bar C^+$. However, we very well may have $m_0 > m_0'$, and thus $m_1 > m_1'$, in which case we have $\Lambda_{r'+1} = \Lambda_{r'+2}$, hence $\bar \Lambda_{r'+1} < \bar \Lambda_{r'+2}$ and so $\bar \Lambda \notin \bar C^+$.

Let $d_1 = m_0 - m_0' = m_1 - m_1' = \frac{1}{2}(r - r')$.
Then
\begin{align}\label{eq:largeterms}
(\bar \Lambda_{r' + 1}, \bar \Lambda_{r' + 2}, \dots, \bar \Lambda_{r}) = (x,x+1,x,x+1,\dots,x,x+1)
\end{align}
where $x = \bar \Lambda_{r' + 1}$. 
Let
\begin{align}\label{eq:barX}
\bar X = \{ \eps_{\bar r + 2i-1} - \eps_{r - 2i+2} \colon i = 1,\dots, \lceil d_1/2\rceil\} \subseteq \Sigma^+.
\end{align}
For $i > \bar m_0$ and $j > \bar m_1$, we may assume that $(i,0) < (j,1) < (i + 1,0) < \dots$, since the equivalence class of the order $<$ on the set of indices of $(\alpha,\beta)$ only depends on the ordering of the indices smaller than $\nu_0 = \max((t(\alpha),0),(t(\beta),1))$.
Thus we have $\bar X \subseteq \Sigma^+ \setminus \Sigma^{M,+}$.
Furthermore, note that $\Lambda_{r'} > \Lambda_{r'+1} + 1$ since $\lambda_r' > 0 = \lambda_{r'+1}$, so $\bar \Lambda_{r'} \geq \bar\Lambda_{r'+1} + 1 = x + 1$.
Thus we have
\begin{align}\label{eq:Lambdadelta}
(\Lambda - \delta) + \sum_{x \in \bar X} x  = (\bar \Lambda_1 \geq \dots \geq \bar \Lambda_{r'} \geq x+1 = \dots = x+1 \geq x = \dots = x) \in \bar C^+,
\end{align}
where the $x+1$ terms appear $d_1$ times and the $x$ terms also appear $d_1$ times. 

Next, we will further describe $\underline \Lambda$.
Note that
\begin{align}
\delta - \delta^M 
=
\frac{1}{2} \sum_{x \in \Sigma^+ \setminus \Sigma^{M,+}} x
=
\frac{1}{2} \left( \sum_{x \in X} x + \sum_{x \in \Sigma^+ \setminus (\Sigma^{M,+} \cup X)} x \right).
\end{align}
Recall that $\bar X \subseteq \Sigma^+ \setminus \Sigma^{M,+}$ and let  $X' = X \cup \bar X \subseteq \Sigma^+ \setminus \Sigma^{M,+}$. Note that $X \cap \bar X = \varnothing$, since for $\eps_i - \eps_j \in \bar X$, we have $\max\{x_{t(\alpha)},y_{t(\beta)}\} \leq r' < i < j$, whereas for $\eps_i - \eps_j \in X$, we have $i < \max\{x_{t(\alpha)},y_{t(\beta)}\}$.
Thus we have
\begin{align}\label{eq:mu}
\mu
:=
\delta^M - \delta + \sum_{x \in X} x + \sum_{x \in \bar X} x
&=
\frac{1}{2} \left( \sum_{x \in X} x - \sum_{x \in \Sigma^+ \setminus (\Sigma^{M,+} \cup X)} x \right) + \sum_{x \in \bar X} x
\\
&=
\frac{1}{2} \left( \sum_{x \in X'} x - \sum_{x \in \Sigma^+ \setminus (\Sigma^{M,+} \cup X')} x \right).
\end{align}
Let $\mu' = -\mu + (w^M)^{-1}(\delta^M)$.
Now \eqref{eq:underlineLambda} and \eqref{eq:mu} together give
\begin{align}
\underline \Lambda 
= 
w^M(\Lambda - \delta - \mu + \sum_{x\in\bar X} x) + \delta^M 
= 
w^M(\Lambda - \delta + \mu' + \sum_{x \in \bar X} x).
\end{align}
There exists a $w \in W$ such that 
\begin{align}
\underline \Lambda^+ 
&= 
w(w^M)^{-1}\underline \Lambda
\\
&=
w(\Lambda - \delta + \sum_{x \in \bar X} x + \mu')
\\
&=
w(\Lambda - \delta + \sum_{x \in \bar X} x) + w(\mu') + (\Lambda - \delta + \sum_{x \in \bar X} x) - (\Lambda - \delta + \sum_{x \in \bar X} x)
\\
&=
\Lambda 
+ w(\Lambda - \delta + \sum_{x \in \bar X} x) - (\Lambda - \delta + \sum_{x \in \bar X} x)
+ w(\mu') - \delta + \sum_{x \in \bar X} x.
\end{align}
We will now show that $\underline\Lambda^+ \leq \Lambda$.
Since $(\Lambda - \delta + \sum_{x \in \bar X} x) \in \bar C^+$ by \eqref{eq:Lambdadelta}, we have
\begin{align}\label{eq:wminus}
w(\Lambda - \delta + \sum_{x \in \bar X} x) - (\Lambda - \delta + \sum_{x \in \bar X} x) \in -{}^+\bar C.
\end{align}
Next, we show that $w(\mu') - \delta \in -{}^+\bar C$. Let 
\begin{align}
E &:= \left\{\frac{1}{2}(\sum_{y \in Y} y - \sum_{y \in \Sigma^+ \setminus Y} y) \in \R^r \colon Y \subseteq \Sigma^+ \right\},
\\
E^M &:= \left\{\frac{1}{2}(\sum_{y \in Y} y - \sum_{y \in \Sigma^{M,+} \setminus Y} y) \in \R^r \colon Y \subseteq \Sigma^{M,+}\right\}.
\end{align}
Note that $E$ and $E^M$ are fixed by the action of $W$ and $W^M$ on $\R^r$, respectively.
For each $H^M \in E^M$, it is easy to see that $-\mu + H^M \in E$. Thus since $\delta^M \in E^M$, we have $\mu' \in E$, and so $w(\mu') \in E$. For any $H \in E$, we have $H - \delta \in -{}^+ \bar C$, so in particular, we have 
\begin{align}\label{eq:wminusmu}
w(\mu') - \delta \in -{}^+ \bar C.
\end{align}
We now have
\begin{align}
\underline \Lambda - \sum_{x \in \bar X} x \in \Lambda - {}^+ \bar C \quad \text{i.e.} \quad \underline \Lambda - \sum_{x \in \bar X} x \leq \Lambda.
\end{align}
As noted before, if $\eps_i - \eps_j \in \bar X$, then $\iota(\nu_0) < r' < i < j$. Thus for $\ell = 1,\dots, r'$, the $\ell^{\text{th}}$ term of $\underline \Lambda - \sum_{x \in \bar X}$ is equal to $\underline\Lambda_\ell$, so
\begin{align}
\sum_{i=1}^\ell \underline\Lambda_\ell \leq \sum_{i=1}^\ell \Lambda_\ell.
\end{align}
Now suppose $\ell \in \{r' + 1,\dots, r\}$. 
Let $\kappa = [k,-\infty[_2 \sqcup [-k,-\infty[_2$.
As noted in Remark \ref{rem:ineq}, the first $r' = m_0' + m_1'$ terms of $\Lambda$ are the first $m_0'$ terms of $\alpha + [k,-\infty[_2$ and the first $m_1'$ terms of $\beta + [-k,-\infty[_2$. Thus
\begin{align}
\sum_{i=1}^\ell \Lambda_i 
&= 
\sum_{i=1}^{r'} \Lambda_i + \sum_{i = \bar r +1}^\ell \Lambda_i
\\
&= 
\sum_{i=1}^{m_0'} (\alpha_i + k + 2 - 2i)
+ \sum_{i=1}^{m_1'} (\beta_i - k + 2 - 2i)
+ \sum_{i=r' + 1}^\ell \kappa_i
\\
&= 
\sum_{i=1}^{m_0'} \alpha_i
+
\sum_{i=1}^{m_1'} \beta_i
+ 
\sum_{i=1}^\ell \kappa_i
\\
&=
n + \sum_{i= 1}^\ell \kappa_i,
\end{align}
where the last equality follows since $m_0' \geq t(\alpha)$ and $m_1' \geq t(\beta)$.
The first $\ell$ terms of $\underline \Lambda^+$ are the first $\ell$ terms of $\Lambda_{k,-k;2}(\alpha',\beta') = (\alpha + [k,-\infty[_2) \sqcup (\beta + [-k,-\infty[_2)$. Thus we have
\begin{align}\label{eq:lambdapluslambda}
\sum_{i=1}^\ell \underline \Lambda^+_i 
\leq
\sum_{i=1}^\ell (\alpha \sqcup \beta)_i + \sum_{i=1}^\ell \kappa_i
\leq 
n + \sum_{i=1}^\ell \kappa_i
=
\sum_{i=1}^\ell \Lambda_i.
\end{align}
To conclude, we have now shown that $\sum_{i=1}^\ell \underline \Lambda_i^+ \leq \sum_{i=1}^\ell \Lambda_i$ for $\ell = 1,\dots,r$, i.e. $\underline \Lambda^+ \leq \Lambda$, as desired.
\end{proof}

\begin{proof}[Proof of Proposition {\ref{prop:multshoji}}]
Since $\lambda$ only has odd parts, we have $(\alpha,\beta) \in H((N-k^2)/2,k)$ by Remark \ref{rem:symbolodd}, and so by Lemma \ref{lemma:QR}, we have $Q_{(\alpha,\beta),k} = R_{(\alpha,\beta),<}$. Thus \Cref{shojiB} follows from Lemma \ref{lemma:multshoji}. 

Suppose that $k=0$. Then $N = 2n$ is even.
Let $P^B$ and $P^D$ be the matrix of Green functions as in Theorem \ref{theorem:shoji} for $k = 0$ and for $W(B_n)$ and $W(D_n)$, respectively.
Note that the Green functions evaluated in $t = 1$ are the same as the Kostka polynomials evaluated in $t=1$.
Also note that $c_{(\alpha,\beta)} = 2$ since $\lambda$ only has odd parts and hence non-degenerate. So for $i \in \{1,2/c_{(\alpha',\beta')}\}$, \Cref{prop:shojigreenD1} gives
\begin{align}
p_{(\alpha',\beta')^i,(\alpha,\beta)}^D(1) 
= 
\frac{c_{(\alpha',\beta')}}{2} (p_{(\alpha',\beta'),(\alpha,\beta)}^B(1) + p_{(\alpha',\beta'),(\beta,\alpha)}^B(1)).
\end{align}
Since $\lambda$ is non-degenerate, we have $\alpha \neq \beta$.
By \eqref{eq:thetadegenerate}, we have $p_{(\alpha',\beta'),(\beta,\alpha)}^B(1) = p_{(\beta',\alpha'),(\alpha,\beta)}^B(1)$, so by \Cref{prop:shojigreenD1} and \Cref{lemma:multshoji}, we have
\begin{align}
\mult(C_\lambda,\mathcal E_{[\eps]};C_{\lambda'}^\pm,\mathcal E_{[\eps']}^\pm) 
=
p_{(\alpha',\beta')^i,(\alpha,\beta)}^D(1)
&=
\frac{c_{(\alpha',\beta')}}{2} (p_{(\alpha',\beta'),(\alpha,\beta)}^B(1) + p_{(\beta',\alpha'),(\alpha,\beta)}^B(1))
\\
&=
\frac{c_{(\alpha',\beta')}}{2} (\mult(\alpha,\beta,<;\alpha',\beta') + \mult(\alpha,\beta,<;\beta',\alpha')). \qedhere
\end{align}
\end{proof}

\section{Maximality and minimality theorems}\label{sec:max}

\subsection{Maximality theorem}
\begin{lemma}\label{lemma:order}
Let $(\lambda,[\eps]),(\lambda',[\eps']) \in \PPort(2N)$.
Then $\lambda \leq \lambda'$ if and only if $p_{\lambda,[\eps]} \leq p_{\lambda',[\eps']}$.
\end{lemma}
\begin{proof}
Let $c \in \N$. We have
\begin{align}
S_c(\lambda)
=
S_c(\lambda + [0,-\infty[_1) - S_c([0,-\infty[_1).
\end{align}
Recall $z$ and $z'$ from \S\ref{sec:gsc}.
Suppose $c$ is even. If $S_c(\lambda)$ is odd, then the largest $c$ terms of $\lambda + [0,-\infty[_1$ are $z_1,\dots,z_{c/2-1}$ and $z_1',\dots,z_{c/2+1}'$, and we have that $\lambda_{c} = \lambda_{c+1}$ is even and $z_{c/2+1}'=z_{c/2}$. If $S_c(\lambda)$ is even, then $z_1,\dots,z_{c/2}$ and $z_1',\dots,z_{c/2}'$. 
Suppose $c$ is odd. If $S_c(\lambda)$ is even, then the largest $c$ terms of $\lambda + [0,-\infty[_1$ are $z_1,\dots,z_{(c+1)/2}$ and $z_1',\dots,z_{(c-1)/2}'$, and we have that $\lambda_{c} = \lambda_{c+1}$ is even and $z_{(c+1)/2}'=z_{(c+1)/2}$. If $S_c(\lambda)$ is odd, then $z_1,\dots,z_{(c-1)/2}$ and $z_1',\dots,z_{(c+1)/2}'$. 
Let $c^+ = \lceil c/2 \rceil$, $c^- = \lfloor c/2 \rfloor$ and $\delta_c(\lambda) = 1$ if $S_c(\lambda) + c$ is odd and $\delta_c(\lambda) = 0$ otherwise. Then
\begin{align}
S_c(\lambda + [0,-\infty[_1) = 2S_{c^+}(z') + 2S_{c^-}(z) - c^+ + \delta_c(\lambda).
\end{align}
Recall $A^\#$ and $B^\#$ from \S\ref{sec:gsc}. By definition, we have
\begin{align}
S_{c^+}(z') &= S_{c^+}(A^\#) - S_{c^+}([0,-\infty[_1),
\\
S_{c^-}(z) &= S_{c^-}(B^\#) - S_{c^-}([0,-\infty[_1),
\end{align}
and since $A_1^\# \geq B_1^\# \geq A_2^\# \geq B_2^\# \geq \dots$, we have
\begin{align}
S_{c^+}(A^\#) + S_{c^-}(B^\#) = S_c(A^\# \sqcup B^\#) = S_c(p_{\lambda,\eps}).
\end{align}
Hence we have 
\begin{align}\label{eq:Sc}
S_c(\lambda) = 2S_c(p_{\lambda,\eps}) + \delta_c(\lambda) + C_c,
\end{align}
for some $C_c \in \Z$ independent of $\lambda$.

Suppose $\lambda < \lambda'$. Then for all $c \in \N$, we have $S_c(\lambda) \leq S_c(\lambda')$, and so by \eqref{eq:Sc}, we have
\begin{align}
S_c(p_{\lambda,\eps}) \leq S_c(p_{\lambda',\eps'}) + \frac{\delta_c(\lambda') - \delta_c(\lambda)}{2} \leq S_c(p_{\lambda',\eps'}) + \frac{1}{2},
\end{align}
hence $S_c(p_{\lambda,\eps}) \leq S_c(p_{\lambda',\eps'})$ since both are integers. Thus $p_{\lambda,[\eps]} \leq p_{\lambda',[\eps']}$. 

Conversely, suppose that $p_{\lambda,[\eps]} \leq p_{\lambda',[\eps']}$. 
Then for all $c \in \N$, we have $S_c(p_{\lambda,\eps}) \leq S_c(p_{\lambda',\eps'})$ and hence by \eqref{eq:Sc}, we have
\begin{align}\label{eq:ineqSc}
S_c(\lambda) \leq S_c(\lambda') + \delta_c(\lambda).
\end{align}
If $c = 1$, then obviously $S_c(\lambda) \leq S_c(\lambda')$, so assume $c > 1$.
If $c + S_c(\lambda)$ is even, then $S_c(\lambda) \leq S_c(\lambda')$.
Suppose $c + S_c(\lambda)$ is odd and suppose $S_c(\lambda) = S_c(\lambda')+1$. Then $\lambda_c = \lambda_{c+1}$ is even. Thus $c-1 + S_{c-1}(\lambda)$ is even and so $S_{c-1}(\lambda) \leq S_{c-1}(\lambda')$ by \eqref{eq:ineqSc}. Thus $\lambda_c \geq \lambda_c'+1$ and so 
\begin{align}
\lambda_{c+1} = \lambda_c \geq \lambda_c' + 1 \geq \lambda_{c+1}' + 1.
\end{align}
Now we have
\begin{align}\label{eq:Sc}
S_{c+1}(\lambda) 
=
S_c(\lambda) + \lambda_{c+1}
\geq
(S_c(\lambda') + 1) + (\lambda_{c+1}' + 1)
>
S_{c+1}(\lambda').
\end{align}
But $c+1 + S_{c+1}(\lambda) = (c+S_c(\lambda) + (1+\lambda_{c+1})$ is even since $\lambda_{c+1} = \lambda_c$ is even, and so $S_{c+1}(\lambda) \leq S_{c+1}(\lambda')$, which contradicts \eqref{eq:Sc}. 
Thus $S_C(\lambda) \leq S_c(\lambda')$. We conclude that $\lambda \leq \lambda'$.
\end{proof}

Recall that for $(\lambda,[\eps]),(\lambda',[\eps']) \in \PPort(N)$ with $\lambda$ non-degenerate, it follows from \Cref{prop:multshoji} that $\mult(C_\lambda,\E_{[\eps]};C_{\lambda'}^+,\mathcal E_{[\eps']}^+) = \mult(C_\lambda,\E_{[\eps]};C_{\lambda'}^-,\mathcal E_{[\eps']}^-)$. 

\begin{theorem}\label{thm:max}
Suppose $(\lambda,[\eps]) \in \PPort(N)$ such that $\lambda$ only has odd parts. Then there exists a unique $(\lambda^{\text{max}},[\epsmax]) \in \PPort(N)$ such that $\lambdamax$ is non-degenerate and 
\begin{enumerate}[(1)]
\item\label{thm:max1} $\mult(C_\lambda,\mathcal E_{[\eps]}; C_{\lambdamax},\mathcal E_{[\epsmax]}) = 1$,
\item\label{thm:max2} For all $(\lambda',[\eps']) \in \PPort(N)$ with $\mult(C_\lambda,\E_{[\eps]};C_{\lambda'}^+,\mathcal E_{[\eps']}^+)= \mult(C_\lambda,\E_{[\eps]};C_{\lambda'}^-,\mathcal E_{[\eps']}^-) \neq 0$, we have $\lambda' < \lambda^{\text{max}}$ or $(\lambda',[\eps']) = (\lambda^{\text{max}},[\epsmax])$.
\end{enumerate}
\end{theorem}

\begin{proof}
Let $k = k(\lambda,[\eps])$ 
and $(\alpha,\beta)_k = \Phi_N(\lambda,[\eps])$.
We have $(\alpha,\beta) \in H((N-k^2)/2,k)$ by \Cref{rem:symbolodd}.
Denote $<_{\alpha,\beta,k}$ by $<$.
By \Cref{prop:mult} and \Cref{lemma:unique}, there exists a unique $(\alphamax,\betamax) \in \mathcal P_2((N-k^2)/2)$ such that 
\begin{enumerate}[(i)]
\item\label{maxmult1} $\mult(\alpha,\beta,<;\alphamax,\betamax) = 1$,
\item\label{maxmult2} For all $(\alpha',\beta') \in \mathcal P \times \mathcal P$ with $\mult(\alpha,\beta,<;\alpha',\beta') \neq 0$, we have $\Lambda_{k,-k;2}(\alpha',\beta') < \Lambda_{k,-k;2}(\alphamax,\betamax)$ or $(\alpha',\beta')_k = (\alphamax,\betamax)_k$ (note that is an equality of ordered (resp. unordered) pairs if $k>0$ (resp. $k = 0$)).
\end{enumerate}
Furthermore, $(\alphamax,\betamax)$ is the unique element of $P_{k,-k;2}(\alpha,\beta,<)$.
By \Cref{prop:symbolterm}, the largest two terms of $\Lambda_{k,-k;2}(\alphamax,\betamax)$ are distinct, so $\alphamax \neq \betamax$. 

Suppose $k > 0$. Then \Cref{shojiB} and \Cref{lemma:order} show that $\Phi_N^{-1}(\alphamax,\betamax)$ satisfies \ref{thm:max1} and \ref{thm:max2}. 

Suppose $k = 0$ and let $(\bar\lambda,\bar\eps) = \Phi_N^{-1}(\alphamax,\betamax)$.
By \Cref{shojiD} and \ref{maxmult1}, we have
\begin{align}
\mult(C_\lambda,\E_{[\eps]};C_{\bar\lambda},\E_{[\bar\eps]})
&=
\mult(\alpha,\beta,<;\alphamax,\betamax) + \mult(\alpha,\beta,<;\betamax,\alphamax)
\\
&= 
1 + \mult(\alpha,\beta,<;\betamax,\alphamax).
\end{align}
Suppose that $\mult(\alpha,\beta,<;\betamax,\alphamax) \neq 0$.
Since $\alphamax \neq \betamax$, we have $(\betamax,\alphamax) \notin P_{0,0;-2}(\alpha,\beta,<)$, but since $\Lambda_{0,0;2}(\betamax,\alphamax) = \Lambda_{0,0;2}(\alphamax,\betamax)$, we have $\mult(\alpha,\beta,<;\betamax,\alphamax) = 0$ by \Cref{prop:mult2}. Hence  $\mult(C_\lambda,\E_{[\eps]};C_{\lambdamax},\E_{[\epsmax]}) = 1$.
Next, let $(\lambda',[\eps']) \in \PPort(N)$ such that $\mult(C_\lambda,\E_{[\eps]};C_{\lambda'}^\pm,\E_{[\eps']}^\pm) \neq 0$ and let $\{\alpha',\beta'\} = \Phi_N(\lambda',[\eps'])$.
Since $\mult(C_\lambda,\E_{[\eps]};C_{\lambda'}^\pm,\E_{[\eps']}^\pm) \neq 0$, we have either $\mult(\alpha,\beta,<;\alpha',\beta') \neq 0$ or $\mult(\alpha,\beta,<;\beta',\alpha') \neq 0$ by \Cref{shojiD}, and so
$\Lambda_{0,0;2}(\alpha',\beta') = \Lambda_{0,0;2}(\beta',\alpha') < \Lambda_{0,-0;2}(\alphamax,\betamax)$ or $\{\alpha',\beta'\} = \{\alphamax,\betamax\}$.
Thus $(\bar\lambda,\bar\eps)$ satisfies \ref{thm:max2} by \Cref{lemma:order}.

Uniqueness follows easily from the fact that $P_{k,-k;2}(\alpha,\beta,<)$ has a unique element.
Hence we conclude that $(\lambdamax,[\epsmax]) = \Phi_N^{-1}(\alphamax,\betamax)$ is the unique element of $\PPort(N)$ that satisfies \ref{thm:max1} and \ref{thm:max2}, and that $\lambdamax$ is non-degenerate, since $\alphamax \neq \betamax$.
\end{proof}

\begin{remark}
We only showed that $\lambdamax$ is non-degenerate, but we shall later see from a corollary of \Cref{thm:algorithm} that $\lambdamax$ in fact only has odd parts.
\end{remark}

\subsection{Minimality theorem}\label{subsec:min}
We state a result from \cite[\S4.6]{waldspurger} without proof. For $\mu \in \mathcal P$, write $\transp\mu$ for the transpose of $\mu$ and let $R_\mu = \mu + [0,-\infty[_1$.
\begin{lemma}\label{lemma:mult}
For all $\mu \in \mathcal P$ and $x \in \Z$, we have
\begin{align}
\mult_{R_{\mu}}(x) + \mult_{R_{\transp{\mu}}}(1-x) = 1.
\end{align}
\end{lemma}

For $\mu \in \mathcal R$, let $\mu^2 = \mu \sqcup \mu$. 
Let $n \in \N$ and let $(\alpha,\beta) \in \mathcal P_2(n)$. As representations of $W(B_n)$, we have $\sgn\otimes\rho_{(\alpha,\beta)} = \rho_{(\transp\beta,\transp\alpha)}$, where $\sgn$ is the sign representation of $W(B_n)$. 
Let $i\in\{1,2/c_{(\alpha,\beta)}\}$ and note that $c_{(\alpha,\beta)} = c_{(\transp\beta,\transp\alpha)}$. As representations of $W(D_n)$, we have $\sgn\otimes\rho_{(\alpha,\beta)_i} = \rho_{(\transp\beta,\transp\alpha)_i}$, where $\sgn$ is the sign representation of $W(D_n)$. 
Let $(\lambda,[\eps]) \in \PPort(N)$, $k = k(\lambda,[\eps])$. If $(\alpha,\beta))_k = \Phi_N(\lambda,[\eps])$, then let $(\transps\lambda,\transps\eps) = \Phi_N^{-1}((\transp\beta,\transp\alpha)_k)$.

\begin{lemma}\label{lemma:transp}
We have
\begin{align}
2\Lambda_{k/2,-k/2;2}(\alpha,\beta) = \transp(\transps\lambda) + ([0,-\infty[_1)^2.
\end{align}
\end{lemma}
\begin{proof}
Let 
$U = 2\alpha + [k,-\infty[_1$, $V = 2\beta + [-k,-\infty[_1$, $X = (\transp\beta)^2 + [k,-\infty[_1$, $Y = (\transp\alpha)^2 + [-k,-\infty[_1$.
Then 
\begin{align}\label{eq:UV}
2\Lambda_{k/2,-k/2;2}(\alpha,\beta) = U \sqcup V.	
\end{align}
For $x \in \Z$, we have $\mult_U(x) = \mult_{R_{2\alpha}}(x-k)$ and $\mult_Y(x) = \mult_{R_{(\transp\alpha)^2}}(x+k)$. Hence by \Cref{lemma:mult}, we have
\begin{align}\label{eq:multUY}
\mult_U(x) + \mult_Y(1-x) = 1,
\end{align}
and similarly
\begin{align}\label{eq:multVX}
\mult_V(x) + \mult_X(1-x) = 1.
\end{align}
Write $(\mu,\tau) = (\transps\lambda,\transps\eps)$. For $\nu \in \mathcal R$, write $\nu - 1 = (\nu_1-1,\nu_2-1\dots)$. We have
\begin{align}
p_{\mu,\tau} = \Lambda_{k,-k;2}(\transp\beta,\transp\alpha) = (\transp\beta + [k,-\infty[_2) \sqcup (\transp\alpha + [-k,-\infty[_2).
\end{align}
Note that $X \sqcup Y = p_{\mu,\tau} \sqcup (p_{\mu,\tau}-1)$, so \eqref{eq:UV}, \eqref{eq:multUY} and, \eqref{eq:multVX} give
\begin{align}\label{eq:mult2}
\mult_{2\Lambda_{k/2,-k/2;2}(\alpha,\beta)}(x) + \mult_{p_{\mu,\tau} \sqcup (p_{\mu,\tau}-1)}(1-x) = 2.
\end{align}
Define $\mu',\mu'' \in \mathcal R$ by $\mu'_j = \lfloor \mu_j/2 \rfloor$ and $\mu''_j = \lceil \mu_j/2 \rceil$ for all $j \in \N$. Then $\mu = \mu' + \mu''$. We show that for all $x \in \Z$, we have
\begin{align}\label{eq:pR}
p_{\mu,\tau} \sqcup (p_{\mu,\tau}-1) = R_{\mu'} \sqcup R_{\mu''}.
\end{align}
Let $i \in \N$. Suppose $\mu_i$ is odd. If $i = 2j+1$ is odd, then
\begin{align}
 (A_\mu^\#)_{j+1} = \frac{\mu_i + (1 - i) + 1}{2} + 1 - \frac{i+1}{2} = \frac{\mu_i+1}{2} + 1 - i = (R_{\mu''})_i
\end{align}
and so $(A_\mu^\#)_{j} - 1 = (R_{\mu'})_i$. If $i = 2j$ is even, then 
\begin{align}
 (B_\mu^\#)_{j} = \frac{\mu_i + 1 - i}{2} + 1 - \frac{i}{2} = \frac{\mu_i+1}{2} + 1 - i = (R_{\mu''})_i
\end{align}
and so $(B_\mu^\#)_{j} - 1 = (R_{\mu'})_i$.
Since $\mu$ is orthogonal, we can partition the set consisting of the $i \in \N$ for which $\mu_i$ is even into pairs such that each pair contains consecutive integers $i,i+1$ and such that $\mu_i = \mu_{i+1}$. For pairs of the form $(2j-1,2j)$, $S_{2j}(\lambda)$ is even and for pairs of the form $(2j,2j+1)$, $S_{2j+1}(\lambda)$ is odd.

Consider a pair $(i,i+1)$ with $i = 2j-1$, $i + 1 = 2j$ and $\mu_i = \mu_{i+1}$ even. Then
\begin{align}
(A_\mu^\#)_j = (B_\mu^\#)_j = \frac{\mu_{i} + 1 - i}{2} + 1 - j = \frac{\mu_{i}}{2} + 1  - i = (R_{\mu'})_{i} &= (R_{\mu''})_{i},
\\
(A_\mu^\#)_j - 1 = (B_\mu^\#)_j - 1 = (R_{\mu'})_{i+1} &= (R_{\mu''})_{i+1}.
\end{align}
Consider a pair $(i,i+1)$ with $i = 2j$, $i + 1 = 2j+1$ and $\mu_i = \mu_{i+1}$ even. Then 
\begin{align}
(A_\mu^\#)_{j+1} = (B_\mu^\#)_j = \frac{\mu_{i} + (1-i) + 1}{2} + 1 - j  =  \frac{\mu_{i+1}}{2} + 2  - i = (R_{\mu'})_{i+1} &= (R_{\mu''})_{i+1},
\\
(A_\mu^\#)_{j+1} - 1 = (B_\mu^\#)_j - 1 = (R_{\mu'})_{i} &= (R_{\mu''})_{i}.
\end{align}
We conclude that \eqref{eq:pR} holds.

We apply \Cref{lemma:mult} to $\mu'$ and $\mu''$, which together with \eqref{eq:mult2} and \eqref{eq:pR} gives
$\mult_{2\Lambda_{k/2,-k/2;2}(\alpha,\beta)}(x) = \mult_{R_{\transp\mu'}}(x) + \mult_{R_{\transp\mu''}}(x)$ for all $x \in \Z$, hence
\begin{align}\label{eq:LambdaRR}
2\Lambda_{k/2,-k/2;2}(\alpha,\beta) = R_{\transp\mu'} \sqcup R_{\transp\mu''}.
\end{align}
We show that $\transp \mu''_j \geq \transp \mu'_j \geq \transp \mu''_{j+1}$ for all $j \in \N$.
Let $h = \transp \mu_j'$. Then $\mu_h' \geq j$, and since $\mu_h'' \geq \mu_h'$, we have $\transp\mu_j'' \geq h$. 
Now let $h = \transp \mu_{j+1}''$. Then $\mu_h'' \geq j + 1$, so $\mu_h' \geq \mu_h'' - 1 \geq j$ and so $\transp\mu_j' \geq h$, as desired.

Since $(\transp\mu'\sqcup\transp\mu'') = \transp(\mu'+\mu'') = \transp\mu$, we now have
\begin{align}
R_{\transp\mu'}\sqcup R_{\transp\mu''} 
= 
(\transp\mu_1'',\transp\mu_1',\transp\mu_2''-1,\transp\mu_2'-1,\dots)
=
(\transp\mu'\sqcup\transp\mu'') + ([0,-\infty[_1)^2
=
\transp\mu + ([0,-\infty[_1)^2,
\end{align}
which together with \eqref{eq:LambdaRR} finishes the proof.
\end{proof} 

For the following, note that for $(\lambda,[\eps]) \in \PPort(N)$, it holds that $\slambdamin$ is degenerate if and only if $\lambdamin$ is degenerate.

\begin{theorem}\label{thm:min}
Let $(\lambda,[\eps]) \in \PPort(N)$ with $\lambda$ only consisting of odd parts. Then there exists a unique $(\lambdamin,\epsmin) \in \PPort(N)$ such that $\slambdamin$ is non-degenerate and
\begin{enumerate}
\item $\mult(C_\lambda,\E_{[\eps]};C_{\slambdamin},\E_{[\sepsmin]}) = 1$,
\item For all $(\lambda',[\eps']) \in \PPort(N)$ with $\mult(C_\lambda,\mathcal E_{[\eps]};C_{\slambda'}^+,\E_{[\seps']}^+) = \mult(C_\lambda,\mathcal E_{[\eps]};C_{\slambda'}^-,\E_{[\seps']}^-) \neq 0$, we have $\lambdamin < \lambda'$ or $(\lambda',[\eps']) = (\lambdamin,[\epsmin])$.
\end{enumerate}
Furthermore, we have $(\lambda^{\text{max}},[\epsmax]) = (\slambdamin,[\sepsmin])$.
\end{theorem}
\begin{proof}
Let $k = k(\lambda,[\eps])$. Let $(\lambda',[\eps']),(\lambda'',\eps'') \in \PPort(N)$ such that $k(\lambda',[\eps']) = k(\lambda'',\eps'') = k$. Let $(\alpha',\beta')_k = \Phi_N(\lambda',[\eps'])$ and $(\alpha'',\beta'')_k = \Phi_N(\lambda'',\eps'')$. 

The theorem is equivalent to the following statement:
there exists unique a $(\underline\lambda,\underline\eps) \in \PPort(N)$ such that $\bar\lambda$ is non-degenerate and 
\begin{enumerate}[(a)]
\item\label{reforma} $\mult(C_\lambda,\E_{[\eps]};C_{\underline\lambda},\E_{\underline\eps}) = 1$,
\item\label{reformb} For all $(\lambda',[\eps']) \in \PPort(N)$ with $\mult(C_\lambda,\E_{[\eps]};C_{\lambda'}^\pm,\E_{[\eps']}^\pm) \neq 0$, we have $\transps{\underline{\lambda}} < \transps\lambda'$ or $(\lambda',[\eps']) = (\underline\lambda,\underline\eps)$,
\end{enumerate}
and furthermore, we have $(\underline\lambda,\underline\eps) = (\lambda^{\text{max}},[\epsmax])$

Since transposition is an order reversing operation on $\mathcal P$, it follows from \Cref{lemma:transp} that
\begin{enumerate}[(I)]
\item\label{fact:order} $\transps\lambda' \leq \transps\lambda''$ if and only if $\Lambda_{k/2,-k/2;1/2}(\alpha'',\beta'') \leq \Lambda_{k/2,-k/2;1/2}(\alpha',\beta')$. 
\end{enumerate}
Denote $<_{\alpha,\beta,k}$ by $<$. By \Cref{lemma:PP}, we have
\begin{align}
P_{k/2,-k/2;1/2}(\alpha,\beta,<)
=
P_{k,-k;2}(\alpha,\beta,<) 
= 
\{(\alphamax,\betamax)\}
\end{align}
and we showed in the proof of \Cref{thm:max} that $(\alphamax,\betamax)_k = \Phi_N(\lambdamax,[\epsmax])$. 
Thus by \Cref{prop:mult} (with $A = -B = k/2$, $s=1/2$), $(\alphamax,\betamax)$ is the unique element $(\underline\alpha,\underline\beta)\in\mathcal P \times \mathcal P$ such that
\begin{enumerate}[(A)]
\item\label{conditiona} $\mult(\alpha,\beta,<;\underline\alpha,\underline\beta) = 1$;
\item\label{conditionb} For each $(\alpha',\beta') \in \mathcal P \times \mathcal P$ such that $\mult(\alpha,\beta,<;\alpha',\beta') \neq 0$, we either have $\Lambda_{k/2,-k/2;1/2}(\alpha',\beta') \leq \Lambda_{k/2,-k/2;1/2}(\underline\alpha,\underline\beta)$ or $(\alpha',\beta')_k = (\underline\alpha,\underline\beta)_k$ (this is an equality of ordered (resp. unordered) pairs if $k>0$ (resp. $k = 0$).
\end{enumerate}
Now $(\underline \lambda, \underline \eps) = \Phi_N^{-1}(\underline\alpha,\underline\beta) = (\lambdamax,[\epsmax])$ satisfies \ref{reforma} by \Cref{thm:max}, 
and \ref{reformb} by \ref{conditionb}, \Cref{prop:multshoji} and \ref{fact:order}. 
Uniqueness follows easily from the fact that $P_{k/2,-k/2;1/2}(\alpha,\beta,<)$ has a unique element.
Thus we have shown that $(\lambdamax,[\epsmax])$ is the unique element of $\PPort(N)$ that satisfies \ref{reforma} and \ref{reformb}.
\end{proof}

\subsection{Orthogonal partitions with even parts}
Let $(\lambda,[\eps]) \in \PPort(N)$, $k=k(\lambda,[\eps])$ and $(\alpha,\beta)_k) = \Phi_N(\lambda,[\eps])$
In \Cref{thm:max}, we assumed that $\lambda \in \PPort(N)$ only has odd parts. 
Note that $\lambda$ only has odd parts if and only if $p_{\lambda,\eps} \in H_S(n,k)$, where and $n = (N-k^2)/2$. 
If $(\alpha,\beta) \notin H(n,k)$ then 
we cannot define $<_{\alpha,\beta,k}$.
However, consider the following order with a weaker condition: for $m_0,m_1$ as above, define $(m_0,m_1,\tilde<)$ such that $(i,0) \tilde< (j,1)$ if $\alpha_i + k + 2 -2i > \beta_j - k + 2 - 2j$. 
If we mimic the proof of \Cref{lemma:QR} for $(\alpha,\beta)$, then \eqref{eq:wminus} is no longer necessarily true. As such, we cannot show that $Q_{(\alpha,\beta)} = R_{(\alpha,\beta)}$, which was a crucial part in the proof of \Cref{thm:max}, since it allowed the use of \Cref{lemma:multshoji}.

In fact, there exists an orthogonal partition $\lambda \in \mathcal P$ with even parts for which $Q_{(\alpha,\beta)} \neq R_{(\alpha,\beta)}$. Although a tedious exercise, one can show that for $\lambda = (4 \, 4 \, 1) \in \PPort(9)$ and any order $\tilde<$ defined above, we have $Q_{(\alpha,\beta)} \neq R_{(\alpha,\beta)}$.

For $(\lambda,[\eps]) \in \PPort(N)$ appearing in the usual Springer correspondence (i.e. $k(\lambda,[\eps]) \in \{0,1\}$) for which $\lambda$ has even parts, we shall study the Green functions via a different perspective using an induction theorem from Lusztig.

Let $G = \SO(N)$, $T$ a maximal torus and $B$ the Borel subgroup of $G$ associated to $T$. 
Let $u \subseteq G$ be a unipotent element and $\phi \in A(u)^\wedge$ an irreducible representation of the component group $A(u) = Z_G(u)/Z_G^\circ(u)$ of $u$ in $G$. Let $\mathcal B = G/B$ be the flag variety and $\mathcal B_u = \{B \in \mathcal B \colon u \in B\}$.
Let $W$ be the Weyl group of $G$.
We have the following classical result (cf. \eqref{eq:multgsc}).
\begin{proposition}
Let $\rho \in W^\wedge$ and suppose it corresponds to $(u,\phi)$ via the original Springer correspondence.
Let $H^\bullet(\mathcal B_u)^\phi$ be the $\phi$-isotypic component of the total cohomology $H^\bullet(\mathcal B_u)$ of $\mathcal B_u$. 
Then we have the following isomorphism as $W$-representations
\begin{align}
H^\bullet(\mathcal B_u)^\phi \cong \bigoplus_{\rho' \in W^\wedge} p_{\rho',\rho}(1)\rho',
\end{align}
where $p_{\rho',\rho}$ are the Green functions of $W$ as in \Cref{theorem:shoji} for $k=1$ if $N$ is odd, and $k = 0$ if $N$ is even. 
\end{proposition}

Let $L$ be a Levi subgroup attached to some standard parabolic of $G$ and let $W'$ be the Weyl group of $L$ (note that $W'$ is not a relative Weyl group of $G$). Then $W'$ can be canonically considered as a subgroup of $W$.
Suppose that $u \in L$ is unipotent and $\phi \in A^L(u)^\wedge$ where $A^L(u) = Z_L(u) / Z_L^\circ(u)$. Since $G= \SO(N)$, we can identify $A^L(u) = A(u)$. 
Let $\mathcal B' = L/B$ and $\mathcal B'_u = \{B' \in \mathcal B' \colon u \in B'\}$.
We have the following induction theorem.
It was first stated in \cite{alvis1982springer} without proof; a proof is given in \cite{lusztig2004induction} (see also \cite[Proposition 3.3.3]{reeder2001euler}. 
\begin{proposition}\label{prop:induction}
As $W$-representations, we have
\begin{align}
H^\bullet(\mathcal B_u)^\phi \cong \ind_{W'}^W H^\bullet(\mathcal B_u')^{\phi}.
\end{align}
\end{proposition}

Let $(\lambda,[\eps]) \in \PPort(N)$ with $k:= k(\lambda,[\eps]) \in \{0,1\}$. 
Let $u \in G$ be a unipotent element parametrised by $\lambda$. Let $\epsodd = \eps$ and write $\lambda = \lambdaodd \sqcup \lambdaeven$ where $\lambdaodd$ (resp. $\lambdaeven$) consists of all the odd (resp. even) parts of $\lambda$. 
Write $\lambdaeven = (a_1,a_1,a_2,a_2,\dots,a_\ell,a_\ell)$ with $a_1 \geq a_2 \geq \dots \geq a_\ell$ and let $a = a_1 + \dots + a_\ell$.
Let $N' = t(\lambdaodd)$ and $n' = \lfloor N' / 2 \rfloor$.  
Let $L$ be a Levi subgroup of $G$ of a standard parabolic such that $u \in L$. 
Suppose $\phi \in A^L(u) = A(u)$ corresponds to $\eps$.
Then
\begin{align}
L \cong \SO(N') \oplus \GL(a_1) \oplus \GL(a_2) \oplus \dots \oplus \GL(a_\ell)
\end{align}
and $W' = W_{n'} \times \prod_{i=1}^\ell S_{a_i}$ where $W_{n'} = W(B_{n'})$ (resp. $W_{n'} = W(D_{n'})$) if $N'$ is odd (resp. even).
Note that $N' \equiv N \mod 2$, so $W_n'$ and $W$ are of the same type.
Let $u' \in L$ be the element parametrised by $\lambdaodd$.
By \Cref{prop:induction}, we have 
\begin{align}\label{eq:induction}
H^\bullet(\mathcal B_{u})^\phi \cong \ind_{W'}^W(H^\bullet(\mathcal B_{\uodd})^{\phi} \boxtimes \prod_{i=1}^\ell \triv_{S_{a_i}}),
\end{align}
where $\triv_{S_{a_i}}$ is the trivial representation of $S_{a_i}$. 

We can explicitly determine the subrepresentations of $H^\bullet(\mathcal B_{u})^\phi$
using the Littlewood-Richardson rule for type $B$ and $D$ Weyl groups, see \cite[\S6]{geck2000characters} for type $B$, \cite{taylor2015induced} for type $D$, or \cite{stembridgepractical} for an overview of type $B$ and $D$. We note that \cite{stembridgepractical} explicitly states the results that we will use.
Let $I_0$ be the set of $(\alpha,\beta) \in \mathcal P_2(n')$ such that $\rho_{(\alpha,\beta)}$ is a $W_{n'}$-subrepresentation of $H^\bullet(\mathcal B_{\uodd})^{\phi}$.
For $i=0,\dots,\ell$, let $n_i = n' + a_1 + \dots + a_i$.
We recursively define sets $I_i$ as follows. 
Suppose $i \in \{1,\dots,\ell\}$.
Let $I_i$ be the set of $(\gamma,\delta) \in \mathcal P_2(n_i)$ such that there exist a $(\tilde\gamma,\tilde\delta) \in I_{i-1}$ such that
\begin{align}\label{eq:pieri}
\tilde\gamma^t_j \leq \gamma^t_j \leq \tilde\gamma^t_j + 1
\quad \text{and} \quad
\tilde\delta^t_j \leq \delta^t_j \leq \tilde\delta^t_j + 1.
\end{align}
By Pieri's rule for type $B$ and $D$ (see \cite{stembridgepractical}), $I_i$ consists of precisely all $(\gamma,\delta)  \in \mathcal P_2(n_i)$ such that $\rho_{(\gamma,\delta)}$ is a $W_{n_i}$-subrepresentation of $\ind_{W_{n_{i-1}}\times S_{a_i}}^{W_{n_i}}(\rho_{(\tilde\gamma,\tilde\delta)} \boxtimes \triv_{S_{a_i}})$. 
Since induction and direct products are compatible, we see from \eqref{eq:induction} that $I_\ell$
consists of all irreducible constituents of $H^\bullet(\mathcal B_{u})^\phi$.

Let $(\lambda,[\eps]) \in \PPort(N)$ with $k:= k(\lambda,[\eps]) \in \{0,1\}$.
Let $((\lambdaodd)^{\text{max}},[(\epsodd)^{\text{max}}])$ be as in \Cref{thm:max} for $(\lambdaodd,[\epsodd])$.
Define $(\lambdamax,[\epsmax]) \in \PPort(N)$ such that for all $j \in \Z_{\geq2}$, we have
\begin{align}
\lambda^{\text{max}}_1 &= (\lambdaodd)^{\text{max}}_1 + \sum_{i\in\N} \lambda^{\text{even}}_i =  (\lambdaodd)^{\text{max}}_1 + 2a,
\\
\lambda^{\text{max}}_j &= (\lambdaodd)^{\text{max}}_j,
\end{align}
and such that $\epsmax = (\eps^{\text{odd}})^{\text{max}}$.
Note that $\lambdamax$ is degenerate.
\begin{theorem}\label{thm:maxeven}
Let $(\lambda,[\eps]) \in \PPort(N)$ with $k:= k(\lambda,[\eps]) \in \{0,1\}$. 
Then $(\lambdamax,[\epsmax])$ is the unique element of $\PPort(N)$ such that 
\begin{enumerate}
\item\label{thm:maxeven1} $\mult(C_\lambda,\E_{[\eps]};C_{\lambdamax},\E_{[\epsmax]}) = 1$,
\item\label{thm:maxeven2} For all $(\lambda',[\eps']) \in \PPort(N)$ with $\mult(C_\lambda,\E_{[\eps]}^\pm;C_{\lambda'},\E_{[\eps']}^\pm) \neq 0$, we have $\lambda' < \lambda^{\text{max}}$ or $(\lambda',[\eps']) = (\lambdamax,[\epsmax])$.
\end{enumerate}
\end{theorem}
\begin{proof}
Let $(\alpha,\beta)_k = \Phi_N(\lambdaodd,\epsodd)$ and let $(\mu,\nu)$ be the unique element of $P_{k,-k;2}(\alpha,\beta,<_{\alpha,\beta,k})$. Then $(\mu,\nu)_k = \Phi_N((\lambdaodd)^{\text{max}},(\epsodd)^{\text{max}})$, as shown in the proof of \Cref{thm:max}. 
If $(1,0) <_{\alpha,\beta,k} (1,1)$, let $(\bm\alpha,\bm\beta) = (\mu + (a,0,0,\dots),\mu)$
and if $(1,1) <_{\alpha,\beta,k} (1,0)$, let $(\bm\alpha,\bm\beta) = (\mu,\nu + (a,0,0,\dots))$.
By definition of the $I_i$, we have $(\bm\alpha,\bm\beta) \in I_\ell$.
Let 
$\Lambda = \Lambda_{k,-k;2}(\mu,\nu)$ and
$\bm\Lambda = \Lambda_{k,-k;2}(\bm\alpha,\bm\beta)$.
Let $(\bm\alpha',\bm\beta') \in I_\ell$.
Since \eqref{eq:pieri} holds for $i=1,\dots,\ell$, there exists an $(\alpha',\beta') \in I_0$ and $x \in \Z_{\geq0}^{t(\alpha)},y \in \Z_{\geq0}^{t(\beta)}$ such that
\begin{align}
\bm\alpha' = \alpha' + x
\quad
\bm\beta' = \beta' + y,
\quad
\sum_j (x_j + y_j) &= a.
\label{eq:gammadelta}
\end{align}
Clearly, if $(\alpha',\beta') \neq (\mu,\nu)$, then $(\bm\alpha',\bm\beta') \neq (\bm\alpha,\bm\beta)$, and by \Cref{thm:max}, the multiplicity of $\rho_{(\mu,\nu)}$ in $H^\bullet(\mathcal B_u')^{\phi'}$ is $1$. 
Furthermore, the multiplicity of $\rho_{(\bm\alpha,\bm\beta)}$ in 
\begin{align}
\ind_{W_{n'}\times \prod_{i=1}^\ell S_{a_i}}^{W_{n}}(\rho_{(\alpha,\beta)} \boxtimes \prod_{i=1}^\ell \triv_{S_{a_i}})
\end{align}
can be computed using results  from \cite[\S2.A, \S2.B]{stembridgepractical} for type $B$ and \cite[\S3.A, \S3.C]{stembridgepractical} for type $D$, and it follows that this multiplicity is $1$.
Thus part \ref{thm:maxeven1} of the theorem follows. 
Let $z$ be the sequence in $\mathcal R$ obtained by rearranging the terms of $x \sqcup y$.
Let
$\Lambda' = \Lambda_{k,-k;2}(\alpha',\beta')$
and
$\bm\Lambda' = \Lambda_{k,-k;2}(\bm\alpha',\bm\beta')$. 
Let $i=1,\dots,\ell$.
By \eqref{eq:gammadelta}, we have $\Lambda'_i + S_i(z) \geq \bm\Lambda_i'$ and $a_1+\dots+a_\ell \geq S_i(z)$. 
Furthermore, $\Lambda_i \geq \Lambda_i'$ by \Cref{lem:max}, so 
\begin{align}
\bm\Lambda_i 
= \Lambda_i + a_1 + \dots + a_\ell 
\geq \Lambda_i' + S_i(z) \geq \bm\Lambda'_i.
\end{align}
Hence $\bm\Lambda \geq \Lambda_{k,-k;2}(\bm\alpha',\bm\beta')$ for all $(\bm\alpha',\bm\beta') \in I_\ell$, so part \ref{thm:maxeven2} of the theorem follows from \Cref{lemma:order}.
To show that $(\lambdamax,[\epsmax])$ is unique, suppose that $\bm\Lambda = \bm\Lambda'$.
Then $\bm\Lambda_1' = \bm\Lambda_1 = \Lambda_1+a$. Since $\Lambda_1 \geq \Lambda'_1$ and $\bm\Lambda_1' \leq \Lambda'_1+a$, we have $\Lambda_1 = \Lambda_1'$. From this and \eqref{eq:gammadelta}, it follows that $\Lambda_i = \bm\Lambda_i = \bm\Lambda_i' = \Lambda_i'$ for all $i \in \N$, and so $\Lambda = \Lambda'$, hence $(\alpha',\beta') \in P_{k,-k;2}(\alpha,\beta,<_{\alpha,\beta,k})$ by \Cref{lem:max}. Since $P_{k,-k;2}(\alpha,\beta,<_{\alpha,\beta,k})$ has a unique element $(\mu,\nu)$, we have $(\alpha',\beta') = (\mu,\nu)$, and using that $\bm\Lambda_1' = \Lambda_1' + a$, we find that $(\bm\alpha,\bm\beta) = (\bm\alpha',\bm\beta')$. Thus $(\lambdamax,[\epsmax])$ is unique. 
\end{proof}

\begin{theorem}\label{thm:mineven}
Let $(\lambda,[\eps]) \in \PPort(N)$ with $k:= k(\lambda,[\eps]) \in \{0,1\}$ and consider the notation as above.
Then there exists a unique $(\lambdamin,[\epsmin]) \in \PPort(N)$ such that $\slambdamin$ is non-degenerate and
\begin{enumerate}
\item $\mult(C_\lambda,\E_{[\eps]};C_{\slambdamin},\E_{[\sepsmin]}) = 1$,
\item For all $(\lambda',[\eps']) \in \PPort(N)$ with $\mult(C_\lambda,\E_{[\eps]};C_{\lambda'}^\pm,\E_{[\eps']}^\pm) \neq 0$, we have $\lambdamin < \lambda'$ or $(\lambda',[\eps']) = (\slambdamin,{}^s[\epsmin])$.
\end{enumerate}
Furthermore, we have $(\lambdamax,[\epsmax]) = (\slambdamin,[\slambdamin])$.
\end{theorem}
\begin{proof}
Let $(\bm\alpha,\bm\beta),(\bm\alpha',\bm\beta')$ as in the proof of \Cref{thm:maxeven}.
By the same arguments as in the proof of \Cref{thm:maxeven}, we can show that $\Lambda_{k/2,-k/2;2}(\bm\alpha,\bm\beta) \geq \Lambda_{k/2,-k/2;2}(\bm\alpha',\bm\beta')$, with equality if and only if $(\bm\alpha,\bm\beta) = (\bm\alpha',\bm\beta')$. Hence
\begin{align}
P_{k/2,-k/2;1/2}(\alpha,\beta,<_{\alpha,\beta,k})
=
P_{k,-k;2}(\alpha,\beta,<_{\alpha,\beta,k}) 
= 
\{(\bm\alpha,\bm\beta)\}.
\end{align}
By the arguments in \Cref{thm:min}, it then follows that $(\lambdamin,[\epsmin]) = (\slambdamax,[\sepsmax])$ is the unique element satisfying the two conditions.
\end{proof}

\begin{remark}
By similar arguments and using analogous results for $\Sp(2n)$ from \cite{waldspurger}, we similarly obtain the analogues of \Cref{thm:maxeven} and \ref{thm:mineven} for $\Sp(2n)$ for any $(\lambda,[\eps]) \in \PPsymp(2n)$ such that $k(\lambda,[\eps])) = 0$.
\end{remark}

\section{Algorithm for $(\lambdamax,\epsmax)$}\label{sec:alg}

Let $N\in \N$.
Recall $\PPPort(N)$ defined in \S\ref{sec:gsc} and let $(\lambda,\eps) \in \PPPort(N)$.
Suppose $\lambda$ only has odd parts. We give an algorithm that outputs an element $(\bar\lambda,\bar\eps) \in \PPPort(N)$
such that $(\bar\lambda,[\bar\eps]) = (\lambdamax,[\epsmax]) \in \PPort(N)$. 
The algorithm and its proof are inspired by the discussion of the algorithm for $\Sp(2n)$ in \cite[\S5]{waldspurger}, 
but similarly as we have seen before, there will be some differences between the proofs since the symbols are different in our setting. 
In particular, when $k(\lambda,[\eps]) = 0$, the symbols are unordered pairs, which we have to be especially careful about.

Let $t = t(\lambda)$. 
We view $\eps$ as a map $\{1,\dots,t\} \to \{\pm1\}$ by setting $\eps(i) = \eps_{\lambda_i}$ for all $i \in \N$. Let $u \in \{\pm1\}$ and consider the finite sets
\begin{align}
\mathfrak S &= \{1\} \cup \{i \in \{2,\dots,t\} \colon \eps(i) = \eps(i-1)\} = \{s_1 < s_2 < \dots < s_p\}, \text{ for some $p \in \N$, }
\\
J^u &= \{ i \in \{1,\dots,t\} \colon \eps(i)(-1)^{i+1} = u\},
\\
\tilde J^u &= J^u \setminus \mathfrak S.
\end{align}
Recall from \S\ref{sec:gsc} that $M(\lambda,\eps) = |J^1| - |J^{-1}|$ and that $k(\lambda,\eps) = |M|$. Note that $\mathfrak S, J^u, M$ depend on the choice of representative $\eps$ of $[\eps]$. In particular, we have $M(\lambda,\eps) = -M(\lambda,-\eps)$.

We define $(\bar\lambda,\bar\eps)$ by induction. For $N=0,1$. Let $(\bar\lambda,\bar\eps) = (\lambda,\eps)$. Suppose $N > 1$ and suppose that we have defined $(\bar\lambda',\bar\eps')$ for any $(\lambda',\eps') \in \PPPort(N')$ with $N'<N$.
Let 
\begin{align}
\bar\lambda_1 
&=
\sum_{i \in \mathfrak S} \lambda_{i} - 2|\tilde J^{-\eps(1)}| - \frac{1 + (-1)^{|\mathfrak S|}}{2}
=
\begin{cases}
\sum_{i \in \mathfrak S} \lambda_{i} - 2|\tilde J^{-\eps(1)}| &\text{if $|\mathfrak S|$ is odd,}
\\
\sum_{i \in \mathfrak S} \lambda_{i} - 2|\tilde J^{-\eps(1)}| - 1 &\text{if $|\mathfrak S|$ is even,}
\end{cases}
\\
\bar \eps(1) &= \eps(1).
\end{align}
Let $r' = |\tilde J^1| + |\tilde J^{-1}| = t(\lambda) - |\mathfrak S|$ and $\phi \colon \{1,\dots,r'\} \to \tilde J^1 \cup \tilde J^{-1}$ be the unique increasing bijection. 
Let 
$N' = N - \bar\lambda_1$. 
We define $\lambda' \in \Z^{r'+(1+(-1)^{|\mathfrak S|})/2}$ and $\eps' \colon \{1,\dots,r'+\frac{1+(-1)^{|\fS|}}{2}\} \to \{\pm1\}$ as follows.
\begin{align}\label{eq:lambdaprime}
(\lambda'_j,\eps'(j)) &= 
\begin{cases}
(\lambda_{\phi(j)},\eps(\phi(j))) &\text{ if } j\leq r', \phi(j) \in \tilde J^{\eps(1)},
\\
(\lambda_{\phi(j)}+2,-\eps(\phi(j))) &\text{ if } j\leq r', \phi(j) \in \tilde J^{-\eps(1)},
\\
(1,(-1)^{r'}\eps(1)) &\text{ if } j = r'+1 \text{ and $|\mathfrak S|$ is even},
\\
0 &\text{ else.}
\end{cases}
\end{align}
Note that we can similarly define $(\lambda',(-\eps)')$, in which case we have $(-\eps)' = -\eps'$, i.e. $[-\eps'] = [\eps']$.
We prove in \Cref{lem:lambdaprime2} that $N' < N$.
By induction, we have defined $(\bar\lambda', \bar \eps') \in \PPPort(N')$ by applying the algorithm to $(\lambda',\eps') \in \PPPort(N')$. 

Let $\ell = t(\bar\lambda')$. We define $\bar\lambda \in \Z^{\ell+1}$ and $\bar\eps \colon \{1,\dots,\ell+1\} \to \{\pm1\}$ as follows: for $i = 1,\dots,\ell$, let
\begin{align}\label{eq:lambdabar}
\bar\lambda_{i+1} &:= \bar\lambda'_{i-1},
\\
\bar \eps(i+1) &:= \bar\eps'(i).
\end{align}

\begin{theorem}\label{thm:algorithm}
Let $(\lambda,\eps) \in \PPPort(N)$ such that $\lambda$ only has odd parts and let $(\lambdamax,\epsmax) \in \PPPort(N)$ such that $(\lambdamax,[\epsmax])$ is as in Theorem \ref{thm:max}.
Let $k = k(\lambda,\eps)$ and let $(\alpha,\beta) \in \cP_2((N-k^2)/2)$ so that $(\alpha,\beta)_k = \Phi_N(\lambda,[\eps])$ and such that if $k=0$, then $\alpha_1 = (A_{\lambda,\eps})_1$ and $\beta_1 = (B_{\lambda,\eps})_1$.
Let $(\alphamax,\betamax) \in P_{k,-k;2}(\alpha,\beta,<_{\alpha,\beta,k})$.
Then it holds that $(\bar\lambda, [\bar \eps]) = (\lambdamax,[\epsmax])$.
Furthermore, we have $M(\bar\lambda,\bar\eps) = M(\lambda,\eps)$ and $(A_{\bar\lambda,\bar\eps},B_{\bar\lambda,\bar\eps}) = (\alphamax + [k,-\infty[_2,\betamax + [-k,-\infty[_2)$.
\end{theorem}

\begin{remark}
For $k>0$, the statement $(A_{\bar\lambda,\bar\eps},B_{\bar\lambda,\bar\eps}) = (\alphamax + [k,-\infty[_2,\betamax + [-k,-\infty[_2)$ is equivalent to the statement $(\bar\lambda,[\bar\eps]) = \Phi_N^{-1}(\alphamax,\betamax) = (\lambdamax,[\epsmax])$.
However, for $k = 0$, the statement $(\bar\lambda,[\bar\eps]) = (\lambdamax,[\epsmax])$ is equivalent to  $\{A_{\bar\lambda,\bar\eps},B_{\bar\lambda,\bar\eps}\} = \{\alphamax + [k,-\infty[_2,\betamax + [-k,-\infty[_2\}$, which is weaker than an equality of ordered pairs. We include this equality of ordered pairs in the theorem, as the theorem will be proved by induction on $N$ and this equality of ordered pairs will become useful for that.
\end{remark}

\subsection{Properties of $(\lambda',\eps')$ and $(\bar\lambda,\bar\eps)$}
Recall that $\mathfrak S = \{s_1,\dots,s_p\}$ and note that
\begin{align}\label{eq:JS}
J^{\eps(1)} &= \bigcup_{i} \{s_{2i-1},\dots,s_{2i}-1\}, 
&J^{-\eps(1)} &= \bigcup_{i} \{s_{2i},\dots,s_{2i+1}-1\},
\\
\tilde J^{\eps(1)} &= \bigcup_{i} \{s_{2i-1}+1,\dots,s_{2i}-1\}, 
&\tilde J^{-\eps(1)} &= \bigcup_{i} \{s_{2i}+1,\dots,s_{2i+1}-1\}. \label{eq:J}
\end{align}
Throughout this subsection, we assume that $N > 1$ and so $\lambda \neq 0$.
\begin{lemma}\label{lem:J}
Suppose $r' \neq 0$ and let $i,j \in \{1,\dots,r'\}$ with $i < j$. Then $\lambda_{\phi(i)} \geq \lambda_{\phi(j)} + 2(j-i)$.
\end{lemma}
\begin{proof}
Suppose first that there exists no $s \in \mathfrak S$ such that $\phi(j-1) < s < \phi(j)$. Then $\phi(j-1) = \phi(j) - 1$, and since $\phi(j) \notin \mathfrak S$, we have $\eps(\phi(j-1)) = - \eps(\phi(j))$. Thus $\lambda_{\phi(j-1)} > \lambda_{\phi(j)}$, and since these are odd, we have $\lambda_{\phi(j-1)} \geq \lambda_{\phi(j)} + 2$. 
Suppose there exists an $s \in \mathfrak S$ such that $\phi(j-1) < s < \phi(j)$.
Pick the largest such $s$ so that $\phi(j) \geq s+1 \notin \mathfrak S$. Then $\eps(s+1) = - \eps(s)$, hence $\lambda_{s} \geq \lambda_{s+1} + 2$. 
Thus we have
\begin{align}
\lambda_{\phi(j-1)} \geq \lambda_s \geq \lambda_{s+1} + 2 \geq \lambda_{\phi(j)}+2.
\end{align}
In both cases, we have $\lambda_{\phi(j-1)} \geq \lambda_{\phi(j)} + 2$. 
Thus we have
\begin{align}
\lambda_{\phi(i)} 
\geq 
\lambda_{\phi(i+1)} + 2
\geq
\lambda_{\phi(i+2)} + 4
\geq
\dots
\geq
\lambda_{\phi(j)} + 2(j-i). &\qedhere
\end{align}
\end{proof}
\begin{proposition}\label{lem:lambdaprime}
\begin{enumerate}[ref={\theproposition(\arabic*)}]
\crefalias{enumi}{proposition}
\item\label{lem:lambdaprime1} We have $(\lambda',\eps') \in \PPPort(N')$ and $\lambda'$ only has odd parts.
\item\label{lem:lambdaprime2} We have $N' < N$.
\item\label{lem:lambdaprime3} 
By the first part, we can define ${J'}^u$, $\tilde J{'}^u$, $\mathfrak S'$ for $(\lambda',\eps')$ as we defined $J^u$, $\tilde J^u$, $\mathfrak S$ for $(\lambda,\eps)$.
For $u \in \{\pm 1\}$, $\phi$ restricts to a bijection $\phi \colon J'^{-u} \cap \{1,\dots,r'\} \to \tilde J^u$.
\item\label{lem:lambdaprime4} 
Let $M = M(\lambda,\eps)$, $M' = M(\lambda',\eps')$, $k = k(\lambda,\eps)$, and $k' = k(\lambda',\eps')$.
Let $w = \sgn(M+1/2)$. 
Then $M' =  - M + \eps(1)$ and $k' = |k - \eps(1)w|$.
\end{enumerate}
\end{proposition}
\begin{proof}
\begin{enumerate}[wide]
\item 
By definition of $\lambda'$ and \Cref{lem:J}, we have
\begin{align}
\lambda_i' \geq \lambda_{\phi(i)} \geq \lambda_{\phi(i+1)}+2 \geq \lambda_{i+1}',
\end{align}
so $\lambda' \in \mathcal P$.
Clearly, $\lambda_i'$ only has odd parts and $S(\lambda') = N'$, and so $\lambda' \in \Port(N')$. Let  $i \in \{1,\dots,t(\lambda')\}$ such that $\lambda_i' = \lambda_{i+1}'$. We have to show that $\eps'(i) = \eps'(i+1)$. 
If $\phi(i), \phi(i+1) \in \tilde J^{u}$ for some $u \in \{\pm 1\}$, then $\lambda_{\phi(i)} = \lambda_{\phi(i+1)}$, which contradicts \Cref{lem:J}.
If $\phi(i) \in \tilde J^{-\eps(1)}$ and $\phi(i+1) \in \tilde J^{\eps(1)}$, then $\lambda_{\phi(i+1)} = \lambda_{i+1}' = \lambda_i' = \lambda_{\phi(i)} + 2$, which contradicts the fact that $\lambda_{\phi(i)} \geq \lambda_{\phi(i+1)}$.
Hence $\phi(i) \in \tilde J^{\eps(1)}$ and $\phi(i+1) \in \tilde J^{-\eps(1)}$. 
Then there exists an $s \in \fS$ such that $\phi(i) < s < \phi(i+1)$, and since $\phi$ is an inreasing bijection, we have $\varnothing \neq \{\phi(i)+1,\phi(i)+2,\dots, \phi(i+1)-1\} \subseteq \mathfrak S$. Thus
\begin{align}\label{eq:epsphi}
\eps(\phi(i)) = \eps(\phi(i) + 1) = \dots = \eps(\phi(i+1) - 1) = -\eps(\phi(i+1)), \quad
\end{align}
and since $\phi(i) \in J^{\eps(1)}$ and $\phi(i+1) \in J^{-\eps(1)}$, we have
$\eps'(i) = \eps(\phi(i)) = -\eps(\phi(i+1) = \eps'(i+1)$. 
\item 
Since $N>1$, we have $\lambda_1 \geq 1$.
If $r' = 0$, then $|\tilde J^{-\eps(1)}| = 0$ and so 
\begin{align}\label{eq:lemmaineq0}
\sum_{i \in \mathfrak S} \lambda_i \geq 1 + 2|\tilde J^{-\eps(1)}|.
\end{align}
Suppose that $r' \neq 0$. Then $\lambda_{\phi(r')} \geq 1$ and $r' \geq |\tilde J^{-\eps(1)}|$.
We also have $\lambda_1\geq\lambda_{\phi(1)}+2$.
Using \Cref{lem:J}, we find
\begin{align}\label{eq:lemmaineq}
\sum_{i \in \mathfrak S} \lambda_i \geq \lambda_1 \geq \lambda_{\phi(1)} + 2 \geq \lambda_{\phi(r')}  + 2(r'-1) + 2 \geq 1 + 2|\tilde J^{-\eps(1)}|.
\end{align}
Thus
\begin{align}
N' 
&= N - \sum_{i \in \mathfrak S} \lambda_{i} + 2|\tilde J^{-\eps(1)}| + \frac{1 + (-1)^{|\mathfrak S|}}{2}
\leq N - 1 + \frac{1 + (-1)^{|\mathfrak S|}}{2}
\leq N.
\end{align}
If $|\mathfrak S|$ is odd, then the last inequality is strict. If $|\mathfrak S|$ is even, then the first inequality is strict, since the first inequality of \eqref{eq:lemmaineq0} and \eqref{eq:lemmaineq} is strict. Thus we have $N' < N$.
\item Suppose $i \in J'^{u}$, i.e. $\eps'(i)(-1)^{i+1} = u$. If $\phi(i) \in \tilde J^{-\eps(1)}$, then $\eps'(i) = -\eps(\phi(i))$.
Let $h \in \N$ be the unique integer such that $s_h < \phi(i) < s_{h+1}$. Then $s_h \in J^{-\eps(1)}$, so $h$ is even by \eqref{eq:J}. Note that $\phi(i) = h + i$, so $i$ and $\phi(i)$ have the same parity. Thus $u = \eps'(i)(-1)^{i+1} = -\eps(\phi(i))(-1)^{\phi(i)+1} = \eps(1)$. If $\phi(i) \in \tilde J^{\eps(1)}$, we can similarly show that $u = -\eps(1)$. Hence we have $\phi(J'^u) \subseteq \tilde J^{-u}$. Since  $\phi \colon \{1,\dots,r'\} \to \tilde J^1 \cup \tilde J^{-1}$ is a bijection and $\{1,\dots,r'\} = J'^1\cup J'^{-1}$, the result follows.
\item 
Let $u = \eps(1)$.
Suppose $|\mathfrak S| = 2q+1$ is odd. Then $|J^u| - |\tilde J^u| = q+1$, $|J^{-u}| - |\tilde J^{-u}| = q$ and $t(\lambda') = r'$. By the previous part, we then have $|J'^{\pm u}| = |\tilde J^{\mp u}|$.
Thus
\begin{align}
uM' 
= |J'^u| - |J'^{-u}| 
= |\tilde J^{-u}| - |\tilde J^{u}|
= -(|J^u| - |J^{-u}|) + 1 
= -uM + 1.
\end{align}
Suppose $|\mathfrak S| = 2q$ is even. Then $|J^u| - |\tilde J^u| = |J^{-u}| - |\tilde J^{-u}| = q$. 
Now, $t(\lambda') = r'+1$ and $\eps'(r'+1) = (-1)^{r'}u$, i.e. $r'+1 \in J'^{u}$. So by the previous part, we have $|J'^{u}| = |\tilde J^{-u}| + 1$ and $|J'^{-u}| = |\tilde J^{u}|$.
Thus
\begin{align}
uM' 
= |J'^u| - |J'^{-u}| 
= |\tilde J^{-u}| - |\tilde J^{u}| + 1
= -(|J^u| - |J^{-u}|) + 1
= -uM + 1.
\end{align}
In both cases, we find $M' = -M + u$. Note that $k = |M| = wM$. 
Thus $k' = |M'| = |-M+u| = |-w||-M + u| = |wM -uw| = |k - uw|$. 
\qedhere
\end{enumerate}
\end{proof}

\begin{proposition}\label{prop:lambdabar}
\begin{enumerate}
\item $\bar\lambda_1$ is odd and $\bar\lambda_1 \geq \lambda_1$.
\item $(\bar\lambda,\bar\eps) \in \PPPort(N)$.
\end{enumerate}
\end{proposition}
\begin{proof}
\begin{enumerate}[wide]
\item It is easy to see from the definition that $\bar\lambda_1$ is odd. We have
\begin{align}
\bar\lambda_1 - \lambda_1
=
\sum_{i \in \mathfrak S,i\geq2} \lambda_{i} - 2|\tilde J^{-\eps(1)}| - \frac{1 + (-1)^{|\mathfrak S|}}{2}.
\end{align}
If $\mathfrak S = \{1\}$, then $\tilde J^{-\eps(1)} = \varnothing$ and so $\bar\lambda_1 - \lambda_1 = 0$. Suppose $|\mathfrak S| \geq 2$. Then $s_2 < j$ for all $j \in \tilde J^{-\eps(1)}$. 
Let $i = \min J'^{\eps(1)}$ so that $r' - i + 1\geq |J'^{\eps(1)} \cap \{i,\dots,r'\}| = |J'^{\eps(1)} \cap \{1,\dots,r'\}| = |\tilde J^{-\eps(1)}|$, where the last equality holds by \Cref{lem:lambdaprime3}. 
By definition of $i$, we have $\phi(i) = \min \tilde J^{-\eps(1)}$, and so we have $\lambda_{s_2} > \lambda_{\phi(i)}$.
\Cref{lem:J} then gives $\lambda_{s_2} \geq \lambda_{\phi(i)} + 2 \geq \lambda_{\phi(r')} + 2(r' - i) + 2 \geq 1 + 2|\tilde J^{-\eps(1)}|$, hence
$\bar\lambda_1 - \lambda_1\geq0$.
\item 
The assertion is clear if $N = 0,1$. Suppose that $N > 1$. 
By the previous part, all the terms of $\bar\lambda$ are odd. 
It suffices to show that $\bar\lambda_1 \geq \bar\lambda_1'$, and that if $\bar \eps(1) = - \bar \eps(2)$, i.e. $\eps(1) = -\eps'(1)$, then $\bar\lambda_1 > \bar\lambda_1'$. 
It then follows that $(\bar\lambda,\bar\eps) \in \PPPort(N)$.

Write $\tilde{\mathfrak S}' = \{1,\dots,r'\} \cap \mathfrak S' = \{s_1'<s_2'<\dots<s_q'\}$. Note that
\begin{align}\label{eq:JSprime}
\bigcup_{i} \{s'_{2i-1},\dots,s'_{2i}-1\} \subseteq J'^{\eps'(1)},
\quad
\bigcup_{i} \{s'_{2i},\dots,s'_{2i+1}-1\} \subseteq J'^{-\eps'(1)} .
\end{align}
Note that $\lambda'_{r'+1} = \frac{1+(-1)^{|\mathfrak S|}}{2}$ and $\mathfrak S' \subseteq \{1,\dots,r'+1\}$, so
\begin{align}\label{eq:ineqS}
\sum_{s'\in\mathfrak S'} \lambda_{s'}'
\leq
\lambda_{r'+1}' + \sum_{h'=1}^q \lambda'_{s'_{h'}}
\leq
\frac{1+(-1)^{|\mathfrak S|}}{2} + \sum_{h'=1}^q \lambda'_{s'_{h'}}.
\end{align}
Define $\psi \colon \tilde{\mathfrak S}' \to \fS$ by 
$\psi(s') = \max\{s \in \mathfrak S \colon s < \phi(s')\} = \phi(s')-1$. 
Then $\psi$ is increasing and injective.
Let $u \in \{\pm1\}$ and let $s' \in \tilde{\mathfrak S}' \cap J'^{-u}$.
By \Cref{lem:lambdaprime3}, we have $\phi(s') \in \tilde J^{u}$. 
Since $\phi(s') \notin \mathfrak S$, we have $\eps(\psi(s')) = - \eps(\phi(s'))$ and so  $\lambda_{\psi(s')} \geq  \lambda_{\phi(s')} + 2$.
Thus if $u = \eps(1)$, then $\lambda_{\psi(s')}\geq\lambda_{s'}'+2$, and if  $u = -\eps(1)$ then $\lambda_{\psi(s')}\geq\lambda_{s'}'$.
Hence we can further bound the right-hand side of \eqref{eq:ineqS}:
\begin{align}
\sum_{s'\in\mathfrak S'} \lambda_{s'}'
\leq
\frac{1+(-1)^{|\mathfrak S|}}{2} + \sum_{h'=1}^q \lambda'_{s'_{h'}}
\leq
\frac{1+(-1)^{|\mathfrak S|}}{2} + \sum_{h \in \img(\psi)} \lambda_{s_h} - 2|\tilde{\mathfrak S}' \cap J'^{-\eps(1)}|.
\end{align}
From this, we find
\begin{align}
&\bar\lambda_1 - \bar\lambda_1' 
\\
\geq 
&\sum_{s \in \mathfrak S \setminus \img(\psi)} \lambda_{s} 
+ 2|\tilde{\mathfrak S}'\cap J'^{-\eps(1)}| 
+ 2|\tilde J'^{-\eps'(1)}| 
- 2|\tilde J^{-\eps(1)}| - (1+(-1)^{|\mathfrak S|}) + \frac{1+(-1)^{|\mathfrak S'|}}{2}
=: X.
\end{align}
Suppose $\eps(1) = \eps'(1)$. 
If $\psi(1) = 1$, then $\phi(1) = 2 \in J^{\eps(1)}$, so $1 \in J'^{-\eps(1)}$ by \Cref{lem:lambdaprime3}, hence $\eps'(1) = -\eps(1) = -\eps'(1)$, a contradiction.
Thus $\psi(1) > 1$ and since $\psi$ is increasing, we have $1 \notin \img(\psi)$, and so $|\mathfrak S| \geq 2$ and 
\begin{align}
X \geq \lambda_1 - 2|\tilde J^{-\eps(1)}| - (1+(-1)^{|\mathfrak S|}).
\end{align}
We showed in the previous part that $\lambda_{s_2} \geq 1 + 2|\tilde J^{-\eps(1)}|$ when $|\mathfrak S| \geq 2$. Thus we have
\begin{align}
X 
\geq 
\lambda_{s_2} - 2|\tilde J^{-\eps(1)}| - (1 + (-1)^{|\mathfrak S|}
\geq
(-1)^{|\mathfrak S|} \geq -1,
\end{align}
and so $\bar\lambda_1 - \bar\lambda_1' \geq -1$. By the previous part, $\bar\lambda_1$ and $\bar\lambda_1'$ are odd, so we have $\bar\lambda_1 - \bar\lambda_1' \geq 0$.

Suppose $\eps(1) = -\eps'(1)$. We show that $X \geq 1$ so that $\bar\lambda_1 > \bar\lambda_2$. 
By \Cref{lem:lambdaprime3}, we have
\begin{align}
\phi(\tilde J^{-\eps(1)})
= 
J'^{\eps(1)} \cap \{1,\dots,r'\}
&=
(\tilde{\mathfrak S}' \cap J'^{\eps(1)}) 
\sqcup 
(\tilde J'^{\eps(1)}\cap\{1,\dots,r'\})
\\
&= 
(\tilde{\mathfrak S}' \cap J'^{\eps(1)}) 
\sqcup 
(\tilde J'^{\eps(1)}) 
\setminus 
(\tilde J'^{\eps(1)} \cap\{r'+1\}),
\end{align}
hence
\begin{align}
|\tilde J'^{\eps(1)}| = |\tilde J^{-\eps(1)}| + |\tilde J'^{\eps(1)} \cap \{r'+1\}| - |\tilde{\mathfrak S}' \cap J'^{\eps(1)}|.
\end{align}
Using that $\eps(1) = - \eps'(1)$, we find
\begin{align}
X 
= 
&\sum_{s \in \mathfrak S \setminus \img(\psi)} \lambda_{s}
+ 2|\tilde{\mathfrak S}'\cap J'^{\eps'(1)}| 
- 2|\tilde{\mathfrak S}' \cap J'^{-\eps'(1)}| 
\\
&+ 2|\tilde J'^{-\eps'(1)} \cap\{r'+1\}|
- (1+(-1)^{|\mathfrak S|}) 
+ \frac{1+(-1)^{|\mathfrak S'|}}{2}.
\end{align}
By \eqref{eq:JSprime}, $|\tilde{\mathfrak S}'\cap J'^{\eps'(1)}|$ (resp. $|\tilde{\mathfrak S}'\cap J'^{-\eps'(1)}|$) is the number of odd (resp. even) integers in $\{1,\dots,q\}$. 

Suppose $q$ is odd. Then
\begin{align}\label{eq:diffodd}
2|\tilde{\mathfrak S}'\cap J'^{\eps'(1)}| 
- 2|\tilde{\mathfrak S}' \cap J'^{-\eps'(1)}| 
=
2.
\end{align}
Suppose $|\mathfrak S|$ is odd. Then $1+(-1)^{|\mathfrak S|} = 0$, so $X \geq 2$. 
Suppose $|\mathfrak S|$ is even. Then $1+(-1)^{|\mathfrak S|} = 2$, $\lambda'_{r'+1}=1$, and by definition of $\eps'(r'+1)$, we have $r'+1 \in J'^{\eps(1)} = J'^{-\eps'(1)}$. Then either $r'+1 \in \tilde J'^{-\eps'(1)}$, in which case we have $|\tilde J'^{-\eps'(1)} \cap\{r'+1\}| = 1$, and so $X \geq 2 + 2 - 2 = 2$, or $r'+1 \in \mathfrak S'$, in which case $|\mathfrak S'| = q + 1$ is even, and so $X \geq 2 - 2 + \frac{1+(-1)^{|\mathfrak S'|}}{2} = 1$.

Suppose $q$ is even. Then
\begin{align}\label{eq:diffeven}
2|\tilde{\mathfrak S}'\cap J'^{\eps'(1)}| 
- 2|\tilde{\mathfrak S}' \cap J'^{-\eps'(1)}| 
=
0.
\end{align}
Suppose $|\mathfrak S|$ is odd. Then $\lambda_{r'+1} = 0$, so $|\mathfrak S'| = |\tilde{\mathfrak S}'| = q$ is even.
Thus $1+(-1)^{|\mathfrak S|} = 0$ and $(1+(-1)^{|\mathfrak S'|})/2 = 1$ and so $X \geq 1$.
Suppose $|\mathfrak S|$ is even. Then $\lambda_{r'+1}=1$ and by definition of $\eps'(r'+1)$, we have $r'+1 \in J'^{\eps(1)} = J'^{-\eps'(1)}$. Since $q$ is even, we have $s_q' \in  J'^{-\eps'(1)}$ by \eqref{eq:JSprime}, and so $r'+1 \notin \mathfrak S'$, thus $r'+1 \in \tilde J'^{-\eps'(1)}$. Hence we have $|\mathfrak S'| = |\tilde{\mathfrak S}'| = q$ is even, so $X \geq 2|\tilde J'^{-\eps'(1)} \cap\{r'+1\}| - (1+(-1)^{|\mathfrak S|}) + \frac{1+(-1)^{|\mathfrak S'|}}{2} = 2 - 2 + 1 = 1$.

We conclude that $\bar\lambda_1 > \bar\lambda_2$ if $\eps(1) = -\eps'(1)$, as desired.
\qedhere
\end{enumerate}
\end{proof}

\subsection{Proof of \Cref{thm:algorithm}}
\begin{proof}
We first set up some notation.
Let $\eta \in \{\pm1\}$. For $(x,y) \in \cR \times \cR$ or $(x,y) \in \Z \times \Z$, let $(x^\eta,y^\eta) := (x,y)$ if $\eta=1$, and $(x^\eta,y^\eta) := (y,x)$ if $\eta=-1$. 
Let 
\begin{align}
\phi^{(\eta)} \colon \{1,\dots,|J^\eta|\} \to J^\eta
\end{align}
be the unique increasing bijection.
Let $(A,B) = (A_{\lambda,\eps},B_{\lambda,\eps})$. 
Let $(\alpha,\beta) \in \cP_2(n)$ so that $(\alpha,\beta)_k = \Phi_N(\lambda,[\eps])$ such that if $k=0$, then 
$\alpha_1 = A_1$ and $\beta_1 = B_1$, which uniquely defines $\alpha$ and $\beta$, since $\lambda$ only has odd parts, and so $A_1 \neq B_1$. 
Let $w = \sgn(M+1/2)$ so that $w = \sgn M$ if $M \neq 0$ and $w = 1$ if $M = 0$.
Using \eqref{eq:gscodd0}, we find
\begin{align}\label{eq:gscsymbols}
A &= ((\lambda_j+1) / 2 + 1 - j \colon j \in J^w) \cup (-t(\lambda) - 2i+2 \colon i \in \N) = (\alpha_i + k + 2 -2i \colon i \in \N),
\\
B &= ((\lambda_j+1)/2 + 1 - j \colon j \in J^{-w}) \cup (-t(\lambda) - 2j+2 \colon j \in \N) = (\beta_j - k + 2 -2j \colon j \in \N).
\end{align}
We note particularly that for $k = 0$, we defined $\alpha$, $\beta$ and $w$ in such a way so that this holds.
Thus for any $\eta \in \{\pm 1\}$ and for $i = 1,\dots,|J^{\eta w}|$, $j = 1,\dots,|J^{-\eta w}|$, and $h \in \N$, we have
\begin{align}\label{eq:gscodd}
A_i^\eta &= \frac{\lambda_{\phi^{(\eta w)}(i)}+1}{2} - \phi^{(\eta w)}(i) + 1 = \alpha_i^\eta + \eta k + 2 - 2i,
\\
B_j^\eta &= \frac{\lambda_{\phi^{(-\eta w)}(j)}+1}{2} - \phi^{(-\eta w)}(j) + 1 = \beta_j^\eta -\eta k + 2 - 2j,
\\
A_{|J^{\eta w}|+h}^\eta = B_{|J^{-\eta w}|+h}^\eta &= -t(\lambda) - 2(h-1).
\label{eq:gsceven}
\end{align}
Let $u = \eps(1)$ and $v = uw$. Then for $i=1,\dots,|J^u|$, $j=1,\dots,|J^{-u}|$, we have
\begin{align}\label{eq:gscod2}
A^v_i &= \frac{\lambda_{\phi^{(u)}(i)}+1}{2} - \phi^{(u)}(i) + 1 = \alpha_i^v + vk + 2 - 2i,
\\
B^v_j &= \frac{\lambda_{\phi^{(-u)}(j)}+1}{2} - \phi^{(-u)}(j) + 1 = \beta_j^v - vk + 2 -2j.
\end{align}
Let
$m_0 = \frac{t(\lambda)+k}{2}$
and
$m_1 = \frac{t(\lambda)-k}{2}$.
As noted in \Cref{rem:symbolodd}, $m_0$ is the smallest integer such that $m_0 \geq t(\alpha)$ and $m_0 - k \geq t(\beta)$. 
Thus we can consider the order $(m_0,m_1,<_{\alpha,\beta,k})$. We also consider the order $(m_0^v,m_1^v,<_{\alpha^v,\beta^v,vk})$ on the indices of $(\alpha^v,\beta^v)$.
We denote $<_{\alpha,\beta,k}$ by $<$ and $<_{\alpha^v,\beta^v,vk}$ by $<^v$.
Since $\lambda$ only has odd parts, $P_{k,-k;2}(\alpha,\beta,<)$ has a unique element $(\alphamax,\betamax)$ as shown in the proof of \Cref{thm:max}. We then have $\{(\alphamax^{,v},\betamax^{,v})\} = P_{vk,-vk;2}(\alpha^v,\beta^v,<^v)$. 
We have $(1,0) <^v (1,1)$, since the largest term of $A^v\sqcup B^v$ is $\frac{\lambda_1+1}{2} = \frac{\lambda_{\phi^{(u)}(1)} + 1}{2} = A_1^v$, so
applying procedure (a) to $(\alpha^v,\beta^v)$ gives us a bipartition $(\bm\alpha,\bm\beta)$ with an order on the set of its indices equivalent to $<_{\bm\alpha,\bm\beta,vk-1}$ by \Cref{lemma:unique}, and positive integers  $\bar\mu_1, a_1,a_2,\dots,a_x, b_1,b_2,\dots,b_y$ 
such that
\begin{align}
(\alphamax^{,v},\betamax^{,v}) &= ((\bar\mu_1)\sqcup \bm \mu,\bm\nu),
\end{align}
where $\{(\bm\mu,\bm\nu)\} = P_{vk-2,-vk;2}(\bm\alpha,\bm\beta,<_{\bm\alpha,\bm\beta,vk-1}) = P_{vk-1,-vk+1;2}(\bm\alpha,\bm\beta,<_{\bm\alpha,\bm\beta,vk-1})$.
From \eqref{eq:gscod2} and by definition of the $a_i$, we see that $\mathfrak S \cap J^u = \{\phi(a_1),\dots,\phi(a_x)\}$ and so $\phi^{(u)}(a_i) = s_{2i-1}$ for $i=1,\dots,x$. Similarly, $\phi^{(-u)}(b_j) = s_{2j}$ for $j=1,\dots,y$.

Having set up the notation, we now prove the theorem by induction on $N$. The theorem is obvious for $N=0,1$, so assume $N>1$.
\begin{enumerate}[wide]
\item We will show that $\bar\lambda_1 = (\lambdamax)_1$. 
Let $q = |\fS|$. Using \eqref {eq:gscod2}, we find that 
\begin{align}
\bar\mu_1 
&= \sum_{h=1}^{\lceil q/2\rceil} \alpha_{a_h}^v + \sum_{h=1}^{\lfloor q/2\rfloor} \beta_{b_h}^v
\\
&= \sum_{h=1}^{\lceil q/2\rceil} \frac{\lambda_{s_{2h-1}}+1}{2} - s_{2h-1} - 1 - vk + 2a_h + \sum_{h=1}^{\lfloor q/2\rfloor} \frac{\lambda_{s_{2h}}+1}{2} - s_{2h} - 1 + vk + 2b_h.
\end{align}
Suppose first that $q$ is odd. Then 
\begin{align}
\bar\mu_1 
&= 
-s_1 - 1 - vk + 2a_1 + \sum_{h=1}^q \frac{\lambda_{s_h}+1}{2} + \sum_{h=1}^{\frac{q-1}{2}} (2a_{h+1}-s_{2h+1}+2b_h-s_{2h}-2).
\end{align}
We have $2a_1 - s_1 - 1 = 2 - 1 - 1 = 0$. Since $a_{h+1} = |\{i \in J^u \colon i \leq s_{2h+1}\}|$ and $b_h = |\{i \in J^{-u} \colon i \leq s_{2h}\}|$, we also have $a_{h+1} + b_h = |\{1,\dots,s_{2h}\}\cup\{s_{2h+1}\}| = s_{2h}+1$.
Thus 
\begin{align}
\bar\mu_1
=
-vk + \sum_{h=1}^q \frac{\lambda_{s_h}+1}{2} + \sum_{h=1}^{\frac{q-1}{2}} s_{2h}-s_{2h+1}
&=
\frac{1}{2} -vk + \sum_{h=1}^q \frac{\lambda_{s_h}}{2} + \sum_{h=1}^{\frac{q-1}{2}} (s_{2h}-s_{2h+1}+1 )
\\
&=
\frac{1}{2} -vk + \sum_{h=1}^q \frac{\lambda_{s_h}}{2} -|\tilde J^{-u}|.
\end{align}
\Cref{prop:symbolterm} implies that $(\lambdamax)_1$ is odd, so we also have $\bar \mu_1 = \frac{(\lambdamax)_1+1}{2} - 2vk$, hence
\begin{align}
(\lambdamax)_1 
&= 2\bar\mu_1 + 2vk -1
= \sum_{h=1}^q \lambda_{s_h} - 2|\tilde J^{-u}| = \bar\lambda_1.
\end{align}
Now suppose that $q$ is even. Then
\begin{align}
\bar\mu_1
&=
\left(\sum_{h=1}^q \frac{\lambda_{s_h}+1}{2}\right) + \sum_{h=1}^{q/2} (2a_{h}-s_{2h-1}+2b_h-s_{2h}-2).
\end{align}
Similarly as above, we find $a_h + b_h = |\{1,\dots,s_{2h-1}\}\cup\{s_{2h}\}| = s_{2h-1}+1$.
Thus 
\begin{align}\label{eq:nu1}
\bar\mu_1
=
\left(\sum_{h=1}^q \frac{\lambda_{s_h}+1}{2}\right) + \sum_{h=1}^{q/2} (s_{2h-1}-s_{2h})
&=
\left(\sum_{h=1}^q \frac{\lambda_{s_h}}{2}\right) + \sum_{h=1}^{q/2} (s_{2h-1}-s_{2h}+1)
\\
&=
\left(\sum_{h=1}^q \frac{\lambda_{s_h}}{2}\right) - |\tilde J^{u}|.
\end{align}
Since $q$ is even, we have 
$|J^{\pm1}| = |\tilde J^{\pm1}| + q/2$, and so
$M = |J^1| - |J^{-1}| = |\tilde J^1| - |\tilde J^{-1}|$. 
Thus we have $vk = uwk = uM = |\tilde J^u| - |\tilde J^{-u}|$, so by \eqref{eq:nu1}, we have
\begin{align}
\bar\mu_1
&=
\left(\sum_{h=1}^q \frac{\lambda_{s_h}}{2}\right) - vk - |\tilde J^{-u}|,
\end{align}
hence
\begin{align}
(\lambdamax)_1 
&= 2\bar\mu_1 + 2vk - 1
= -1 + \sum_{h=1}^q \lambda_{s_h} - 2|\tilde J^{-u}| = \bar\lambda_1.
\end{align}
\item\label{proofpart3} 
We will show that $\Phi_{N'}(\lambda',[\eps']) = (\bm{\alpha}^{-uw'},\bm{\beta}^{-uw'})_{k'}$.
Let $m_u' := |\tilde J^{u}| = |J'^{-u} \cap \{1,\dots,r'\}|$ and
let $\iota\colon \{1,\dots,m_u'\} \to \N \setminus \{a_1,\dots,a_x\}$ be the unique increasing injection. 
Define
\begin{align}
\psi'^{(u)} = \phi^{-1} \circ \phi^{(u)} \circ \iota.
\end{align}
Then $\psi'^{(u)}$ is the unique increasing bijection from $\{1,\dots,m_u'\}$ to $J'^{-u}\cap\{1,\dots,r'\}$. 
Let $j \in \{1,\dots,m_u'\}$ and $q = \psi'^{(u)}(j)$. For $h := \iota(j) - j$ we have $a_h < \iota(j) < a_{h+1}$, and applying $\psi'^{(u)}$ to the three terms gives $s_{2h-1} <\phi(q) < s_{2h+1}$.
Thus we have $\phi(q) = q + 2h -1$. Using \eqref {eq:gscod2}, we find
\begin{align}
\bm\alpha_j
= 
\alpha_{\iota(j)}^v
&=
\frac{\lambda_{\phi(q)}+1}{2} - \phi(q) + 1 - (vk + 2 - 2\iota(j))
\\
&=
\frac{\lambda_{\phi(q)}+1}{2} - (q+2h-1) + 1 - (vk + 2 - 2j-2h)
\\
&=
\frac{\lambda_{\phi(q)}+1}{2} + 2j - vk -q.
\end{align}
Since $\phi(q) \in J^u$, we have $\lambda_{\psi'^{(u)}(j)}'  = \lambda_{\phi(q)}$, and so
\begin{align}\label{eg:alphapsymbol}
\bm \alpha_j + (vk-1) + 2 - 2j
&= \frac{\lambda_{\phi(q)}+1}{2} + 1 - q 
= \frac{\lambda_{\psi'^{(u)}(j)}'+1}{2} + 1 - \psi'^{(u)}(j).
\end{align}
Let $u' = \eps'(1)$, $w' = \sgn(M' + 1/2)$, $v' = u'w'$ and $(A',B') = (A_{\lambda',\eps'},B_{\lambda',\eps'})$. Let $\phi'^{(\pm 1)} \colon \{1,\dots,|J'^{\pm 1}|\} \to J'^{\pm 1}$ be the unique increasing bijections. We have $\phi'^{(-u)}|_{\{1,\dots,m_{u}'\}} = \psi'^{(u)}$, and by the analogue of \eqref{eq:gscodd} for $(\lambda',\eps')$, we have $A_j'^{-uw'} = \frac{\lambda'_{\phi'^{(-u)}(j)}+1}{2} + 1 - \phi'^{(-u)}(j)$ for $j=1,\dots,m_u'$, so by \eqref{eg:alphapsymbol}, we get
\begin{align}
A_j'^{-uw'} 
= \frac{\lambda'_{\phi'^{(-u)}(j)}+1}{2} + 1 - \phi'^{(-u)}(j)
= \frac{\lambda'_{\psi'^{(u)}(j)}+1}{2} + 1 - \psi'^{(u)}(j)
= \bm \alpha_j + (vk-1) + 2 - 2j.
\end{align}
Thus the first $m_u'$ terms of $\bm\alpha + [vk-1,-\infty[_2$ are equal to the first $m_u'$ terms of $A'^{-uw'}$. 
By similar arguments, it holds that the first $m_{-u}'$ terms of $\bm\beta + [-vk+1,-\infty[_2$ are equal to the first $m_{-u}'$ terms of $B'^{-uw'}$. 
Let $h \in \N$.
We consider two cases.
Suppose $|\mathfrak S|$ is even. Then $r'+1 \in J'^{u}$ and $t(\lambda') = r'+1$, so $|J'^{u}| = m_{-u}'+1$, and $|J'^{-u}| = m_u'$.
By the analogue of \eqref{eq:gsceven} for $(\lambda',\eps')$, we have
\begin{align}\label{eq:Auw}
&(A'^{-uw'})_{m_u' + h} 
= (A'^{-uw'})_{|J'^{-u}| + h}
= -t(\lambda') - 2(h-1)  
= -r' + 1 - 2h,
\\
&(B'^{-uw'})_{m_{-u}' + 1 + h} 
= (B'^{-uw'})_{|J'^{u}| + h}
= -t(\lambda') - 2(h-1)  
= -r' + 1 - 2h.
\end{align}
We have $r' = |\tilde J^u| + |\tilde J^{-u}| = m_u' + m_{-u}'$, and since $\mathfrak S$ is even, we have
$m_u' - m_{-u}' = |\tilde J^u| - |\tilde J^{-u}| = |J^u| - |J^{-u}| = vk$. Thus $r' = -vk + 2m_u' = vk + 2m_{-u}'$. Since $\bm\alpha_{m_u'+h} = \bm\beta_{m'_{-u}+1+h} = 0$, we find
\begin{align}
\bm\alpha_{m_u'+h} + (vk-1) + 2 - 2m_u' - 2h &= -r' + 1 -2h 
\\&\quad\quad= \bm\beta_{m_{-u}' + 1 + h} + (1-vk) - 2m_{-u}' - 2h.
\end{align}
It remains to consider the ($m'_{-u}+1$)-th term of $B'^{-uw'}$: we have $\bm\beta_{m'_{-u}+1} = 0$ and $\lambda'_{r'+1} = 1$, so 
\begin{align}
\bm\beta_{m'_{-u}+1} + (1-vk) - 2m_{-u}' = 1 - r' =  \frac{\lambda'_{r'+1}+1}{2}  - r' = (B'^{-uw'})_{m'_{-u}+1}.
\end{align}
Thus we have shown that if $|\mathfrak S|$ is even, then $(\bm\alpha + [vk-1,-\infty[_2,\bm\beta+[-vk+1,-\infty[_2) = (A'^{-uw'},B'^{-uw'})$.

Suppose $|\mathfrak S|$ is odd. Then $\lambda'_{r'+1} = 0$, so $|J'^u| = m'_{-u}$, and $|J'^{-u}| = m'_u$. By similar arguments as in the previous case, we have $t(\lambda') = r'$, $r' = -vk+1 + 2m'_u = vk - 1 + 2m_{-u}'$, so for all $h \in \N$, we have
\begin{align}
(A'^{-uw'})_{m_u' + h} = (B'^{-uw'})_{m_{-u}' + h} &= -r' + 2 - 2h,
\\
\bm\alpha_{m_u'+h} + (vk-1) + 2 - 2m_u' - 2h 
= \bm\beta_{m_{-u}' + h} + (1-vk) + 2 - 2m_{-u}' - 2h 
&= -r' + 2 -2h.
\end{align}
Thus we have shown that 
\begin{align}\label{eq:kprime}
(\bm\alpha + [vk-1,-\infty[_2,\bm\beta+[-vk+1,-\infty[_2)  = (A'^{-uw'},B'^{-uw'}).
\end{align}
Note that $-uw'(vk-1) = -ww'(k-v)$, so this gives
\begin{align}\label{eq:symbolprime}
(\bm\alpha^{-uw'} + [-ww'(k-v),-\infty[_2,\bm\beta^{-uw'}+[ww'(k-v),-\infty[_2) = (A',B').
\end{align}
We need to show that $k' = -ww'(k-v)$; it then follows from \eqref{eq:symbolprime} that $\Phi_{N'}(\lambda',[\eps']) = (\bm{\alpha}^{-uw'},\bm{\beta}^{-uw'})_{k'}$, as desired.
By \Cref{lem:lambdaprime4}, we have $k' = |k-v|$, so we need to show that $|k-v| = -ww'(k-v)$ 
This is obvious if $k-v = 0$.
Suppose $k - v \geq 1$. Then $M \neq 0$ and $M' \neq 0$. Since $M = -M' + u$ by \Cref{lem:lambdaprime4}, we have $\sgn M = - \sgn M'$, so $w = - w'$, and thus $-ww'(k-v) = k-v = |k-v|$.
Suppose $k-v = -1$. Then $k = 0$ and $v=1$. 
From $k = 0$, we get $M = 0$ and so $w = 1$.
Thus $u = v/w = 1$. 
Next, we get $M' = -M + u = 1$, so $w' = 1$.
Thus $-ww'(k-v) = |k-v|$.

\item\label{pf:alg5} 
We will show that $(A_{\bar\lambda,\bar\eps},B_{\bar\lambda,\bar\eps})_k = (A_{\lambdamax},B_{\epsmax})_k$,
hence $(\bar\lambda,[\bar\eps]) = (\lambdamax,[\epsmax])$. 
We have
\begin{align}\label{eq:maxsymbol}
(A_{\lambdamax,\epsmax}^v,B_{\lambdamax,\epsmax}^v)_k
= 
((\bar\mu_1 + vk) \sqcup (\bm\mu + [vk-2,-\infty[_2), \bm\nu + [-vk,-\infty[_2)_k.
\end{align}
Recall that $\{(\bm\mu,\bm\nu)\} = P_{vk-1,-vk+1;2}(\bm\alpha,\bm\beta,<_{\bm\alpha,\bm\beta,vk-1})$.
We just saw that $-uw'(vk-1) = k'$, so
$\{(\bm\mu^{-uw'},\bm\nu^{-uw'})\} = P_{k',-k';2}(\bm\alpha^{-uw'},\bm\beta^{-uw'},<_{\bm\alpha^{-uw'},\bm\beta^{-uw'},k'})$. So by the induction hypothesis, we have $(\bar\lambda', [\bar \eps']) = (\lambda'^{\text{max}},[\eps'^{\text{max}}])$, $M(\bar\lambda',\bar\eps') = M(\lambda',\eps')$, and
\begin{align}\label{eq:ABbarprime}
A_{\bar\lambda',\bar\eps'}
=
\bm\mu^{-uw'} + [k',-\infty[_2
\quad\text{and}\quad
B_{\bar\lambda',\bar\eps'}
=
\bm\nu^{-uw'} + [-k',-\infty[_2.
\end{align}
Let $\bar J^{\pm 1} = \{i \in \N \colon \bar\eps(i)(-1)^{i+1} = \pm 1\}$ and $\bar J'^{\pm 1} = \{i \in \N \colon \bar\eps'(i)(-1)^{i+1} = \pm 1\}$.
Note that $M(\bar\lambda',\bar\eps') = M(\lambda',\eps')$ implies that $w' = \sgn(M(\bar\lambda',\bar\eps')+1/2)$.
By \Cref{prop:lambdabar}, $\bar\lambda'$ and $\bar\lambda$ only have odd parts, so by the analogue of \eqref{eq:gscsymbols} for $(\bar\lambda',\bar\eps')$, it holds that the largest $|\bar J'^{-u}|$ terms of $\bm\mu + [vk-1,-\infty[_2 = A_{\bar\lambda',\bar\eps'}^{-uw'}$ are the largest $|\bar J'^{-u}|$ terms of
\begin{align}
&\left(\frac{\bar\lambda_j'+1}{2} + 1 -j \colon j \in \N, (-1)^{j+1} \bar\eps'(j) = -uw'w' = -u\right)
\\
=& 
\left(\frac{\bar\lambda_{j+1}+1}{2} + 1 - j \colon j \in \N, (-1)^{j+1}\bar\eps(j+1) =-u\right)
\\
=& 
\left(\frac{\bar\lambda_{j}+1}{2} + 2 - j \colon j \in \N_{\geq 2}, (-1)^{j} \bar\eps(j) = -u\right)
\\
=& 
\left(\frac{\bar\lambda_{j}+1}{2} + 2 - j \colon j \in \N_{\geq 2}, (-1)^{j+1} \bar\eps(j) = u\right).  
\end{align}
Thus the largest $|\bar J'^{-u}|+1$ terms of $(\bar\mu_1 + vk) \sqcup (\bm \mu + [vk - 2, -\infty[_2)$ are the largest $|\bar J'^{-u}|+1$ terms of
\begin{align}\label{eq:barA}
\left(\frac{\bar\lambda_1+1}{2}\right) 
\sqcup 
\left(\frac{\bar\lambda_{j}+1}{2} + 1 - j \colon j \in \N_{\geq 2}, (-1)^{j+1} \bar\eps(j) = u\right).
\end{align}
Note that $\bar\eps(1) = \eps(1) = u$ and recall that $M(\bar\lambda',\bar\eps') = M(\lambda',\eps')$. We also have $M(\lambda',\eps') = -M(\lambda,\eps) + u$ by \Cref{lem:lambdaprime4}. By definition of $(\bar\lambda,\bar \eps)$, we have $M(\bar\lambda,\bar\eps) = -M(\bar\lambda',\bar\eps') + u$, hence we have $M(\bar\lambda,\bar\eps) = M(\lambda,\eps)$ (and so also $k(\bar\lambda,\bar\eps) = k$).
Thus if we define $\bar u := \bar \eps(1)$, $\bar w = \sgn(M(\bar\lambda,\bar\eps)+1/2)$ 
and $\bar v := \bar u \bar w$, we find that $v = \bar v$, $u = \bar u$ and $w = \bar w$.
Thus by the analogue of \eqref{eq:gscod2} for $(\bar\lambda,\bar\eps)$, we see that the largest $|\bar J^u|$ terms of \eqref{eq:barA} are equal to the first $|\bar J^{u}|$ terms of $A_{\bar\lambda,\bar\eps}^v$. By definition of $(\bar\lambda,\bar\eps)$, we see that $|\bar J^u| = |\bar J'^{-u}| + 1$, so the largest $|\bar J^u|$ terms of $A_{\bar\lambda,\bar\eps}^v$ are equal to the largest $|\bar J^u| = |\bar J'^{-u}|+1$ terms of $(\bar\mu_1 + vk) \sqcup (\bm \mu + [vk - 2, -\infty[_2)$.

By similar arguments as above, we find that the largest $|\bar J^{-u}|= |\bar J'^{u}|$ terms of $(\bm \nu + [-vk,-\infty[_2)$ are equal to the largest $|\bar J'^{u}|$ terms of
\begin{align}\label{eq:barB}
\left(\frac{\bar\lambda_{j}+1}{2} + 1 - j \colon j \in \N, (-1)^{j+1} \bar\eps(j) = -u\right),
\end{align}
hence equal to the largest $|\bar J^{-u}|$ terms of $B_{\bar\lambda,\bar\eps}^v$.

Clearly, we have $t(\bar\lambda) = t(\bar\lambda') + 1$. 
Let $h \in \N$ and $i := |\bar J^u| + h = |\bar J'^{-u}| + h + 1$. Then
\begin{align}
(A_{\bar\lambda,\bar\eps}^v)_{i} = -t(\bar\lambda) - 2(h - 1).
\end{align}
Note that $\bm\mu_{i+h-1} \leq \bm \mu_{|\bar J'^{-u}|+1} = 0$, so the $i^{\text{th}}$ term of $(\bar\mu_1 + vk) \sqcup (\bm \mu + [vk - 2, -\infty[_2)$ is 
\begin{align}
\bm\mu_{i - 1} + (vk-2) + 2 - 2(i - 1) = vk - 2|\bar J^{u}| -2h + 2.
\end{align}
We have $t(\bar\lambda) = |\bar J^u| + |\bar J^{-u}|$.
We showed earlier that $\bar u = u$, $\bar v = v$,$\bar w = w$, and $k(\bar\lambda,\bar\eps) = k$, so we have $vk = |\bar J^u| - |\bar J^{-u}|$, hence $t(\bar\lambda) = 2|\bar J^u| - vk$. Thus $(A_{\bar\lambda,\bar\eps}^v)_{i}$ is the $i^{\text{th}}$ term of $(\bar\mu_1 + vk) \sqcup (\bm \mu + [vk - 2, -\infty[_2)$.
Hence we now have
\begin{align}\label{eq:Av}
A_{\bar\lambda,\bar\eps}^v = (\bar\mu_1 + vk) \sqcup (\bm \mu + [vk - 2, -\infty[_2). 
\end{align}
Similarly, for $j := |\bar J^{-u}| + h = |\bar J'^{u}| + h$, we find that $(B_{\bar\lambda,\bar\eps}^v)_{j}$ is the $j^{\text{th}}$ term of $(\bm \nu + [-vk,-\infty[_2)$, hence 
\begin{align}\label{eq:Bv}
B_{\bar\lambda,\bar\eps}^v = \bm \nu + [-vk,-\infty[_2.
\end{align} 
By \eqref{eq:maxsymbol}, we then have
\begin{align}
(A_{\bar\lambda,\bar\eps},B_{\bar\lambda,\bar\eps})_k = (A_{\lambdamax},B_{\epsmax})_k.
\end{align}
We conclude that $(\bar\lambda,[\bar\eps]) = (\lambdamax,[\epsmax])$.
Note that $((\bar\mu_1)\sqcup\bm\mu,\bm\nu) = ((\alphamax)^v,(\betamax)^v)$, so we also have
 $(A_{\bar\lambda,\bar\eps},B_{\bar\lambda,\bar\eps}) = (\alphamax + [k,-\infty[_2,\betamax + [-k,-\infty[_2)$ by \eqref{eq:Av} and \eqref{eq:Bv}.
We showed earlier that $M(\bar\lambda,\bar\eps) = M(\lambda,\eps)$, and so we completed the proof.
 \qedhere
\end{enumerate}
\end{proof}

\begin{corollary}
It holds that $\lambdamax = \bar\lambda$ only has odd parts.
\end{corollary}

\subsection{Example}
\begin{example}
Let $(\lambda,\eps) \in \PPPort(2n+1)$ for some $n \in \Z_{\geq 0}$ with $\lambda_1 \geq \lambda_2 \geq \dots \geq \lambda_7 > \lambda_8 = 0$. Write $\eps = (\eps(1) \, \eps(2) \, \dots ) = (-++-+-+)$. 
We compute $(\lambda^1,\eps^1) := (\lambda',\eps'), (\lambda^2,\eps^2) := (\lambda'', \eps''), \dots$. For each $j \in \N$ for which $(\lambda^j,\eps^j)$ is defined, we define $\tilde \eps^j = (\eps^j(1),-\eps^j(2),\eps^j(3),\dots)$. Following the algorithm, we find
\begin{table}[h!]
\centering
\begin{tabular}{r l c c c c c c r r r r r r r}
$\lambda =$ & $\lambda_1$ & $\lambda_2$ & $\lambda_3$ & $\lambda_4$ & $\lambda_5$ & $\lambda_6$ & $\lambda_7$ &
\\
$\eps =$ & $-$ & $+$ & $+$ & $-$ & $+$ & $-$ & $+$ &
\\
$\tilde \eps =$ & $-$ & $-$ & $+$ & $+$ & $+$ & $+$ & $+$ &
\\
\\
$\lambda^1 =$ & $$ & $\lambda_2$ & & $\lambda_4+2$ & $\lambda_5+2$ & $\lambda_6+2$ & $\lambda_7+2$ & $1$
\\
$\eps^1 =$ & $$ & $+$ & & $+$ & $-$ & $+$ & $-$ & $+$
\\
$\tilde \eps^1 =$ & $$ & $+$ & & $-$ & $-$ & $-$ & $-$ & $-$
\\
\\
$\lambda^2 =$ & $$ & & & & $\lambda_5+4$ & $\lambda_6+4$ & $\lambda_7+4$ & $3$ & $1$
\\
$\eps^2 =$ & $$ & & & & $+$ & $-$ & $+$ & $-$ & $+$
\\
$\tilde \eps^2 =$ & $$ & & & & $+$ & $+$ & $+$ & $+$ & $+$
\end{tabular}
\end{table}
\begin{align}
&\bar\lambda_1 = \lambda_1 + \lambda_3 - 9, &&\bar \eps(1) = -1,
\\
&\bar\lambda_2 = \lambda_2 + (\lambda_4 + 2) - 9, &&\bar \eps(2) = 1.
\end{align}
Next, we have $(\bar\lambda^2,\bar\eps^2) = (\lambda^2,\eps^2)$, since for any $N \in \N$ and $(\bm{\lambda},\bm{\varepsilon}) \in \PPPort(N)$ such that $\bm{\lambda}$ only has odd parts and $\bm{\varepsilon}(i) = (-1)^{i+1}$ for all $i \in \{1,\dots,t(\bm{\lambda})\}$, it is easy to see from the algorithm that $(\bar{\bm{\lambda}},\bar{\bm{\varepsilon}}) = (\bm{\lambda},\bm{\varepsilon})$. Finally, we have $\bar\lambda = (\bar\lambda_1,\bar\lambda_2)\sqcup\lambda^2$ and $\bar\eps= (\bar\eps(1),\bar\eps(2),\eps^2(1),\eps^2(2),\eps^2(3),\eps^2(4),\eps^2(5)) = (-++-+-+)$.
\end{example}

\section{Some results for $\Sp(2n)$ and $\SO(N)$}\label{sec:extra}

A unipotent class $U \in \cU$ is called quasi-distinguished if for any $u \in U$, there exists a semisimple element $s \in G$ such that $su = us$ and the centraliser $Z_G(su)$ of $su$ does not contain any non-trivial torus. A unipotent class in $\SO(N)$ (resp. $\Sp(2n)$) is quasi-distinguished if and only if the orthogonal (resp. symplectic) partition that parametrises $U$ contains only odd (resp. even) parts, and each part occurs with multiplicity at most $2$.

\begin{proposition}\label{prop:ellipticorth}
Let $N \in \N$ and $(\lambda,[\eps]) \in \PPort(N)$ such that $\lambda$ only has odd parts and let $(\bar\lambda,\bar\eps) = (\lambdamax,\epsmax)$.
Then for each $i \in \N$, we have $\mult_{\bar\lambda}(i) \leq 2$, i.e. $\bar\lambda$ parametrises a quasi-distinguished unipotent class of $\SO(N)$.
\end{proposition}
\begin{proof}
Since $\bar \lambda = (\bar\lambda_1) \sqcup \bar \lambda' = (\bar\lambda_1,\bar\lambda_2) \sqcup \bar \lambda'' = \dots$, it suffices to show that if $\bar \lambda_1 = \bar \lambda_2$, then $\bar \lambda_2 \neq \bar \lambda_3$. Suppose $\bar \lambda_1 = \bar \lambda_2$.
Let $k = k(\lambda,\eps) = k(\bar\lambda,\bar\eps)$, $n = (N-k^2)/2$, and suppose $(\alpha,\beta)_k = \Phi_N(\lambda,[\eps])$ and $(\mu,\nu)_k = \Phi_N(\bar\lambda,[\bar\eps])$.
Write $\bar \Lambda = \Lambda_{k,-k;2}(\mu,\nu)$. Denote $<_{\alpha,\beta,k}$ by $<$.
Suppose that $(1,0) < (1,1)$. Then $\bar\Lambda_1 = \mu_1 + k$. 
Since $\bar \lambda_1 = \bar \lambda_2$ is odd, we have $\bar\Lambda_1 = \bar\Lambda_2 + 1$. Since $\mu_2 + k -2 \leq \bar\Lambda_1 - 2$, we cannot have $\bar\Lambda_2 = \mu_2 + k -2$, so $\bar\Lambda_2 = \nu_1 - k$, and so $\mu_1 + k = \nu_1 - k + 1$, i.e. $\mu_1 = \nu_1 - 2k + 1$. Note that $\mu_1 - \nu_1 \geq \alpha_1$ (this is easy to see from the definition of $\mu_1$ and $\nu_1$ via procedure (a) and (b)), so $\alpha_1 \leq -2k + 1$, so $\alpha_1 + 2k \leq 1$. 
Since $(1,0) < (1,1)$ and $(\alpha,\beta) \in H(n,k)$, we have $\alpha_1 + k > \beta_1 - k$,
hence $\beta_1 < \alpha_1 + 2k \leq 1$, and thus $k=0$, $\beta_1 = 0$, and we have $\alpha_1 = 0$ or $\alpha_1 = 1$. Since $(\alpha,\beta) \in \cP_2(n)$, we must have $(\alpha,\beta) = ((1^n),\varnothing)$ or $(\alpha,\beta) = (\varnothing,\varnothing)$, and so $(\mu,\nu) = ((n),\varnothing)$ or $(\mu,\nu) = (\varnothing,\varnothing)$, hence $\bar\lambda = (N)$ or $\bar\lambda = \varnothing$, which contradicts the assumption $\bar\lambda_1 = \bar\lambda_2$ and $N \geq 1$.

Hence we have $(1,1) < (1,0)$. By similar arguments as above, we have $\bar\Lambda_1 = \nu_1 - k = \bar\Lambda_2 + 1= \mu_2 + k + 1$, hence $\nu_1 = \mu_1 + 2k + 1$, and so $\beta_1 \leq \nu_1 - \mu_1 = 2k + 1$. On the other hand, since $(1,1) < (1,0)$ and $(\alpha,\beta) \in H(n,k)$, we have $\beta_1 - k > \alpha_1 + k$, so $\beta_1 > \alpha_1 + 2k \geq 2k$. Thus $\beta_1 = 2k+1$ and $\alpha_1 = 0$. 
Now since $\alpha = 0$, we have $A_{\lambda,\eps} = \alpha + [k,-\infty[_2 = (k,k-2,k-4,k-6,\dots)$.
Thus the consecutive terms of $A_{\lambda,\eps}$ decrease by $2$ each step. 
Since $(\alpha,\beta) \in H(n,k)$, it holds that $(B_{\lambda,\eps})_j = \beta_j + k + 2 - 2j$ is not equal to any term of $A_{\lambda,\eps}$, and so 
for each $i \in \N$, there exists at most one $j \in \N$ such that $\beta_j \neq 0$ and $(A_{\lambda,\eps})_i < (B_{\lambda,\eps})_j< (A_{\lambda,\eps})_{i+1}$, i.e. $(i,0) < (j,1) < (i+1,0)$. Thus we have $\nu_1 = \sum_j \beta_j$, $\mu_1 = 0$.
Recall that $\bar \Lambda_1 = \bar \Lambda_2 + 1$, i.e. $\nu_1 - k = \mu_1 + k + 1$, and so
$\sum_j \beta_j = \nu_1 = 2k + 1$. 
Thus we must have $\beta_2 = 0$, i.e. $(\alpha,\beta) = (\varnothing,(2k+1))$ and hence $(\nu,\mu) = (\varnothing,(2k+1))$.
The generalised Springer correspondence gives $\bar\lambda = (2k+1) \sqcup (2k+1, 2k-1, 2k-3, \dots, 1)$ and
$\bar\eps(1) = \bar\eps(2) = 1$ and $\bar\eps(i) = (-1)^{i}$ for $i \geq 3$.
We have $\bar\lambda_1 = \bar\lambda_2 \neq \bar\lambda_3$, as desired.
\end{proof}

By similar arguments, we find an analogous result for the $\Sp(2n)$. Note that the symbols appearing in the generalised Springer correspondence of $\Sp(2n)$ are of the form $(\alpha+[k,-\infty[_2,\beta+[-k-1,-\infty[_2)$ for $k \in \Z_{\geq0}$ and $(\alpha,\beta) \in \mathcal P_2(n - k(k+1)/2)$.
\begin{proposition}
Suppose $(\lambda,\eps) \in \PPsymp(2n)$ such that $\lambda$ only has even parts. Then for each $i \in \N$, we have $\mult_{\lambdamax}(i) \leq 2$, i.e. $\lambdamax$ parametrises a quasi-distinguished unipotent class of $\Sp(2n)$.
\end{proposition}

\begingroup\footnotesize
\linespread{0}\selectfont
\bibliographystyle{alpha}
\addcontentsline{toc}{section}{Bibliography}
\bibliography{main}
\endgroup

\end{document}